\documentclass[review,a4page]{elsarticle}
\usepackage[utf8]{inputenc}
\usepackage{graphicx}
\usepackage{srcltx}
\usepackage{eurosym}
\usepackage{mathtools}
\allowdisplaybreaks
\usepackage{enumerate}
\usepackage{amsmath}
\usepackage{amsfonts}
\usepackage{amssymb}
\usepackage{amsthm}
\usepackage{graphicx}
\usepackage{mathrsfs}
\usepackage{xcolor}
\usepackage{exscale}
\usepackage{latexsym}
\usepackage{esvect}
\usepackage[T1]{fontenc}

\usepackage{calc}
\usepackage[titletoc,toc,page]{appendix}
\usepackage[colorlinks,plainpages=true,pdfpagelabels,hypertexnames=true,colorlinks=true,pdfstartview=FitV,linkcolor=blue,citecolor=red,urlcolor=black]{hyperref}
\PassOptionsToPackage{unicode}{hyperref}
\PassOptionsToPackage{naturalnames}{hyperref}
\AtBeginDocument{%
\hypersetup{citecolor=red}}
\usepackage{enumerate}
\usepackage[shortlabels]{enumitem}
\usepackage{bookmark}
\usepackage{wasysym}
\usepackage{esint}
\usepackage{setspace}
\usepackage[ddmmyyyy]{datetime}
\numberwithin{equation}{section}
\everymath{\displaystyle}
\usepackage[margin=2.2cm]{geometry}

\usepackage[capitalize,nameinlink]{cleveref}
\crefname{section}{Section}{Sections}
\crefname{subsection}{Subsection}{Subsections}
\crefname{condition}{Condition}{Conditions}
\crefname{hypothesis}{Hypothesis}{Hypothesis}
\crefname{assumption}{Assumption}{Assumptions}
\crefname{lemma}{Lemma}{Lemmas}
\crefname{claim}{Claim}{Claims}
\crefname{remark}{Remark}{Remarks}

\crefformat{equation}{\textup{#2(#1)#3}}
\crefrangeformat{equation}{\textup{#3(#1)#4--#5(#2)#6}}
\crefmultiformat{equation}{\textup{#2(#1)#3}}{ and \textup{#2(#1)#3}}
{, \textup{#2(#1)#3}}{, and \textup{#2(#1)#3}}
\crefrangemultiformat{equation}{\textup{#3(#1)#4--#5(#2)#6}}%
{ and \textup{#3(#1)#4--#5(#2)#6}}{, \textup{#3(#1)#4--#5(#2)#6}}%
{, and \textup{#3(#1)#4--#5(#2)#6}}

\Crefformat{equation}{#2Equation~\textup{(#1)}#3}
\Crefrangeformat{equation}{Equations~\textup{#3(#1)#4--#5(#2)#6}}
\Crefmultiformat{equation}{Equations~\textup{#2(#1)#3}}{ and \textup{#2(#1)#3}}
{, \textup{#2(#1)#3}}{, and \textup{#2(#1)#3}}
\Crefrangemultiformat{equation}{Equations~\textup{#3(#1)#4--#5(#2)#6}}%
{ and \textup{#3(#1)#4--#5(#2)#6}}{, \textup{#3(#1)#4--#5(#2)#6}}%
{, and \textup{#3(#1)#4--#5(#2)#6}}

\crefdefaultlabelformat{#2\textup{#1}#3}
\newtheorem{theorem}{Theorem}[section]
\newtheorem{lemma}[theorem]{Lemma}

\newtheorem{corollary}[theorem]{Corollary}
\newtheorem{proposition}[theorem]{Proposition}

\newtheorem{definition}[theorem]{Definition}
\newtheorem{remark}[theorem]{Remark}        

\numberwithin{equation}{section}
\def\Yint#1{\mathchoice
{\YYint\displaystyle\textstyle{#1}}%
{\YYint\textstyle\scriptstyle{#1}}%
{\YYint\scriptstyle\scriptscriptstyle{#1}}%
{\YYint\scriptscriptstyle\scriptscriptstyle{#1}}%
\!\iint}
\def\YYint#1#2#3{{\setbox0=\hbox{$#1{#2#3}{\iint}$}
\vcenter{\hbox{$#2#3$}}\kern-.50\wd0}}
\def\longdash{-\mkern-9.5mu-} 
\def\tiltlongdash{\rotatebox[origin=c]{18}{$\longdash$}}
\def\fiint{\Yint\tiltlongdash}

\def\XXint#1#2#3{{\setbox0=\hbox{$#1{#2#3}{\int}$}
\vcenter{\hbox{$#2#3$}}\kern-.50\wd0}}

\makeatletter
\def\namedlabel#1#2{\begingroup
\def\@currentlabel{#2}%
\label{#1}\endgroup
}
\makeatother
\makeatletter
\newcommand{\rmh}[1]{\mathpalette{\raisem@th{#1}}}
\newcommand{\raisem@th}[3]{\hspace*{-1pt}\raisebox{#1}{$#2#3$}}
\makeatother

\newcommand{\descref}[2]{\hyperref[#1]{\textcolor{black}{(}\textcolor{blue}{\bf #2}\textcolor{black}{)}}}

\newcommand{\dref}[2]{\hyperref[#1]{\textcolor{black}{(}\textcolor{blue}{\bf #2}\textcolor{black}{)}}}

\makeatletter
\g@addto@macro\normalsize{%
\setlength\abovedisplayskip{3pt}
\setlength\belowdisplayskip{3pt}
\setlength\abovedisplayshortskip{1pt}
\setlength\belowdisplayshortskip{3pt}
}
\makeatletter
\def\ps@pprintTitle{%
\let\@oddhead\@empty
\let\@evenhead\@empty
\def\@oddfoot{}%
\let\@evenfoot\@oddfoot}
\makeatother
\newcounter{whitney}
\refstepcounter{whitney}

\newcounter{ineqcounter}
\refstepcounter{ineqcounter}

\begin{document}
\begin{frontmatter}
\title{On a mixed local-nonlocal evolution equation with singular nonlinearity}

\author{Kaushik Bal and Stuti Das}
\ead{kaushik@iitk.ac.in and stutid21@iitk.ac.in}

\address{Department of Mathematics and Statistics,\\ Indian Institute of Technology Kanpur, Uttar Pradesh, 208016, India}

\newcommand*{\avint}{\mathop{\, \rlap{--}\!\!\int}\nolimits}

\begin{abstract}
We will prove several existence and regularity results for the mixed local-nonlocal parabolic equation of the form
\begin{eqnarray}
\begin{split}
u_t-\Delta u+(-\Delta)^s u&=\frac{f(x,t)}{u^{\gamma(x,t)}}  \text { in } \Omega_T:=\Omega \times(0, T), \\ u&=0  \text { in }(\mathbb{R}^n \backslash \Omega) \times(0, T), \\ u(x, 0)&=u_0(x)  \text { in } \Omega ;
\end{split}
\end{eqnarray}
where 
\begin{equation*}
(-\Delta )^s u= c_{n,s}\operatorname{P.V.}\int_{\mathbb{R}^n}\frac{u(x,t)-u(y,t)}{|x-y|^{n+2s}} d y. 
\end{equation*}
Under the assumptions that $\gamma$ is a positive continuous function on $\overline{\Omega}_T$ and $\Omega$ is a bounded domain 
with Lipschitz boundary
in $\mathbb{R}^{n}$, $n> 2$, $s\in(0,1)$, $0<T<+\infty$, $f\geq 0$, $u_0\geq 0$, $f$ and $u_0$ belongs to suitable Lebesgue spaces. Here $c_{n,s}$ is a suitable normalization constant, and $\operatorname{P.V.}$ stands for Cauchy Principal Value.
\end{abstract}



\end{frontmatter}

\begin{singlespace}
\tableofcontents
\end{singlespace}

\section{Introduction}
 In this article, we study the evolution of a mixed local-nonlocal operator under the effect of a singular nonlinearity given by:

 \begin{equation}{\label{prob}}
\begin{split}
u_t-\Delta u+(-\Delta)^su&=\frac{f(x,t)}{u^{\gamma(x,t)}}
 \; \mbox{ in } \Omega_T:=\Omega \times (0,T)
 ,\\
u&=0 \;  \mbox{ in } (\mathbb{R}^n\backslash \Omega) \times (0,T),\\
u(x,0)&=u_0(x)   \mbox{ in } \Omega;
\end{split}
\end{equation}
where 
\begin{equation*}
(-\Delta )^s u:= c_{n,s}\operatorname{P.V.}\int_{\mathbb{R}^n}\frac{u(x,t)-u(y,t)}{|x-y|^{n+2s}} d y, 
\end{equation*}
$\gamma$ is a positive continuous function on $\overline{\Omega}_T$ with $\Omega$ being a bounded domain 
with Lipschitz boundary
in $\mathbb{R}^{n}$, $n> 2$, $s\in(0,1)$ and $0<T<+\infty$. Assuming that $f\geq 0$ and $u_0\geq 0$ with variable summability, in this paper we show several existence and regularity results. To this aim, we start by reviewing the literature concerning our problem.

Singular elliptic problems have been extensively studied in the literature for the past few decades starting with the now classical work of Crandall-Rabinowitz-Tartar \cite{CrRaTa}, who showed that the stationary state of \eqref{prob}, under Dirichlet boundary conditions given by
\begin{eqnarray}{\label{pb}}
\begin{split}
    -\Delta u&=\frac{f}{u^{\gamma(x)}}  \text { in } \Omega, \\
u&>0\;  \text { in } \Omega, \\
u&=0\;  \text { in } \partial \Omega.
\end{split}
\end{eqnarray}
admits a unique solution $u\in C^2(\Omega)\cap C(\bar{\Omega})$ for any $\gamma>0$ constant along with the fact that the solution must behave like a distance function near the boundary provided $f$ is H\"older Continuous. Interestingly enough Lazer-Mckenna \cite{LaMc} showed that the unique solution obtained by \cite{CrRaTa} is indeed in $H_0^1(\Omega)$ iff $0<\gamma<3$. Boccardo-Orsina \cite{orsina} in a beautiful paper showed that the followings regarding solutions of \eqref{pb} 
\begin{equation*}
    \begin{cases}u \in W_0^{1, \frac{n r(1+\gamma)}{n-r(1-\gamma)}}(\Omega) & \text { if } 0<\gamma<1 \text { and } f \in L^r(\Omega) \text { with } r \in\left[1,\left(2^* /(1-\gamma)\right)^{\prime}\right), \\ u \in H_0^1(\Omega) & \text { if } 0<\gamma<1 \text { and } f \in L^r(\Omega) \text { with } r=\left(2^* /(1-\gamma)\right)^{\prime}, \\ u \in H_0^1(\Omega) & \text { if } \gamma=1 \text { and } f \in L^1(\Omega), \\ u^{\frac{1+\gamma}{2}} \in H_0^1(\Omega) & \text { if } \gamma>1 \text { and } f \in L^1(\Omega),\end{cases}
\end{equation*}
hold, which was extended for variable $\gamma$, introducing certain conditions on its behaviour near the boundary in \cite{variablegamma}. The nonlocal variant given by
\begin{eqnarray*}
\begin{split}
    (-\Delta)_p^s u&=\frac{f}{u^{\gamma(x)}}  \text { in } \Omega, \\
u&>0  \text { in } \Omega, \\
u&=0  \text { in } \partial \Omega.
\end{split}
\end{eqnarray*}
was studied in \cite{fracsingular} for $p=2$ and $\gamma(x)=\gamma \in \mathbb{R}_*^{+}$, the authors proved the existence and uniqueness of positive solutions, according to the range of $\gamma$ and summability of $f$. For the quasilinear case, we refer \cite{fracvari1} for constant $\gamma>0$ and for variable singular exponent, the existence results have been obtained in \cite{fracvarisin}.
As for the Mixed local-nonlocal elliptic problem given by
\begin{eqnarray*}
\begin{split}
   -\Delta_p u+ (-\Delta)_p^s u&=\frac{f}{u^{\gamma(x)}}  \text { in } \Omega, \\
u&>0  \text { in } \Omega, \\
u&=0  \text { in } \partial \Omega;
\end{split}
\end{eqnarray*}
 Arora \cite{arora} for $p=2$ and $\gamma(x)=\gamma\in \mathbb{R}_*^{+}$, obtained the existence, uniqueness and regularity properties of the weak solutions by deriving uniform a priori estimates and using the approximation technique. They also obtained some existence and nonexistence results when $f$ behaves like a distance function. For the case $p>1$, the constant exponent $\gamma$ case has been considered in \cite{garain}, and the variable exponent can be found in \cite{Biroud}. If one considers the parabolic counterpart i.e, the equation given by:
\begin{eqnarray*}
\begin{split}
u_t-\Delta_p u&=\frac{f(x,t)}{u^{\gamma}}  \text { in } \Omega_T, \\ u&=0  \text { in }(\mathbb{R}^n \backslash \Omega) \times(0, T), \\ u(x, 0)&=u_0(x)  \text { in } \Omega .
\end{split}
\end{eqnarray*}
For such an equation with $p \geq 2,\;0 \leq f \in L^m\left(\Omega_T\right)$ with $m \geq 1$ and assuming that
$\forall \omega \subset \subset \Omega,\exists\; d_\omega>0$, such that $u_0 \geq d_\omega$,
the authors \cite{parabolic1} proved the existence of a solution $u$ of the above problem such that
\begin{eqnarray*}
    \begin{cases}u \in L^{q_0}(0, T ; W_0^{1, q_0}(\Omega))&\text { if } 0<\gamma<1 \text { and } f \in L^r(\Omega) \text { with } r \in\left[1,\frac{p(n+2)}{p(n+2)-n(1-\gamma)}\right)
\\&\text { and } q_0=\frac{m[n(p+\gamma-1)+p(\gamma+1)]}{n+2-m(1-\gamma)},
\\ u \in L^p(0, T ; W^{1, p}_0(\Omega))& \text { if } 0<\gamma\leq1 \text { and } f \in L^{m_0}\left(\Omega_T\right) \text{ with } m_0=\frac{p(n+2)}{p(n+2)-n(1-\gamma)},\\ u \in L^p(0, T ; W_{\operatorname{loc}}^{1, p}(\Omega)), & \text { if } \gamma>1 \text { and } f \in L^1(\Omega_T).
  \end{cases}
\end{eqnarray*}
The Nonlocal case $(s\in(0,1))$ for the parabolic problem was handled by \cite{abdellaoui1} for $\gamma>0$ constant to show existence and uniqueness results along similar lines. If one restricts the range of $\gamma$ then various existence, uniqueness and regularity results can be found in Bal-Badra-Giacomoni \cite{BaBaGi, BaBaGi1,BaBaGi2} and Giacomoni-Bougherera \cite{BoGi}. We would also like to mention that the regularity theory of mixed local and non-local operators plays a major role in our problem and we cite the following papers \cite{paper1,mingi,localb,garainholder,weakharnack,garain2023higher,par2,par3} and the references therein. \smallskip\\
As for the boundedness of our solutions, we refer Aronson-Serrin \cite{AS}, where the summability requirement of initial data for boundedness was introduced by Aronson and Serrin, for the case of second-order differential equations without singularity. Outside of the Aronson-Serrin domain, the optimal summability of solutions for the local case without singularity was obtained in Boccardo-Porzio-Primo \cite{Summability}. These results for the nonlocal case have been obtained in Peral \cite{Peral}, this too for the nonsingular case. For the mixed local-nonlocal operator with singularity, we will be able to get similar types of results here depending on the choice of $\gamma$. We will use suitable approximating problems to get the existence and other summability properties of weak solutions.
\subsection*{\textbf{Organization of the article}}
In the next section, we will describe some basic notations and fix some preliminary function spaces to define our solutions, followed by embedding results and other properties regarding those spaces. \smallskip\\Then we will introduce the notion of weak solution for the case $\gamma=0$ and show its existence, uniqueness, positivity and other properties. After that, we write about the existence of weak solutions for approximating problems and give definitions of weak solutions for the singular cases, both for constant and variable exponents. We end this section by stating our main theorems regarding the existence and summability of weak solutions and appropriate comments.
\smallskip\\The next section contains the proofs of our main results, and we end with another section that gives the asymptotic behaviour of the solutions in a suitable sense.
\section{Preliminaries} 
\subsection{\textbf{Notations}} We gather here all the standard notations that will be used throughout the paper.
\smallskip\\$\bullet$ We will take $n$ to be the space dimension and denote by $z=(x, t)$ to be a point in $\mathbb{R}^n \times(0, T)$, where $(0,T)\subset \mathbb{R}$ for some $0<T<\infty$.\smallskip\\
$\bullet$ Let $\Omega$ be an open bounded domain in $\mathbb{R}^n$ with boundary $\partial \Omega$ and for $0<T <\infty$, let $\Omega_T:=\Omega \times(0, T)$.\smallskip\\
$\bullet$ We denote the parabolic boundary $\Gamma_T$ by $\Gamma_T=(\Omega\times\{t=0\})\cup(\partial\Omega\times(0,T))$.\smallskip\\ 
$\bullet$ We define the set $(\Omega_T)_\delta=\{(x,t)\in\Omega_T:\operatorname{dist}((x,t),\Gamma_T)<\delta\}$ for $
\delta>0$ fixed.
\smallskip\\$\bullet$ We shall alternately use $\partial_t g$ or $\frac{\partial g}{\partial t}\text{ or } g_t $ to denote the time derivative (partial) of a function $g$.\smallskip\\
$\bullet$ For $r>1$, the H\"older conjugate exponent of $r$ will be denoted by $r^\prime=\frac{r}{r-1}$.\smallskip\\
$\bullet$ The Lebesgue measure of a measurable subset $\mathrm{S}\subset \mathbb{R}^n$ will be denoted by $|\mathrm{S}|$.\smallskip\\
$\bullet$ For any open subset $\Omega$ of $\mathbb{R}^n$, $K\subset\subset \Omega $ will imply $K$ is compactly contained in $\Omega.$\smallskip\\
$\bullet$ $\int$ will denote integration concerning either space or time only, and integration on $\Omega \times \Omega$ or $\mathbb{R}^n \times \mathbb{R}^n$ will be denoted by a double integral $\iint$.\smallskip\\
$\bullet$ We will use $\iiint$ to denote integral over $\mathbb{R}^n \times \mathbb{R}^n \times(0, T)$.
\smallskip\\
$\bullet$ Average integral will be denoted by $\fint$.\\
$\bullet$ The notation $a \lesssim b$ will be used for $a \leq C b$, where $C$ is a universal constant which only depends on the dimension $n$ and sometimes on $s$ too. $C$ may vary from line to line or even in the same line.\smallskip\\
$\bullet$ $\langle.,.\rangle$ will denote the usual inner product in some associated Hilbert space.\smallskip\\
$\bullet$ For any function $h$, we denote the positive and negative parts of it by $h_+=\operatorname{max}\{h,0\}$ and $h_-=\operatorname{max}\{-h,0\}$ respectively.\smallskip\\
$\bullet$ For $k\in \mathbb{N}$, we denote $T_k(\sigma)=\max \{-k, \min \{k, \sigma\}\}$, for $\sigma \in \mathbb{R}$.
\subsection{\textbf{Function Spaces}}
In this section, we present 
 definitions and properties of some function spaces that will be useful for our work. We recall that for $E \subset \mathbb{R}^n$, the Lebesgue space
$L^p(E), 1 \leq p<\infty$, is defined to be the space of $p$-integrable functions $u: E \rightarrow \mathbb{R}$ with the finite norm
\begin{equation*}
\|u\|_{L^p(E)}=\left(\int_E|u(x)|^p d x\right)^{1 / p} .
\end{equation*}
By $L_{\operatorname{loc }}^p(E)$ we denote the space of locally $p$-integrable functions, which means, $u \in L_{\operatorname{loc }}^p(E)$ if and only if $u \in L^p(F)$ for every $F \subset\subset E$. In the case $0<p<1$, we denote by $L^p(E)$ a set of measurable functions such that $\int_E|u(x)|^p d x<\infty$.
\begin{definition}
    The Sobolev space $W^{1, p}(\Omega)$, for $1 \leq p<\infty$, is defined as the Banach space of locally integrable weakly differentiable functions $u: \Omega \rightarrow \mathbb{R}$ equipped with the following norm
\begin{equation*}
\|u\|_{W^{1, p}(\Omega)}=\|u\|_{L^p(\Omega)}+\|\nabla u\|_{L^p(\Omega)} .
\end{equation*}
\end{definition}
The space $W_0^{1, p}(\Omega)$ is defined as the closure of the space $\mathcal{C}_0^{\infty}(\Omega)$, in the norm of the Sobolev space $W^{1, p}(\Omega)$, where $\mathcal{C}^\infty_0(\Omega)$ is the set of all smooth functions whose supports are compactly contained in $\Omega$.
\begin{definition}
    Let $0<s<1$ and $\Omega$ be a open connected subset of $\mathbb{R}^n$. The fractional Sobolev space $W^{s, q}(\Omega)$ for any $1\leq q<+\infty$ is defined by
\begin{equation*}
    W^{s, q}(\Omega)=\left\{u \in L^q(\Omega): \frac{|u(x)-u(y)|}{|x-y|^{\frac{n}{q}+s}} \in L^q(\Omega\times\Omega)\right\},
\end{equation*}
and it is endowed with the norm
\begin{equation}{\label{norm}}
\|u\|_{W^{s, q}(\Omega)}=\left(\int_{\Omega}|u(x)|^q d x+\int_{\Omega} \int_{\Omega} \frac{|u(x)-u(y)|^q}{|x-y|^{n+q s}}d x d y\right)^{1/q}.
\end{equation}
\end{definition}
It can be treated as an intermediate space between $W^{1,q}(\Omega)$ and $L^q(\Omega)$. For $0<s\leq s^{\prime}<1$, $W^{s^{\prime},q}(\Omega)$ is continuously embedded in $W^{s,q}(\Omega)$, see [\citealp{frac}, Proposition 2.1]. The fractional Sobolev space with zero boundary values is defined by
\begin{equation*}
W_0^{s, q}(\Omega)=\left\{u \in W^{s, q}(\mathbb{R}^n): u=0 \text { in } \mathbb{R}^n \backslash \Omega\right\}.
\end{equation*}
However $W_0^{s, q}(\Omega)$ can be treated as the closure of $\mathcal{C}^\infty_0(\Omega)$ in $W^{s,q}(\Omega)$ with respect to the fractional Sobolev norm defined in \cref{norm}. Both $W^{s, q}(\Omega)$ and $W_0^{s, q}(\Omega)$ are reflexive Banach spaces, for $q>1$, for details we refer to the readers [\citealp{frac}, Section 2].\smallskip\\
The following result asserts that the classical Sobolev space is continuously embedded in the fractional Sobolev space; see [\citealp{frac}, Proposition 2.2]. The idea applies an extension property of $\Omega$ so that we can extend functions from $W^{1,q}(\Omega)$ to $W^{1,q}(\mathbb{R}^n)$ and that the extension operator is bounded.
\begin{lemma}{\label{embedding}}
    Let $\Omega$ be a smooth bounded domain in $\mathbb{R}^n$ and $0<s<1$. There exists a positive constant $C=C(\Omega, n, s)$ such that
\begin{equation*}
\|u\|_{W^{s, q}(\Omega)} \leq C\|u\|_{W^{1,q}(\Omega)},
\end{equation*}
for every $u \in W^{1,q}(\Omega)$.
\end{lemma}
For the fractional Sobolev spaces with zero boundary value, the next embedding result follows from [\citealp{frac2}, Lemma 2.1]. The fundamental difference of it compared to \cref{embedding} is that the result holds for any bounded domain (without any condition of smoothness of the boundary), since for the Sobolev spaces with zero boundary value, we always have a zero extension to the complement.
\begin{lemma}{\label{embedding2}} Let $\Omega$ be a bounded domain in $\mathbb{R}^n$ and $0<s<1$. There exists a positive constant $C=C(n, s, \Omega)$ such that
\begin{equation*}
\int_{\mathbb{R}^n} \int_{\mathbb{R}^n} \frac{|u(x)-u(y)|^q}{|x-y|^{n+q s}} d x d y \leq C \int_{\Omega}|\nabla u|^qd x
\end{equation*}
for every $u \in W_0^{1,q}(\Omega)$. Here, we consider the zero extension of $u$ to the complement of $\Omega$.
\end{lemma}
We now proceed with the basic Poincar\'{e} inequality, which can be found in [\citealp{LCE}, Chapter 5, Section 5.8.1].
\begin{lemma}{\label{p}}
  Let $\Omega\subset \mathbb{R}^n$ be a bounded domain with $\mathcal{C}^1$ boundary and $q\geq 1$. Then there exist a positive constant $C>0$ depending only on $n$ and $ \Omega$, such that \begin{equation*} 
  \int_\Omega |u|^q d x\leq C\int_\Omega |\nabla u|^q d x, \qquad\forall u\in W^{1,q}_0(\Omega).
  \end{equation*}
  Specifically if we take $\Omega=B_r$, then we will get for all $u\in W^{1,q}(B_r)$,
  \begin{equation*}
  \fint_{B_{r}}\left|u-(u)_{B_r}\right|^q d x \leq c r^q \fint_{B_{r}} |\nabla u|^qd x,
  \end{equation*}
  where $c$ is a constant depending only on $n$, and $(u)_{B_r}$ denotes the average of $u$ in $B_r$, and $B_r$ denotes a ball of radius $r$ centered at $x_0\in \mathbb{R}^n$. Here, $\fint$ denotes the average integration.
 \end{lemma}
Using \cref{embedding2}, and the above Poincar\'e inequality, we observe that the following norm on the space $W^{1,q}_0(\Omega)$ defined by 
 \begin{equation*}
\|u\|_{W^{1,q}_0(\Omega)}=\left(\int_\Omega |\nabla u|^q d x +\int_{\mathbb{R}^n} \int_{\mathbb{R}^n} \frac{|u(x)-u(y)|^q}{|x-y|^{n+q s}} d x d y \right)^{\frac{1}{q}},
\end{equation*}
is equivalent to the norm
 \begin{equation*}
\|u\|_{W^{1,q}_0(\Omega)}=\left(\int_\Omega |\nabla u|^q d x  \right)^{\frac{1}{q}} .     
 \end{equation*}
The following is a version of fractional Poincar\'{e}.
\begin{lemma}{\label{fracpoin}}
Let $\Omega\subset \mathbb{R}^n$ be a bounded domain with $\mathcal{C}^1$ boundary and let $s \in(0,1)$ and $q\geq 1$. If $u \in W^{s,q}_0(\Omega)$, then
\begin{equation*}
\int_\Omega |u|^q d x \leq c\int_{\Omega} \int_{\Omega} \frac{|u(x)-u(y)|^q}{|x-y|^{n+qs}} d x d y ,
\end{equation*}
holds with $c \equiv c(n, s,\Omega)$.
\end{lemma}
In view of \cref{fracpoin}, we observe that the Banach space $W_0^{s, q}(\Omega)$ can be endowed with the norm
\begin{equation*}
\|u\|_{W_0^{s, q}(\Omega)}=\left(\int_{\Omega} \int_{\Omega} \frac{|u(x)-u(y)|^q}{|x-y|^{n+q s}} d x d y\right)^{\frac{1}{q}},
\end{equation*}
which is equivalent to that of $\|u\|_{W^{s, q}(\Omega)}$. For $q=2$, the space $W^{s,2}(\Omega)$ enjoys certain special properties and we denote $W^{s, 2}(\Omega)=H^s(\Omega)$ and $W_0^{s, 2}(\Omega)=H_0^s(\Omega)$. Endowed with the inner product
\begin{equation*}
\langle u, v\rangle_{H_0^s(\Omega)}=\int_{\Omega} \int_{\Omega} \frac{(u(x)-u(y))(v(x)-v(y))}{|x-y|^{n+2 s}} d x d y ,
\end{equation*}
we note that $(H_0^s(\Omega),\|\cdot\|_{H_0^s(\Omega)})$ is a Hilbert space. Similar thing holds for the space $W^{1,2}_0(\Omega)=H^1_0(\Omega).$
\begin{definition}
The space $X_0^s(\Omega)$ is defined as
\begin{equation*}
X_0^s(\Omega)=\left\{f \in H^s(\mathbb{R}^n) \text { s.t. } f=0 \text { a.e. in } \mathcal{C} \Omega\right\},
\end{equation*}
where $\mathcal{C} \Omega=\mathbb{R}^n \backslash \Omega$, and is endowed with the norm
\begin{equation*}
\|u\|_{X_0^s(\Omega)}=\left(\int_Q \frac{|u(x)-u(y)|^2}{|x-y|^{n+2 s}} d x d y\right)^{\frac{1}{2}},
\end{equation*}
where $Q:=\mathbb{R}^{2 n} \backslash(\mathcal{C} \Omega \times \mathcal{C} \Omega)$.    
\end{definition}
Moreover $W^{-1, q^{\prime}}(\Omega)$, $W^{-s, q^{\prime}}(\Omega)$ and $X^{-s}(\Omega)$ are defined to be the dual spaces of $W_0^{1, q}(\Omega)$, $W_0^{s, q}(\Omega)$ and $X_0^s(\Omega)$ respectively, where $q^{\prime}:=\frac{q}{q-1}$. Now, we define the local spaces as
\begin{equation*}
     W_{\operatorname{loc }}^{1, q}(\Omega)=\left\{u: \Omega \rightarrow \mathbb{R}: u \in L^q(K), \int_K |\nabla u|^q d x<\infty, \text { for every } K \subset \subset \Omega\right\} ,
\end{equation*}
and 
\begin{equation*}
     W_{\operatorname{loc }}^{s, q}(\Omega)=\left\{u: \Omega \rightarrow \mathbb{R}: u \in L^q(K), \int_K \int_K \frac{|u(x)-u(y)|^q}{|x-y|^{n+q s}} d x d y<\infty, \text { for every } K \subset \subset \Omega\right\} .
\end{equation*}
Now for $n>2$, we define the critical Sobolev exponent as $2^*=\frac{2n}{n-2}$, then we get the following embedding result for any open subset $\Omega$ of $\mathbb{R}^n$, see for details [\citealp{LCE}, Chapter 5],
\begin{theorem}{\label{Sobolev embedding}} Let $n>2 $. Then, there exists a constant $C$ depending only on $n$ and $\Omega$, such that for all $u \in \mathcal{C}_0^{\infty}(\Omega)$
\begin{equation*}
    \|u\|_{L^{2 ^*}(\Omega)}^2 \leq C \int_{\Omega} |\nabla u|^2 d x.
\end{equation*}
\end{theorem}
Similarly, for $n>2 s$, we define the fractional Sobolev critical exponent as $2_s^*=\frac{2 n}{n-2 s}$. The following result is a fractional version of the Sobolev inequality(\cref{Sobolev embedding}) which also implies a continuous embedding of $H_0^s(\Omega)$ in the critical Lebesgue space $L^{2_s^*}(\Omega)$. One can see the proof in \cite{frac}.
 \begin{theorem}{\label{Fractional Sobolev embedding}} Let $0<s<1$ be such that $n>2 s$. Then, there exists a constant $S(n, s)$ depending only on $n$ and $s$, such that for all $u \in \mathcal{C}_0^{\infty}(\Omega
 )$
\begin{equation*}
\|u\|_{L^{2_s^*}(\Omega
)}^2 \leq S(n, s) \int_{\Omega}\int_{\Omega
} \frac{|u(x)-u(y)|^2}{|x-y|^{n+2 s}} d x d y .    
\end{equation*}
\end{theorem}
We now recall the Gagliardo-Nirenberg interpolation inequality that will be useful for proving the boundedness of weak solutions.
\begin{theorem}{\label{gagliardo}}
    (Gagliardo-Nirenberg) Let $1 \leq q <+\infty$ be a positive real number. Let $j$ and $m$ be non-negative integers such that $j<m$. Furthermore, let $1 \leq r \leq+\infty$ be a positive extended real quantity, $p \geq 1$ be real and $\theta \in[0,1]$ such that the relations
\begin{equation*}
\frac{1}{p}=\frac{j}{n}+\theta\left(\frac{1}{r}-\frac{m}{n}\right)+\frac{1-\theta}{q}, \quad \frac{j}{m} \leq \theta \leq 1
\end{equation*}
hold. Then,
\begin{equation*}
\left\|D^j u\right\|_{L^p\left(\mathbb{R}^n\right)} \leq C\left\|D^m u\right\|_{L^r\left(\mathbb{R}^n\right)}^\theta\|u\|_{L^q\left(\mathbb{R}^n\right)}^{1-\theta}
\end{equation*}
for any $u \in L^q(\mathbb{R}^n)$ such that $D^m u \in L^r(\mathbb{R}^n)$. Here, the constant $C>0$ depends on the parameters $j, m, n, q, r, \theta$, but not on $u$.
\end{theorem}The article will extensively use the embedding results and corresponding inequalities. Now, we need to deal with spaces involving time for the parabolic equations, so we introduce them here. As in the classical case, we define the corresponding Bochner spaces as the following
\begin{equation*}
    \begin{array}{c}
L^q(0, T ; W_0^{1, q}(\Omega)) =\left\{u \in L^q(\Omega \times(0, T)),\|u\|_{L^q(0, T ; W_0^{1, q}(\Omega))}<\infty\right\}, \smallskip\\
L^q(0, T ; W_0^{s, q}(\Omega)) =\left\{u \in L^q(\Omega \times(0, T)),\|u\|_{L^q(0, T ; W_0^{s, q}(\Omega))}<\infty\right\}, \smallskip\\
L^2(0, T ; X_0^s(\Omega)) =\left\{u \in L^2(\mathbb{R}^n \times(0, T)),\|u\|_{L^2(0, T ; X^s_0(\Omega))}<\infty\right\},
    \end{array}
\end{equation*}
where
\begin{equation*}
\begin{array}{c}
\|u\|_{L^q(0, T ; W_0^{1, q}(\Omega))}  =\left(\int_0^T \int_{\Omega} |\nabla u|^q d x d t\right)^{\frac{1}{q}}, \smallskip\\
\|u\|_{L^q(0, T ; W_0^{s, q}(\Omega))}  =\left(\int_0^T \int_{\Omega} \int_{\Omega} \frac{|u(x, t)-u(y, t)|^q}{|x-y|^{n+q s}} d x d y d t\right)^{\frac{1}{q}}, \smallskip\\
\|u\|_{L^2(0, T ; X_0^s(\Omega))}  =\left(\int_0^T \int_Q \frac{|u(x, t)-u(y, t)|^2}{|x-y|^{n+2 s}} d x d y d t\right)^{\frac{1}{2}},
\end{array}
\end{equation*}
with their dual spaces $L^{q^{\prime}}(0, T ; W^{-1, q^{\prime}}(\Omega))$, $L^{q^{\prime}}(0, T ; W^{-s, q^{\prime}}(\Omega))$ and $L^2(0, T ; X^{-s}(\Omega))$ respectively. Again, the local spaces are defined as
\begin{equation*}
    L^2(0, T ; H_{l o c}^1(\Omega))=\left\{u \in L^2(K \times(0, T)) :\int_0^T \int_K |\nabla u|^q d x d t<\infty,
\mbox { for every }K \subset \subset \Omega\right\},
\end{equation*}and\begin{equation*}
    L^2(0, T ; H_{l o c}^s(\Omega))=\left\{u \in L^2(K \times(0, T)) :\int_0^T \int_K \int_K \frac{|u(x, t)-u(y, t)|^2}{|x-y|^{n+2 s}} d x d y d t<\infty,
\mbox { for every }K \subset \subset \Omega\right\}.
\end{equation*}
We now recall the following algebraic inequality that can be found in [\citealp{abdellaoui1}, Lemma 2.22].
\begin{lemma}{\label{algebraic}}
\begin{enumerate}[label=\roman*)]
    \item Let $\alpha>0$. For every $x, y \geq 0$ one has
\begin{equation*}
    (x-y)(x^\alpha-y^\alpha) \geq \frac{4 \alpha}{(\alpha+1)^2}\left(x^{\frac{\alpha+1}{2}}-y^{\frac{\alpha+1}{2}}\right)^2.
\end{equation*}
\item Let $0<\alpha \leq 1$. For every $x, y \geq 0$ with $x \neq y$ one has
\begin{equation*}
    \frac{x-y}{x^\alpha-y^\alpha} \leq \frac{1}{\alpha}(x^{1-\alpha}+y^{1-\alpha}).
\end{equation*}
\item Let $\alpha \geq 1$. Then there exists a constant $C_\alpha$ depending only on $\alpha$ such that
\begin{equation*}
|x+y|^{\alpha-1}|x-y| \leq C_\alpha\left|x^\alpha-y^\alpha\right| .   
\end{equation*}    
\end{enumerate}
\end{lemma}
\subsection{\textbf{Weak Solutions}}
In this subsection, along with the next subsection, we will introduce notions of very weak solutions to our problem and state the main results that we are going to prove. We begin with the definitions of weak solutions for the nonsingular case. We first take $f$ and $u_0$ to be in $L^2$ spaces and then relax the condition. We also state some important properties of the weak solutions that we need to use in the rest of the article. Further, we will introduce suitable approximating problems and properties of their solutions.
\begin{definition}{\label{solution1}} Assume $(f, u_0) \in L^2(\Omega_T) \times L^2(\Omega)$, then we say that $u$ is an energy solution to problem
\begin{equation}{\label{problem1}}
\begin{array}{lcr}
\begin{cases}
u_t-\Delta u+(-\Delta)^su={f(x,t)} 
 & \mbox{ in } \Omega_T
 ,\\
u=0  & \mbox{ in } (\mathbb{R}^n\backslash \Omega) \times (0,T),\\
u(x,0)=u_0(x) &  \mbox{ in } \Omega;
\end{cases}
\end{array}
\end{equation}
if $u \in 
L^2(0, T;H_0^1(\Omega)) \cap \mathcal{C}([0, T], L^2(\Omega)), u_t \in 
L^2(0, T ; H^{-1}(\Omega))$, and for all $\phi \in 
L^2(0, T;H_0^1(\Omega))$ we have
\begin{equation*}
    \int_0^T\left\langle u_t, \phi\right\rangle d t+\int_0^T\int_{\Omega}\nabla u\cdot \nabla \phi d x d t+\frac{1}{2} \int_0^T \int_{\mathbb{R}^n}\int_{\mathbb{R}^n} \frac{(u(x, t)-u(y, t))(\phi(x, t)-\phi(y, t))}{|x-y|^{n+2 s}} d x d y d t=\int_0^T\int_{\Omega}f\phi d x  d t
\end{equation*}
and $u(\cdot, t) \rightarrow u_0$ strongly in $L^2(\Omega)$, as $t \rightarrow 0$.
\end{definition}
We denote
\begin{equation*}
    \mathcal{E}(u(x,t), \phi(x,t)):=\frac{1}{2} \int_{\mathbb{R}^n} \int_{\mathbb{R}^n}(u(x, t)-u(y, t))(\phi(x, t)-\phi(y, t))\times K(x, y, t) d x d y,
\end{equation*}
where $K(x,y,t)=\frac{1}{{{\left|x-y\right|}^{n+2s}}}.$ Following the way for fractional Laplacian in \cite{Peral}, we give the proof of existence for mixed local-nonlocal case for the sake of completeness.
\begin{theorem}{\label{exist1}}
There exists a solution to problem \cref{problem1} in the sense of \cref{solution1}. Moreover, if $f$ is also a nonnegative function and $u_0\geq 0$, then the solution is also nonnegative.
\end{theorem}
\begin{proof}
Let us denote $\mathcal{C}_*^{\infty}(\Omega \times[0, T])$ as the $\mathcal{C}^{\infty}(\Omega \times[0, T])$ functions that vanish in $(\mathbb{R}^n \backslash\Omega) \times[0, T]$ and in $\Omega \times\{T\}$. Choosing $\phi \in \mathcal{C}_*^{\infty}(\Omega \times[0, T])$, for $ u \in 
    L^2(0, T ; H_0^1(\Omega))$, we define the operator
\begin{equation*}
    L_\phi(u):=\int_0^T \int_{\Omega}-u \phi_t d x d t+\int_0^T\int_{\Omega}\nabla u\cdot \nabla \phi d x d t+\int_0^T \mathcal{E}(u(x, t), \phi(x, t)) d t .
\end{equation*}
We now observe that $u$ is an energy solution to \cref{problem1} with $f \in 
L^2(\Omega_T)\subset L^2(0, T ; H^{-1}(\Omega))$ if and only if
\begin{equation*}
    L_\phi(u)=\int_0^T\langle f, \phi\rangle d t+\int_{\Omega} u(x, 0) \phi(x, 0) d x ,
\end{equation*}
where $\langle f, \phi \rangle$ denote the usual inner product of $f$ and $\phi$ in $L^2(\Omega)$ or the pairing of $f$ and $\phi$ between $H^{-1}(\Omega)$ and $H^{1}_0(\Omega)$.\\
We also define the following inner product, for $\phi,\varphi\in \mathcal{C}_*^{\infty}(\Omega \times[0, T])$
\begin{equation}{\label{innerprod}}
\langle\varphi, \phi\rangle_*=\frac{1}{2}\langle\varphi(x, 0), \phi(x, 0)\rangle_{L^2(\Omega)}+\int_0^T\int_{\Omega}\nabla \phi\cdot \nabla \varphi~ d x d t+\int_0^T \mathcal{E}(\varphi(x, t), \phi(x, t)) d t,   
\end{equation}
and denote by $H^*(\Omega \times[0, T])$ the Hilbert space built as the completion of $\mathcal{C}_*^{\infty}(\Omega \times[0, T])$ with the norm $\|\phi\|_*$ induced by the inner product \cref{innerprod}.\\
Now  $L_\phi$ is clearly a linear functional in $
H^*(\Omega \times[0, T])
$ and for any $\varphi \in 
L^2(0, T ; H_0^1(\Omega))$, by Hölder and Sobolev inequalities, we have
\begin{equation*}
\left|L_\phi(\varphi)\right| \leq c_\phi\left(\|\varphi\|_{ L^2(0,T;L^2(\Omega))}+\|\varphi\|_{L^2(0, T, H_0^1(\Omega))}+\|\varphi\|_{L^2(0, T, H_0^s(\Omega))}\right) \leq \tilde{c}_\phi\|\varphi\|_*.
\end{equation*}
Therefore, $L_\phi$ is a bounded linear functional in $H^*(\Omega \times[0, T])$, and hence by the Fréchet-Riesz Theorem, there exists $\mathcal{T} \phi \in H^*(\Omega \times[0, T])$ such that
\begin{equation*}
L_\phi(\varphi)=\langle\varphi, \mathcal{T} \phi\rangle_* \quad \text { for all } \varphi \in H^*(\Omega \times[0, T]) .
\end{equation*}
It is trivial to show that $\mathcal{T}$ is a linear operator in $H^*(\Omega \times[0, T])$. Moreover
\begin{equation*}
    L_\phi(\phi)=\frac{1}{2} \int_{\Omega} \phi^2(x, 0) d x+\int_0^T\int_{\Omega}|\nabla \phi|^2 d x d t+\int_0^T \mathcal{E}(\phi(x, t), \phi(x, t)) d t=\|\phi\|_*^2,
\end{equation*}
and consequently, $\langle\phi, \mathcal{T} \phi\rangle_*=\|\phi\|_*^2$. Then we get by the Cauchy-Schwartz inequality,
\begin{equation*}
    \|\phi\|_*^2 \leq\|\phi\|_*\|\mathcal{T} \phi\|_*, \quad \text { i.e., } \quad\|\phi\|_* \leq\|\mathcal{T} \phi\|_* .
\end{equation*}
Therefore, this implies that $\mathcal{T}$ is injective and hence bijective on its range, and its inverse $\mathcal{T}^{-1}$ has a norm less than or equal to $1$ and can be extended to the $\operatorname{closure} M$ of $\operatorname{Range}(\mathcal{T})$.\smallskip\\
Now, on the other hand, we define
\begin{equation*}
    B_{u_0, f}(\phi):=\int_{\Omega} u_0 \phi(x, 0) d x+\int_0^T \int_{\Omega} \phi f d x d t .
\end{equation*}
Denoting $\phi_0:=\phi(x, 0)$, we get by Hölder inequality 
\begin{equation*}
    \left|B_{u_0, f}(\phi)\right| \leq\left\|u_0\right\|_{L^2(\Omega)}\left\|\phi_0\right\|_{L^2(\Omega)}+\|f\|_{L^2(0, T ; L^2(\Omega))}\|\phi\|_{L^2(0, T ; L^2(\Omega))} \leq c_{u_0, f}\|\phi\|_*,
\end{equation*}
and thus,
\begin{equation*}
    \left|B_{u_0, f}(\mathcal{T}^{-1} \psi)\right| \leq c\left\|\mathcal{T}^{-1} \psi\right\|_* \leq c\|\psi\|_*.
\end{equation*}
Since $B_{u_0,f}$ and $\mathcal{T}^{-1}$ both are linear, therefore their composition is also so, and by above line, $B_{u_0, f}\circ\mathcal{T}^{-1}$ is bounded too, therefore, by applying the Fréchet-Riesz Theorem again, there exists a unique $u \in M$ such that $B_{u_0, f}(\mathcal{T}^{-1} \psi)=\langle\psi, u\rangle_*$ for every $\psi \in M$. We denote $\phi=\mathcal{T}^{-1} \psi$ and so
\begin{equation*}
B_{u_0, f}(\phi)=\langle\mathcal{T} \phi, u\rangle_*=L_\phi(u)    
\end{equation*}
that is,
\begin{equation*}
\begin{array}{c}
\int_0^T \int_{\Omega}-u \phi_t d x d t+\int_0^T\int_{\Omega}\nabla u\cdot \nabla \phi d x d t+\int_0^T \mathcal{E}(u(x, t), \phi(x, t)) d t=\int_0^T \int_{\Omega} f \phi d x d t+\int_{\Omega} u(x, 0) \phi(x, 0) d x,    
\end{array}
\end{equation*}
where $\phi \in L^2(0, T ; H_0^s(\Omega))\cap L^2(0, T ; H_0^1(\Omega))\equiv L^2(0, T ; H_0^1(\Omega))$ and $\phi_t \in
L^2(0, T ; H^{-1}(\Omega))$. Finally, by a density argument, one can conclude, integrating by parts, that $u \in 
L^2(0, T ; H_0^1(\Omega))$, $u_t \in 
L^2(0, T ; H^{-1}(\Omega))$, and
\begin{equation*}
    \int_0^T \int_{\Omega} u_t \phi d x d t+\int_0^T\int_{\Omega}\nabla u\cdot \nabla \phi d x d t+\int_0^T \mathcal{E}(u(x, t), \phi(x, t)) d t=\int_0^T \int_{\Omega} f \phi d x d t .
\end{equation*}
Since $u \in 
L^2(0, T ; H_0^1(\Omega))$, $u_t \in 
L^2(0, T ; H^{-1}(\Omega))$ implies that $u\in \mathcal{C}([0, T], L^2(\Omega))$ and so we have $u(\cdot, t) \rightarrow u_0$ strongly in $L^2(\Omega)$, as $t \rightarrow 0$. Thus $u(x, t)$ is an energy solution of \cref{problem1}.\smallskip\\
Now we show that $u\geq 0$ 
provided that $f$ and $u_0$ are nonnegative. 
We write $u=u_+-u_-$, where $u_+=\operatorname{max}{\{u,0}\}\chi_\Omega$ and $u_-=\operatorname{max}{\{-u,0}\}\chi_\Omega$. We take $\phi=u_-\chi_{(0,\tilde{t})}, \tilde{t}>0$, as a test function. 
Since $f\geq 0$, and $\phi \geq 0$, we have
\begin{equation}{\label{contradiction1}}
0\leq \int_0^{T}\int_{\Omega}f\phi d x d t=\int_0^{\tilde{t}}\int_{\Omega}fu_{-} d x d t.
\end{equation}
On the other hand, since $u$ is $0$ in $(\mathbb{R}^n\backslash \Omega)\times (0,T)$, we have that
\begin{equation*}
    \begin{array}{c}
\int_0^{T}\iint_{\mathbb{R}^{2 n} }(u(x,t)-u(y,t))(\phi(x,t)-\phi(y,t)) K(x,y,t) d x d y d t\smallskip\\ =\int_0^{\tilde{t}}\iint_{\mathbb{R}^{2 n} \backslash(\Omega^c \times \Omega^c)}(u(x,t)-u(y,t))(u_{-}(x,t)-u_{-}(y,t)) K(x,y,t) d x d y d t \smallskip\\
=\int_0^{\tilde{t}} \int_{\Omega} \int_{\Omega}(u(x,t)-u(y,t))(u_{-}(x,t)-u_{-}(y,t)) K(x,y,t) d x d y d t +2 \int_0^{\tilde{t}}\int_{\Omega} \int_{\Omega^c}u(x,t) u_{-}(x,t) K(x,y,t) d y d x d t .
    \end{array}
\end{equation*}
\smallskip Moreover, $(u_{+}(x,t)-u_{+}(y,t))(u_{-}(x,t)-u_{-}(y,t)) \leq 0$, and thus 
\begin{equation*}
\begin{array}{l}
\int_0^{\tilde{t}}\int_{\Omega} \int_{\Omega}(u(x,t)-u(y,t))(u_{-}(x,t)-u_{-}(y,t)) K(x,y,t) d x d y d t \smallskip\\
\leq-\int_0^{\tilde{t}}\int_{\Omega} \int_{\Omega}(u_{-}(x,t)-u_{-}(y,t))^2 K(x,y,t) d x d y d t\leq 0 .
\end{array}
\end{equation*}
Further, we have
\begin{equation*}
\int_0^{\tilde{t}}\int_{\Omega}\int_{\Omega^c}u(x,t) u_{-}(x,t) K(x,y,t) d y d x d t =-\int_0^{\tilde{t}}\int_{\Omega}\int_{\Omega^c}u_{-}^2(x,t) K(x,y,t) d y d x d t \leq 0 .    
\end{equation*}
Therefore, we have shown that
\begin{equation*}
    \int_0^T\iint_{\mathbb{R}^{2 n} }(u(x,t)-u(y,t))(\phi(x,t)-\phi(y,t)) K(x,y,t) d x d y d t\leq 0,
\end{equation*}
Similarly since $\nabla u_+\cdot\nabla u_-= 0$, we get
\begin{equation*}
    \int_0^T\int_{\Omega}\nabla u\cdot \nabla \phi d x d t=\int_0^{\tilde{t}}\int_{\Omega}\nabla u_+\cdot \nabla u_- d x d t-\int_0^{\tilde{t}}\int_{\Omega}| \nabla u_-|^2 d x d t\leq 0.
\end{equation*}
Now 
since $u_0\geq 0$, so $u_-(\cdot,0)\equiv 0$, and we get
\begin{equation*}
    \begin{array}{rcl}
         \int_0^T\left\langle u_t, \phi\right\rangle d t
         &=&\int_0^{\tilde{t}}\int_\Omega  \frac{\partial u}{\partial t} u_- d x d t
        = \int_\Omega\int_{u(x,0)}^{u(x,\tilde{t})}\sigma_-d\sigma d x 
         =-\frac{1}{2}\int_\Omega (u_-(x,{\tilde{t}}))^2d x.
    \end{array}
\end{equation*}
Combining the above three inequalities, we get from \cref{contradiction1} that 
\begin{equation*}
\begin{array}{rcl}
0\leq    \int_0^T\int_{\Omega}f\phi d x d t&=& \int_0^T\left\langle u_t, \phi\right\rangle d t+\int_0^T\int_{\Omega}\nabla u\cdot \nabla \phi d x d t\smallskip\\&&+\frac{1}{2} \int_0^T \iint_{\mathbb{R}^{2n}}\frac{(u(x, t)-u(y, t))(\phi(x, t)-\phi(y, t))}{|x-y|^{n+2 s}} d x d y d t\leq-\frac{1}{2}\int_\Omega (u_-(x,{\tilde{t}}))^2d x \leq 0,
    \end{array}
\end{equation*}
and this gives that $||u_{-}(\cdot,\tilde{t})||_{L^2(\Omega)}=0$ 
for each $\tilde{t}>0$. 
Therefore $u_{-}\equiv 0$. So we conclude that $u\geq 0$. We observe that this comparison result also guarantees the uniqueness of energy solution to \cref{problem1}.
\end{proof}
\begin{remark}{\label{bounded}}
    Observing that $(u(x,t)-u(y,t))((u(x,t)-k)_+-(u(y,t)-k)_+)\geq 0$, for each $k$, we can show that for $(f,u_0)\in L^\infty(\Omega_T)\times L^\infty(\Omega)$, the weak solution $u\in L^\infty(\Omega_T)$. The proof will follow exactly similar to that of [\citealp{bdd}, Theorem 4.2.1].
\end{remark}
Now we relax the spaces where $f$ and $u_0$
lie. For the case of $L^1$ data, we 
consider the set
\begin{equation*}
    \begin{array}{rl}
         \mathcal{T}:=&\{\phi: \Omega \times[0, T] \rightarrow \mathbb{R}, \text { s.t. }-\phi_t-\Delta \phi+(-\Delta)^s \phi=\varphi \text{ in } \Omega_T, \\
&\varphi \in \mathcal{C}_0^\infty(\Omega_T) ,\phi \in L^{\infty}(\Omega \times(0, T)) \cap \mathcal{C}_{\operatorname{loc}}^{\alpha, \beta}(\Omega \times(0, T)),\\& \phi(x,0)\in L^\infty(\Omega),  \phi=0 \text { in }(\mathbb{R}^n \backslash \Omega) \times(0, T], \phi(x, T)=0 \text { in } \Omega\},
    \end{array}
\end{equation*}
where $\alpha, \beta \in(0,1)$.
\begin{definition}{\label{solution2}}
    Let $(f,u_0)\in L^1(\Omega_T)\times L^1(\Omega) $ be nonnegative functions. Then $u \in L^1(\Omega_T)
    $ is a very weak solution to \cref{problem1} if we have 
\begin{equation*}
    \iint_{\Omega_T} u\left(-\phi_t-\Delta \phi+(-\Delta)^s \phi\right) d x d t =\iint_{\Omega_T} f \phi d x d t+\int_{\Omega} u_0(x) \phi(x, 0) d x, \quad \forall \phi\in\mathcal{T}.
\end{equation*}
\end{definition}
The next existence result is following the lines of \cite{Peral}.
\begin{theorem}{\label{existence2}}
    For $(f, u_0) \in L^1(\Omega_T) \times L^1(\Omega)$ being nonnegative, \cref{problem1} has a unique nonnegative very weak solution $u$ in the sense of \cref{solution2}. 
\end{theorem}
\begin{proof}
    Firstly, we observe that the existence of valid test functions is guaranteed by the result in \cite{paper1}, \cite{localb}, \cite{par2}. 
We will now obtain the solution as a limit of solutions to approximated problems. Let $u_m\in 
L^2(0, T ; H_0^1(\Omega)) \cap L^{\infty}(\Omega_T)$ be the solution (exists by \cref{exist1}) to the approximated problem
\begin{eqnarray*}
    \begin{cases}
    (u_m)_t-\Delta u_m+(-\Delta)^s u_m=f_m & \text { in } \Omega_T , \\ u_m(x, t)=0 & \text { in }(\mathbb{R}^n \backslash \Omega) \times(0, T), \\ u_m(x, 0)=u_{0m}(x) & \text { in } \Omega;
    \end{cases}
\end{eqnarray*}
where $f_m=T_m(f(x, t))$ and $u_{0 m}=T_m(u_0(x))$ 
are $L^\infty$ functions. Using $(T_k(u_m))\chi_{(0,t)}$, for $t>0$ (admissible by [\citealp{Peral}, Proposition 3]) as a test function in the approximated problem, it holds that, 
\begin{equation}{\label{test}}
    \begin{array}{l}
         \int_{\Omega} L_k(u_m(x, t)) d x +\int_0^t\int_\Omega \nabla u_m\cdot\nabla T_k(u_m)d x d \theta\smallskip\\
 +\frac{1}{2} \int_0^t \int_Q \frac{\left(T_k\left(u_m(x, \theta)\right)-T_k\left(u_m(y, \theta)\right)\right)\left(u_m(x, \theta)-u_m(y, \theta)\right)}{|x-y|^{n+2 s}} d x d y d \theta\smallskip \\
 \leq k\|f\|_{L^1(\Omega_T)}+C_3(k)\left\|u_0\right\|_{L^1(\Omega)}+C_4(k)|\Omega|,
    \end{array}
\end{equation}
where $
L_k(\rho)=\int_0^\rho\left(T_k(\xi)\right) d \xi.
$ Notice that 
\begin{equation*}
    \left(T_k\left(u_m(x, \theta)\right)-T_k\left(u_m(y, \theta)\right)\right)\left(u_m(x, \theta)-u_m(y, \theta)\right)\geq \left(T_k\left(u_m(x, \theta)\right)-T_k\left(u_m(y, \theta)\right)\right)^2
\end{equation*}
and so
\begin{equation*}
\nabla u_m\cdot\nabla T_k(u_m)\geq|\nabla T_k(u_m)|^2  
\end{equation*}
and for $\rho>0$ we have
\begin{equation*}
    C_1(k) \rho-C_2(k) \leq L_k(\rho) \leq C_3(k) \rho+C_4(k) 
    \end{equation*}and 
    \begin{equation*}
        L_k(\rho) \geq C\left(T_k(\rho)\right)^{2} ,
\end{equation*}
where $C_1, C_2, C_3, C_4$ are constants depending only on $k$ and independent of $m$. 
Therefore taking supremum over $t\in(0,T]$ in \cref{test}, we get that $\left\{T_k(u_m)\right\}_m$ is bounded in $L^{\infty}(0, T ; L^2(\Omega)) \cap L^2(0, T ; H_0^s(\Omega))\cap L^2(0, T ; H_0^1(\Omega))$ and $\left\{u_m\right\}_m$ is bounded in $L^{\infty}(0, T ; L^1(\Omega))\subset L^1(\Omega_T)$. \smallskip\\
Now since by comparison principle proved in \cref{exist1}, $\left\{u_m\right\}_m$ is increasing in $m$, we get the existence of a measurable function $u$ such that $T_k(u)\in L^{\infty}(0, T ; L^2(\Omega)) \cap L^2(0, T ; H_0^s(\Omega)) \cap L^2(0, T ; H_0^1(\Omega))$, $u_m\uparrow u$ strongly in $L^1(\Omega_T)$ and $u_m \uparrow u$ a.e in $\Omega_T$. As each $u_m$ is an energy solution, therefore, $u_m\in \mathcal{C}([0,T],L^2(\Omega))\subset\mathcal{C}([0,T],L^1(\Omega))$ and at each time level $t\in[0,T]$, we have $u_m(\cdot,t)\in L^1(\Omega)$, this along with the monotonicity of $\{u_m\}_m$ in $m$ allows us to define the pointwise limit (a.e.) $u$ of $\{u_m\}_m$ in $\Omega$ for each time $t\in[0,T]$. Also $u$ satisfies $u(\cdot,0)=u_0(\cdot)$ in $L^1$ sense. Again as each $u_m=0$ in $(\mathbb{R}^n\backslash\Omega)\times(0,T)$, therefore $u$ also satisfies the same. We now prove that $u$ is a weak solution to \cref{problem1} in the sense of \cref{solution2}. Let $\phi \in \mathcal{T}
$, then as $u_m$ is the energy solution to the approximated problem, we have
\begin{equation*}
\iint_{\Omega_T}\left((u_m)_t-\Delta u_m+(-\Delta)^s u_m\right) \phi d x d t=\iint_{\Omega_T} f_m\phi d x d t .
\end{equation*}
Using the fact that $u_m \rightarrow u$ strongly in $L^1(\Omega_T)$ and $u_m(x,0)=u_{m0}(x)\rightarrow u_0(x)=u(x,0)$ in $L^1(\Omega)$, we have
\begin{equation*}
    \begin{array}{rcl}
\iint_{\Omega_T}\left((u_m)_t-\Delta u_m+(-\Delta)^s u_m\right) \phi d x d t 
& =&\iint_{\Omega_T} u_m\left(-(\phi)_t-\Delta \phi+(-\Delta)^s \phi\right) d x d t-\int_\Omega u_m(x,0)\phi(x,0) d x \smallskip\\&=&\iint_{\Omega_T} u_m\varphi d x d t-\int_\Omega u_m(x,0)\phi(x,0) d x\smallskip\\&\rightarrow& \iint_{\Omega_T} u\varphi d x d t-\int_\Omega u(x,0)\phi(x,0) d x  \smallskip\\&=&\iint_{\Omega_T} u\left(-(\phi)_t-\Delta \phi+(-\Delta)^s \phi\right) d x d t-\int_\Omega u(x,0)\phi(x,0) d x.
    \end{array}
\end{equation*}
Notice that in the second last line, in order to pass the limit, we have used the facts that $\phi(x, 0)\in L^\infty(\Omega)$ and $\varphi\in \mathcal{C}_0^\infty(\Omega_T)$. Also, since $\phi \in L^\infty(\Omega\times(0,T))$, and $f_m\rightarrow f$ in $L^1(\Omega_T)$, we get 
\begin{equation*}
    \iint_{\Omega_T} f_m \phi d x d t \rightarrow \iint_{\Omega_T} f \phi d x d t \text { as } m \rightarrow \infty .
\end{equation*}
Thus
\begin{equation*}
    \iint_{\Omega_T} u\left(-(\phi)_t-\Delta \phi+(-\Delta)^s \phi\right) d x d t =\iint_{\Omega_T} f\phi d x d t+\int_\Omega u_0(x)\phi(x,0) d x,
\end{equation*}
and $u$ is a weak solution to \cref{problem1}.
    For the uniqueness let $w$ be a weak solution of \cref{problem1} with $(f,u_0)=(0,0)$, i.e.
\begin{eqnarray*}
\begin{cases}
w_t-\Delta w+(-\Delta)^s w=0 & \text { in } \Omega_T, \\ w=0 & \text { in }(\mathbb{R}^n \backslash \Omega) \times(0, T), \\ w(x, 0)=0 & \text { in } \Omega ;
\end{cases}
\end{eqnarray*}
we want to prove that $w \equiv 0$. For that we take $F \in \mathcal{C}_0^{\infty}(\Omega_T)$, and let $\phi_F$ be the solution of the backward problem
\begin{eqnarray*}
\begin{cases}
-(\phi_F)_t-\Delta \phi_F+(-\Delta)^s \phi_F=F & \text { in } \Omega_T, \\ \phi_F(x, t)=0 & \text { in }(\mathbb{R}^n \backslash \Omega) \times(0, T], \\ \phi_F(x, T)=0 & \text { in } \Omega .
\end{cases}
\end{eqnarray*}
Taking $\phi_F$ as a test function we deduce that for any $F \in \mathcal{C}_0^{\infty}(\Omega_T)$,
\begin{equation*}
    \int_0^T \int_{\Omega} w F d x d t=0,
\end{equation*}
that means, $w=0$ in $\mathcal{D}^{\prime}(\Omega_T)$.
\end{proof}
Next, we state some important facts regarding the solution of \cref{problem1}.  
\begin{proposition}{\label{Harnack}} Let $\left(f, w_0\right)$ are non-negative functions such that $\left(f, w_0\right) \in L^\infty(\Omega_T) \times L^\infty(\Omega)$. Assume that $w$ be the weak solution to the problem
\begin{eqnarray*}
\begin{cases}
w_t-\Delta w+(-\Delta)^s w=f & \text { in } \Omega_T, \\ w=0 & \text { in }(\mathbb{R}^n \backslash \Omega) \times(0, T), \\ w(x, 0)=w_0(x) & \text { in } \Omega ;
\end{cases}
\end{eqnarray*}
then for all $0<t_1<T$ 
and for each $\omega \subset \subset \Omega$ fixed, there exists $C:=C\left(t_1, n,s ,\omega\right)>0$ such that
\begin{equation*}
    w(x, t) \geq C\left(t_1, n, s, \omega\right) \text { in } \omega\times [t_1,T).
\end{equation*}
\end{proposition}
\begin{proof}
We will use the weak Harnack inequality as in \cite{weakharnack}. First, we show the result for any arbitrary $t_2<T$. As $\bar{\omega}\times[t_1,t_2]\subset\Omega_T$ contains $\omega\times[t_1,t_2]$, so it is enough to show the result for $\bar{\omega}\times[t_1,t_2]$. Now since $\bar{\omega}\times[t_1,t_2]$ is compact and hence every open cover admits a finite subcover, it suffices to show the positivity of $w$ in a uniform neighbourhood of any arbitrary point $(\tilde{x},\tilde{t})\in \bar{\omega}\times[t_1,t_2]$. \\
As $\bar{\omega}\times[t_1,t_2]$ is compact, it has a finite and positive distance from the boundary of $\Omega_T$; let us denote this by $D$. We choose $0<r< \operatorname{min}\{\frac{D}{2}, 1\}$, and $(x_0,t_0)\in \Omega_T$ such that 
$
    t_0=\tilde{t}-\frac{13}{16}r^2 \text{ and } x_0=\tilde{x}.
$
Now for this $(x_0,t_0)$, we can choose $0<R=D$ such that $B_R(x_0)\times(t_0-r^2,t_0+r^2)\subset\Omega_T$, and then $r$ satisfies $r<\frac{R}{2}$. Clearly, as $f\geq 0$, $w$ is a supersolution to the homogeneous problem. Now by using the facts that $w\geq 0$ in $\Omega\times(0,T)$ and $w=0$ in $(\mathbb{R}^n\backslash\Omega)\times(0,T)$, we get $w\geq 0 $ in $\mathbb{R}^n\times(0,T)$. Then using [\citealp{weakharnack}, Theorem 2.8 and Corollary 2.9], and $f>0$ a.e. in $\Omega_T$, we get the existence of a positive  constant $C$ depending only on $\omega, [t_1,t_2],n $ and $s$ such that 
\begin{equation*}
    0<C=\fiint_{V^-\left(\frac{r}{2}\right)}w(x,t) d x d t\leq \underset{V^+\left(\frac{r}{2}\right)}{\operatorname{{ess }\operatorname{inf}}}w,
\end{equation*}
where $V^-\left(\frac{r}{2}\right)=B_{\frac{r}{2}(x_0)}\times(t_0-r^2,t_0-\frac{3}{4}r^2)$ and $V^+\left(\frac{r}{2}\right)=B_{\frac{r}{2}(x_0)}\times(t_0+\frac{3}{4}r^2,t_0+r^2)$. As $(\tilde{t}-\frac{1}{16}r^2, \tilde{t}+\frac{1}{16}r^2)\subset (t_0+\frac{3}{4}r^2,t_0+r^2)$, so the result follows.
\\Now, we will extend this up to $T$. For this we denote $D_0=\operatorname{dist}\{\omega,\partial \Omega\}$. As $T>0$, so $\exists \tilde{r}\in(0,1]$ and $\tilde{r}<\frac{D_0}{2}$ such that $\omega\times(T-2\tilde{r}^2,T)\subset \Omega_T$ with $T-\frac{1
}{4}\tilde{r}^2>t_1$. By the above argument for compact time intervals, we get 
\begin{equation*}
     w(x, t) \geq C\left(t_1, n, s, \omega\right)>0 \text { in } \omega\times [t_1,T-\frac{1}{4}\tilde{r}^2].
\end{equation*}
Therefore it just remains to show the positivity of $w$ in $\omega\times(T-\frac{1}{4}\tilde{r}^2,T)$. For $x_0\in\bar{\omega}$, $B_{\frac{\tilde{r}}{2}}(x_0)$ will form an open cover of $\bar{\omega}$, and by compactness, it will have a finite subcover. So it is enough to show the positivity of $w$ in $B_{\frac{\tilde{r}}{2}}(x_0)\times(T-\frac{1}{4}\tilde{r}^2,T)$. Now as $w$ is nonnegative in $\mathbb{R}^n\times(0,T)$ and $B_{D_0}(x_0)\times(T-2\tilde{r}^2,T)\subset \Omega_T$, so by [\citealp{weakharnack}, Theorem 2.8 and Corollary 2.9], we have
\begin{equation*}
    0<C=\fint_{T-2\tilde{r}^2}^{T-\frac{7}{8}\tilde{r}^2}\fint_{B_{\frac{\tilde{r}}{2}}(x_0)}w(x,t) d x d t\leq \underset{B_{\frac{\tilde{r}}{2}}(x_0)\times(T-\frac{1}{4}\tilde{r}^2,T)}{\operatorname{{ess }\operatorname{inf}}}w,
\end{equation*}
and hence we conclude.
\end{proof}
We will use the next parabolic Kato-type inequality to prove comparison results or a priori estimates.
 \begin{proposition}{\label{Kato}}
     Let $u \in 
     L^2(0, T ; H^1(\Omega)) 
     $ be a weak solution of 
     \begin{equation}{\label{katoeqn}}
         u_t-\Delta u+(-\Delta)^s u= f \text{ in } \Omega_T,
     \end{equation} with $f\in L^1(\Omega_T)$ 
    and let $\Phi \in \mathcal{C}^2(\mathbb{R})$ be a convex function such that $\Phi^\prime$ is bounded. Then
\begin{equation*}
    (\Phi(u))_t-\Delta \Phi(u)+(-\Delta)^s \Phi(u) \leq \Phi^{\prime}(u)\left(u_t-\Delta u+(-\Delta)^s u\right)
\end{equation*}
in the sense that for all $\psi \in \mathcal{C}^2(\Omega_T) \cap \mathcal{C}([0, T], L^2(\Omega))$, with the property that $\psi$ has spatial support compactly contained in $\Omega$, and $\psi \geq 0$ in $\Omega_T$ and $\psi(\cdot,T)=0$ and $\psi(\cdot,0)=0$, we have
\begin{equation*}
    \iint_{\Omega_T} \Phi(u)\left(-\psi_t-\Delta \psi +(-\Delta)^s \psi\right) d x d t\leq \iint_{\Omega_T} \Phi^{\prime}(u)f\psi d x d t.
\end{equation*}
\end{proposition}
\begin{proof} Assume that $u$ is smooth enough, otherwise one can use approximation argument.  Since $\Phi$ is convex so \begin{equation*}
    \Phi(u(x,t))-\Phi(u(y,t))\leq \Phi^\prime(u(x,t))(u(x,t)-u(y,t)),
\end{equation*} and we get
\begin{equation*}
    \begin{array}{rcl}
         (-\Delta)^s(\Phi(u(x,t)))&=&\int_{\mathbb{R}^n}\frac{\left(\Phi(u(x,t))-\Phi(u(y,t))\right)}{|x-y|^{n+2s}}d y \smallskip\\
&\leq& \int_{\mathbb{R}^n} \Phi^{\prime}(u(x,t))\frac{(u(x,t)-u(y,t))}{|x-y|^{n+2s}}d y\smallskip\smallskip\\&=&\Phi^{\prime}(u(x,t)) (-\Delta )^su(x,t) .
    \end{array}
\end{equation*}
Again, using the fact that the double derivative of a convex function is nonnegative, we get
\begin{equation*}
    -\Delta (\Phi(u))= -\Phi^\prime(u)\Delta u-\Phi^{\prime\prime}(u)\sum_{i=1}^n\left(\frac{\partial u}{\partial x_i}\right)^2\leq -\Phi^\prime(u)\Delta u.
\end{equation*}
Now if $\psi \in \mathcal{C}_0^\infty\left(\Omega_T\right) $, with $\psi \geq 0$ in $\Omega_T$, we have
\begin{equation*}
    \begin{array}{rcl}
          \iint_{\Omega_T} \Phi(u)\left(-\psi_t-\Delta \psi +(-\Delta)^s \psi\right) d x d t&=&\iint_{\Omega_T}\left(\Phi^\prime(u)u_t-\Delta (\Phi(u)) +(-\Delta)^s( \Phi(u))\right) \psi(x,t) d x d t \\
&\leq & \iint_{\Omega_T} \Phi^{\prime}(u)(u_t-\Delta u + (-\Delta )^su) \psi d x d t=\iint_{\Omega_T} \Phi^{\prime}(u)f\psi d x d t.
    \end{array}
\end{equation*}
Since $\psi \in \mathcal{C}_0^\infty\left(\Omega_T\right) $, with $\psi \geq 0$ in $\Omega_T$, can approximate the choice of test functions of our hypothesis, we get the desired result.
\end{proof}
\begin{remark}{\label{katore}} We choose the regularization of $|u|$ by
\begin{equation*}
\Phi_\epsilon(u)=\left(|u|^2+\epsilon^2\right)^{\frac{1}{2}}-\epsilon.
\end{equation*}
for $\epsilon \in(0,1)$. It is then easy to verify that,
\begin{enumerate}
    \item $\Phi_\epsilon(u) \rightarrow|u|$ uniformly as $\epsilon \rightarrow 0$,
    \item $\Phi_\epsilon^{\prime}(u)$ is bounded uniformly with respect to $\epsilon$.
    \item $\Phi_\epsilon(u)$ is convex.
\end{enumerate}
So, using the above theorem we get
\begin{equation*}
    \begin{array}{rcl}
        \iint_{\Omega_T}\Phi_\epsilon(u)\left(-\psi_t-\Delta \psi +(-\Delta)^s \psi\right) d x d t
&\leq &\iint_{\Omega_T} \Phi^\prime_\epsilon(u)(u_t-\Delta u + (-\Delta )^su) \psi d x d t=\iint_{\Omega_T} \Phi^\prime_\epsilon(u)f \psi d x d t.
    \end{array}
\end{equation*}
Now letting $\epsilon \rightarrow 0$, we get using \cref{katoeqn}
\begin{equation*}
    \begin{array}{rcl}
        \iint_{\Omega_T} |u|\left(-\psi_t-\Delta \psi +(-\Delta)^s \psi\right) d x d t
&\leq &\iint_{\Omega_T} sign(u)f \psi d x d t.
    \end{array}
\end{equation*}
Adding $\iint_{\Omega_T} u(-\psi_t-\Delta \psi + (-\Delta )^s\psi) d x d t$ to both side, we get
\begin{equation*}
    \iint_{\Omega_T} u_+\left(-\psi_t-\Delta \psi +(-\Delta)^s \psi\right) d x d t
\leq \iint_{\Omega_T} \chi_{\{u\geq 0\}}f \psi d x d t.
\end{equation*}
Therefore
\begin{equation*}
    \frac{\partial u_+}{\partial t}-\Delta u_++(-\Delta)^su_+\leq \left(\frac{\partial u}{\partial t}-\Delta u+(-\Delta)^ su\right)\operatorname{sign}^+u
\end{equation*}
in the weak sense.
\end{remark}
We now take $\gamma$ to be a positive continuous function over $\overline{\Omega}_T$ and proceed with the following comparison principle. 
\begin{lemma}{\label{comparison}}
Let $s \in(0,1)$ and $a \geq 0$. Consider $\left(f, u_0\right) \in L^{\infty}(\Omega_T) \times L^{\infty}(\Omega)$ to be non-negative bounded functions such that $\left(f, u_0\right) \neq(0,0)$. Assume that $v_1, v_2$ are two non-negative functions with finite energy such that $v_1, v_2 \in 
L^2((0, T) ; H_0^1(\Omega)) \cap L^{\infty}(\Omega_T)$ with
\begin{eqnarray*}
\begin{cases}
(v_1)_t-\Delta v_1+(-\Delta)^s v_1\leq\frac{f}{(v_1+a)^{\gamma(x,t)}} & \text { in } \Omega_T, \\ v_1(x,t)=0 & \text { in }(\mathbb{R}^n \backslash \Omega) \times(0, T), \\ v_1(x, 0)\leq u_0(x) & \text { in } \Omega ;
\end{cases}
\end{eqnarray*}
and
\begin{eqnarray*}
\begin{cases}
(v_2)_t-\Delta v_2+(-\Delta)^s v_2\geq\frac{f}{(v_2+a)^{\gamma(x,t)}} & \text { in } \Omega_T, \\ v_2(x,t)=0 & \text { in }(\mathbb{R}^n \backslash \Omega) \times(0, T), \\ v_2(x, 0)\geq u_0(x) & \text { in } \Omega ;
\end{cases}
\end{eqnarray*}
where $\gamma\in\mathcal{C}(\overline{\Omega}_T)$ and is positive. Then, $v_2 \geq v_1$ in $\Omega_T$.
\end{lemma}
\begin{proof}
    Define $u=v_1-v_2$, then $u \in 
    L^2(0, T ; H_0^1(\Omega))\cap L^{\infty}(\Omega_T)$. We show that $u^{+}=0$. By hypotheses, we get
\begin{equation*}
    u_t-\Delta u+(-\Delta)^s u \leq f\left(\frac{1}{\left(v_1+a\right)^{\gamma(x,t)}}-\frac{1}{\left(v_2+a\right)^{\gamma(x,t)}}\right).
\end{equation*}
Now since $\left(\frac{1}{\left(v_1+a\right)^{\gamma(x,t)}}-\frac{1}{\left(v_2+a\right)^{\gamma(x,t)}}\right) \leq 0$ in the set $\{u \geq 0\}$, then using \cref{Kato,katore}, it holds that
\begin{equation*}
(u_+)_t-\Delta (u_+)+(-\Delta)^s (u_+) \leq 0.    
\end{equation*}
Again we notice that $u_+(x,0)\leq 0$; therefore by comparison principle as of \cref{exist1}, we have $u_+\equiv 0$, and then we conclude.
\end{proof}
\begin{corollary}{\label{corl1}}
As a consequence of the previous comparison principle we get that for $a>0$ and $\gamma $ fixed, if $\left(f, u_0\right)$ are non-negative functions with $\left(f, u_0\right) \in L^\sigma(\Omega_T) \times L^\sigma(\Omega)$, $\sigma\geq 2$, then the problem
\begin{equation}{\label{approb}}
\begin{array}{lcr}
\begin{cases}
v_t-\Delta v+(-\Delta)^s v=\frac{f}{(v+a)^{\gamma(x,t)}} & \text { in } \Omega_T, \\ v(x,t)=0 & \text { in }(\mathbb{R}^n \backslash \Omega) \times(0, T), \\ v(x, 0)= u_0(x) & \text { in } \Omega ;
\end{cases}
\end{array}
\end{equation}
has a unique energy solution $v_a \in 
L^2(0, T ; H_0^1(\Omega))$. 
\\We note that for $\gamma$ being a constant, 
the existence of an energy solution can be shown using a monotonicity argument by observing that $\underline{v}=0$ is a subsolution and $\bar{v}$, the unique solution to the problem
\begin{eqnarray*}
\begin{cases}
\bar{v}_t-\Delta \bar{v}+(-\Delta)^s \bar{v}=\frac{f}{a^\gamma} & \text { in } \Omega_T, \\ \bar{v}(x,t)=0 & \text { in }(\mathbb{R}^n \backslash \Omega) \times(0, T), \\ \bar{v}(x, 0)= u_0(x) & \text { in } \Omega ;
\end{cases}
\end{eqnarray*}
is a supersolution. Then \cref{comparison} allows us to get the existence of a unique solution $v$ to \cref{approb} such that $v \in
L^2(0, T ; H_0^1(\Omega))$ and $0 \leq v \leq \bar{v}$ in $\Omega_T$.\\
Now if $\gamma$ is not constant and belongs to $\mathcal{C}(\overline{\Omega}_T)$, we denote \begin{equation*}
    \gamma^*=\underset{(x,t)\in\Omega_T}{\operatorname{sup}}\gamma(x,t) \text{ and }\gamma_*=\underset{(x,t)\in\Omega_T}{\operatorname{inf}}\gamma(x,t),
\end{equation*} and observe that the unique nonnegative solutions to the problems
\begin{eqnarray*}
\begin{cases}
\bar{v}_t-\Delta \bar{v}+(-\Delta)^s \bar{v}=\frac{f}{a^{\gamma^*}} & \text { in } \Omega_T, \\ \bar{v}(x,t)=0 & \text { in }(\mathbb{R}^n \backslash \Omega) \times(0, T), \\ \bar{v}(x, 0)= u_0(x) & \text { in } \Omega ;
\end{cases}
\end{eqnarray*}
and
\begin{eqnarray*}
\begin{cases}
\bar{v}_t-\Delta \bar{v}+(-\Delta)^s \bar{v}=\frac{f}{a^{\gamma_*}} & \text { in } \Omega_T, \\ \bar{v}(x,t)=0 & \text { in }(\mathbb{R}^n \backslash \Omega) \times(0, T), \\ \bar{v}(x, 0)= u_0(x) & \text { in } \Omega ;
\end{cases}
\end{eqnarray*}
are sub and supersolution (respectively super and subsolution) of 
\begin{equation}{\label{approb1}}
\begin{array}{rcl}
\begin{cases}
\bar{v}_t-\Delta \bar{v}+(-\Delta)^s \bar{v}=\frac{f}{a^{\gamma(x,t)}} & \text { in } \Omega_T, \\ \bar{v}(x,t)=0 & \text { in }(\mathbb{R}^n \backslash \Omega) \times(0, T), \\ \bar{v}(x, 0)= u_0(x) & \text { in } \Omega ;
\end{cases}
\end{array}
\end{equation}
for $a\geq1$ (respectively for $a<1$). Then, by monotonicity argument and comparison principle similar to \cref{comparison}, we get the existence of a unique nonnegative solution to \cref{approb1}, which will be a supersolution to \cref{approb} with $\underline{v}=0$ being its subsolution. So \cref{approb} has a unique nonnegative solution in this case too. We note that these solutions also satisfy the comparison principle, in the sense that if $a_1<a_2$, then we have $v_{a_2} \leq v_{a_1}$. 
\end{corollary}
We now list approximated problems for $\gamma$ being a constant. 
Firstly when $(f,u_0)\in L^\infty(\Omega_T)\times L^\infty(\Omega)$, we consider the following problem for each $k\in \mathbb{N}$,
\begin{equation}{\label{approx}}
\begin{array}{lcr}
\begin{cases}
(u_k)_t-\Delta u_k+(-\Delta)^su_k=\frac{f(x,t)}{(u_k+\frac{1}{k})^\gamma} 
 & \mbox{ in } \Omega_T
 ,\\
u_k=0 &  \mbox{ in } (\mathbb{R}^n\backslash \Omega) \times (0,T),\\
u_k(x,0)=u_0(x)  & \mbox{ in } \Omega.
\end{cases}
\end{array}
\end{equation}
We now take $f$ and $u_0$ may not be bounded and consider 
\begin{equation}{\label{approx4}}
\begin{array}{lcr}
\begin{cases}
(u_k)_t-\Delta u_k+(-\Delta)^su_k=\frac{f_k(x,t)}{(u_k+\frac{1}{k})^\gamma} 
 & \mbox{ in } \Omega_T
 ,\\
u_k=0 &  \mbox{ in } (\mathbb{R}^n\backslash \Omega) \times (0,T),\\
u_k(x,0)=u_{0k}(x)  & \mbox{ in } \Omega;
\end{cases}
\end{array}
\end{equation}
where $f_k(x,t)=T_k(f(x,t))$ and $u_{0k}(x)=T_k(u_0(x))$.\\
To treat variable exponent, we take $(f,u_0)\in L^1(\Omega_T)\times L^1(\Omega)$ and consider the following approximating problem. 
\begin{equation}{\label{approxvari}}
\begin{array}{lcr}
\begin{cases}
(u_k)_t-\Delta u_k+(-\Delta)^su_k=\frac{f_k(x,t)}{(u_k+\frac{1}{k})^{\gamma(x,t)}} 
 & \mbox{ in } \Omega_T
 ,\\
u_k=0 &  \mbox{ in } (\mathbb{R}^n\backslash \Omega) \times (0,T),\\
u_k(x,0)=u_{0k}(x)  & \mbox{ in } \Omega.
\end{cases}
\end{array}
\end{equation}
We will prescribe conditions on $f$ and $u_0$ later on. With these settings, we have the following existence result.
\begin{theorem}{\label{approxexxist}}
    For each of the above problems \cref{approx,approx4,approxvari}, there exists a unique nonnegative energy solution $u_k$ for each $k\in\mathbb{N}$ satisfying the followings:
    \\(a) each $u_k$ is bounded,\\
    (b) the sequence $\{u_k\}_k$ is increasing in $k$,\\
    (c) for each $t_0>0$ and $\omega\subset\subset\Omega_T$, $\exists C(\omega,t_0,n,s)$ such that $u_k\geq C(\omega,t_0,n,s)$.
\end{theorem}
\begin{proof}
The existence, uniqueness and nonnegativity of $u_k$'s follow from the comparison principle in \cref{comparison} and \cref{corl1}. Since $f$ and $u_0$ are nonnegative, therefore $0\leq f_k\leq f_{k+1}$ and $0\leq u_{0k}\leq u_{0(k+1)}$ for each $k$ and then using the similar technique as that of \cref{comparison}, we obtain that the sequence $\left\{u_k\right\}_k$ is increasing in $k$. 
Therefore $u_k \geq u_1$ for all $k$. Now for $f\in L^\infty(\Omega_T)$ we have $
    \frac{f}{(u_1+1)^{\gamma}}
\leq f$ and for $f\notin L^\infty(\Omega_T)$ we have $
    \frac{f_1}{(u_1+1)^{\gamma(x,t)}}
\leq f_1$ 
and hence $\frac{f_1}{(u_1+1)^{\gamma(x,t)}}\in L^\infty(\Omega_T)$. Therefore $u_1\in L^\infty(\Omega_T)$ and by \cref{Harnack}, we obtain that for all $\omega \subset \subset \Omega$ and for all $t_0>0
$, $ u_1(x, t) \geq$ $C\left(\omega, t_0,n,s\right) $ in $\omega \times\left[t_0, T\right)$. Thus for all $k \geq 1$, we have $u_k(x, t) \geq C\left(\omega, t_0,n,s\right) $ in $\omega \times\left[t_0, T\right)$. Also, we observe that for each $k$, $
    \frac{f_k}{(u_k+\frac{1}{k})^{\gamma(x,t)}}
\leq \operatorname{max}\{1,k^{\gamma^*}\}f_k, 
$
    where $\gamma^*=\underset{(x,t)\in\Omega_T}{\operatorname{sup}}\gamma(x,t)$. Therefore each $u_k$ is bounded.
\end{proof}
\begin{remark}{\label{ubase}}
    As each $u_k\in \mathcal{C}([0,T],L^2(\Omega))\subset\mathcal{C}([0,T],L^1(\Omega))$, therefore at each time level $t\in[0,T]$, we have $u_k(\cdot,t)\in L^1(\Omega)$, this along with the monotonicity of $\{u_k\}_k$ in $k$ allows us to define the pointwise limit (a.e.) $u$ of $\{u_k\}_k$ in $\Omega$ at each time $t\in[0,T]$. Also $u$ satisfies $u(\cdot,0)=u_0(\cdot)$ in $L^1$ sense. Again as each $u_k=0$ in $(\mathbb{R}^n\backslash\Omega)\times(0,T)$, therefore $u$ also satisfies the same. Further we have $u\geq u_k$ for each $k$ and hence $u(x, t) \geq C\left(\omega, t_0,n,s\right) $ in $\omega \times\left[t_0, T\right)$ for each $\omega\subset\subset\Omega$ and $t_0>0$.
\end{remark}
To treat the singular case, we define next as:
\begin{definition}{\label{very weak}}
    Let $\left(f, u_0\right) \in L^1\left(\Omega_T\right) \times L^1(\Omega)$ be a pair of non-negative functions and $\gamma>0$ is a constant. We say that $u$ is a very weak solution to the problem
    \begin{equation}{\label{problem}}
    \begin{array}{lcr}
\begin{cases}
u_t-\Delta u+(-\Delta)^su=\frac{f(x,t)}{u^\gamma} 
&  \mbox{ in } \Omega_T
,\\
u=0 &  \mbox{ in } (\mathbb{R}^n\backslash \Omega) \times (0,T),\\
u(x,0)=u_0(x)  & \mbox{ in } \Omega;
\end{cases}
\end{array}
\end{equation}
if $u \in 
L^1(\Omega_T)$ satisfying $u\equiv 0$ in $(\mathbb{R}^n\backslash\Omega)\times(0,T)$,
$\forall \omega \subset \subset \Omega$ and $\forall t_0>0$, $\exists c\equiv c(\omega, t_0,n,s)>0$ such that $ u(x, t) \geq c>0$, in $\omega \times\left[t_0, T\right)$, $u(\cdot,0)=u_0(\cdot)$, and for all $\varphi \in \mathcal{C}_0^{\infty}\left(\Omega_T\right)$, we have
\begin{equation*}
    \iint_{\Omega_T} u\left(-\varphi_t-\Delta \varphi+(-\Delta)^s \varphi\right) d x d t=\iint_{\Omega_T} \frac{f \varphi}{u ^\gamma} d x d t.
\end{equation*}
\end{definition}
\begin{definition}{\label{very weak variable}}
    Let $\left(f, u_0\right) 
    $ be a pair of non-negative functions with $u_0\in L^1(\Omega)$, $f\in L^1(\Omega_T)$ 
    and $\gamma\in \mathcal{C}(\overline{\Omega}_T)$ be a positive function. Then we say that $u$, such that $u\equiv 0$ in $(\mathbb{R}^n\backslash\Omega)\times(0,T)$, is a very weak solution to the problem
    \begin{equation}{\label{problemvari}}
    \begin{array}{lcr}
\begin{cases}
u_t-\Delta u+(-\Delta)^su=\frac{f(x,t)}{u^{\gamma(x,t)}} 
&  \mbox{ in } \Omega_T
,\\
u=0 &  \mbox{ in } (\mathbb{R}^n\backslash \Omega) \times (0,T),\\
u(x,0)=u_0(x),& \mbox{ in } \Omega;
\end{cases}
\end{array}
\end{equation}
if $u \in L^1(\Omega_T)
$ and
$\forall \omega \subset \subset \Omega$ and $\forall t_0>0$, $\exists c\equiv c(\omega, t_0,n,s)>0$ such that $ u(x, t) \geq c>0$, in $\omega \times\left[t_0, T\right)$, $u(\cdot,0)=u_0(\cdot)$, and for all $\varphi \in \mathcal{C}_0^{\infty}\left(\Omega_T\right)$, we have
\begin{equation*}
    \iint_{\Omega_T} u\left(-\varphi_t-\Delta \varphi+(-\Delta)^s \varphi\right) d x d t=\iint_{\Omega_T} \frac{f \varphi}{u ^{\gamma(x,t)}} d x d t.
\end{equation*}
\end{definition}
\subsection{\textbf{Main Results}}
First, we consider $\gamma$ to be a constant and then take it as a positive continuous function. In this section, we state our main results depending on $\gamma$ and the regularity of initial conditions.
\subsubsection{\textbf{Existence for bounded data for $\gamma$ being a constant}}
In the case of bounded data, we will have the next existence result.
\begin{theorem}{\label{mainth1}} Let $\left(f, u_0\right) \in L^\infty\left(\Omega_T\right) \times L^\infty(\Omega)$ be a pair of non-negative functions and $\gamma>0$. Then \cref{problem} has a bounded very weak positive solution $u
$ in the sense of \cref{very weak} such that $u^\frac{\gamma+1}{2}\in 
L^2(0, T ; H_0^1(\Omega))$ and $u\in \mathcal{C}([0,T],L^2(\Omega))$. Moreover if $\gamma \leq 1$ or $\operatorname{Supp}(f) \subset \subset \Omega_T$, then $u \in 
L^2(0, T ; H_0^1(\Omega))$. 
\end{theorem}
\subsubsection{\textbf{Existence for general data for $\gamma$ being a constant}} In this subsection, we will consider general data. According to the regularity of initial conditions, we will consider two cases as $u_0\in L^{\gamma+1}(\Omega)$ and $u_0\in L^{1}(\Omega)$. Let us state the following existence results.
\begin{theorem}{\label{mainth2}} Let $\left(f, u_0\right) \in L^1\left(\Omega_T\right) \times L^{\gamma+1}(\Omega)$ be a pair of non-negative functions and $\gamma>0$. Then \cref{problem} has a very weak positive solution $u$ in the sense of \cref{very weak} such that $u^\frac{\gamma+1}{2}\in 
L^2(0, T ; H_0^1(\Omega))$ and $u\in L^2\left(0,T;L^{\sigma}(\Omega)\right)$, where $\sigma=\frac{n(1+\gamma)}{n-2}$ and $\sigma\geq 2$ if $\gamma\geq 1$.
\end{theorem}
\begin{theorem}{\label{mainth3}} Let $\left(f, u_0\right) \in L^1(\Omega_T) \times L^{1}(\Omega)$ be a pair of non-negative functions and $\gamma>0$. Then \cref{problem} has a very weak positive solution $u$ in the sense of \cref{very weak} such that $T_k(u)^\frac{\gamma+1}{2}\in 
L^2(0, T ; H_0^1(\Omega))$ for all $k>0$.
\end{theorem}
\subsubsection{\textbf{Improved results for $\gamma $ being a constant}}
We now break $\gamma$ in three parts namely $0<\gamma <1$, $\gamma=1$ and $\gamma >1$. We improve our results to find solutions in better spaces. For this, we take $u_0\in L^{\operatorname{max}(\gamma+1,2)}(\Omega).$ The case $\gamma=1$ will be same as that of \cref{mainth2}. For $\gamma >1$, we cannot find solutions in $
L^2(0, T ; H_0^1(\Omega))$. In fact if we look for $
L^2(0, T ; H^1(\Omega))$ estimates, we will only get them in $
L^2(t_0, T ; H_{\operatorname{loc}}^1(\Omega))$ for each $t_0>0$. We state our next result for $\gamma >1$ as follows:
\begin{theorem}{\label{mainth4}} Let $\gamma >1$. Assume that $\left(f, u_0\right) \in L^1\left(\Omega_T\right) \times L^{\gamma+1}(\Omega)$ be a pair of non-negative functions. Then \cref{problem} has a very weak positive solution $u$ in the sense of \cref{very weak} such that $u^\frac{\gamma+1}{2}\in 
L^2(0, T ; H_0^1(\Omega))$ and $u\in 
L^2(t_0, T ; H_{\operatorname{loc}}^1(\Omega))\cap L^\infty(0,T;L^{1+\gamma}(\Omega))$ for each $t_0>0$.
\end{theorem}
Now we consider the case $0<\gamma<1$, and state our results as follows:
\begin{theorem}{\label{mainth5}} Let $0<\gamma<1$. Assume that $u_0 \in L^{2}(\Omega)$ with $u_0\geq 0$ and $f \geq 0$ is such that\smallskip\\
$i)$ $f \in L^{\frac{2}{\gamma+1}}\left(0, T ; L^{\left(\frac{2^*}{1-\gamma}\right)^{\prime}}(\Omega)\right)$,
or\\
$ii)$ $f \in L^{\bar{m}}\left(\Omega_T\right)$ with $\bar{m}:=\frac{2(n+2 )}{2(n+2 )-n(1-\gamma)}$.\\
Then \cref{problem} has a very weak solution $u \in 
L^2(0, T ; H_0^1(\Omega))\cap L^{\infty}(0, T ; L^2(\Omega))$ in the sense of \cref{very weak}.
\end{theorem}
\begin{remark} As $\gamma<1$, so
    \begin{equation*}
        \left(\frac{2^*}{1-\gamma}\right)^{\prime}=\frac{2 n}{2 n-(1-\gamma)(n-2 )}<\bar{m}<\frac{2}{\gamma+1} ,
    \end{equation*}
and the two spaces $L^{\frac{2}{\gamma+1}}\left(0, T ; L^{\left(\frac{2 ^*}{1-\gamma}\right)^{\prime}}(\Omega)\right)$ and $L^{\bar{m}}\left(\Omega_T\right)$ are not comparable. Also the case $\gamma=1$ cannot be considered here since in the proof of \cref{mainth5}, we will use Hölder inequality with exponents $\frac{2}{\gamma+1}$ and $\frac{2}{1-\gamma}$. 
If $\gamma\to1$, then $\bar{m}\to1$ and $\frac{2}{\gamma+1}$ and $\left(\frac{2^*}{1-\gamma}\right)^{\prime}
$ both tend to $1$, so that $f$ will belong to $L^1\left(\Omega_T\right)$.
\end{remark}
Now for $m<\bar{m}$, we no longer find solutions in $
L^2(0, T ; H_0^1(\Omega))$ but in a larger space depending on $m$.
\begin{theorem}{\label{mainth6}} Let $0<\gamma<1$. Assume that $0 \leq f \in L^m\left(\Omega_T\right)$, with $1 \leq m<\bar{m}$, and that $u_0 \in L^{2}(\Omega)$ be nonnegative. Then the problem \cref{problem} admits a very weak solution $u \in 
L^{\bar{q}}(0, T ; W_0^{1, \bar{q}}(\Omega)) \cap L^{\infty}(0, T ; L^{1+\gamma}(\Omega))$, with
\begin{equation*}
    \bar{q}=\frac{m(\gamma+1)(n+2 )}{n+2 -m(1-\gamma)} .
\end{equation*}
Moreover $u \in L^\sigma\left(\Omega_T\right)$, where
\begin{equation*}
    \sigma=\frac{m(\gamma+1)(n+2 )}{n-2 (m-1)} .
\end{equation*}
\end{theorem}
\begin{remark}
    Observe that we can get rid of the fact that $u\in L^{\bar{q}}(0, T ; W_0^{s_1, \bar{q}}(\Omega))$ with $s_1<s$ as is done in [\citealp{abdellaoui1}, Theorem 11]. For our case, due to the presence of the leading Laplacian operator, we will get $u\in L^{\bar{q}}(0, T ; W_0^{s, \bar{q}}(\Omega))$.
    \end{remark}
    \begin{remark}Clearly $\bar{q} \geq m(\gamma+1)>1$ and $\sigma \geq m(\gamma+1)>1$. Also $m<\bar{m}$ is equivalent to $\bar{q}<2$ which implies $L^{2}\left(0, T ; H_0^1(\Omega)\right) \subset$ $L^{\bar{q}}(0, T ; W_0^{1, \bar{q}}(\Omega))$. In \cref{mainth6} the case $\gamma=1$ is not allowed since it yields $\bar{q}=2 m \geq 2$ which contradicts $\bar{q}<2$. 
\end{remark}
\subsubsection{\textbf{Further summability for $\gamma$ being a constant}} Here we state two results regarding the optimal summability of our solutions in terms of summability of the initial data $f\in L^r(0, T ; L^q(\Omega))$. For the $L^\infty$-boundedness, we take $u_0\in L^\infty(\Omega)$, and $1/r$ and $1/q$ are in the Aronson-Serrin domain, see \cite{AS,Peral}. We now state the corresponding theorem as:
\begin{theorem}{\label{mainth5.5}}
Assume that $f \in L^r(0, T ; L^q(\Omega))$ with $r, q$ satisfying
\begin{equation*}
\frac{1}{r}+\frac{n}{2 q }<1  , 
\end{equation*}
and suppose that $u_{0} \in L^{\infty}(\Omega)$. Then there exists a positive constant $c$ such that the unique finite energy solution of \cref{approx4} satisfies
\begin{equation*}
    \left\|u_k(x, t)\right\|_{L^{\infty}\left(\Omega_T\right)} \leq c
    ,
\end{equation*}
and the solution $u$ obtained as the limit of $u_k$ satisfies, 
$u\in L^\infty(\Omega_T)
$. Moreover for $\gamma\leq 1$, $u\in L^2(0,T;H^1_0(\Omega))$.
\end{theorem}
Outside the Aronsom-Serrin zone, the solutions are not expected to be bounded. We divide the region by the straight line $\frac{1}{r}=\frac{n}{n-2} \frac{1}{q}-\frac{2}{n-2}$ in two parts. Further, we take $u_0\equiv 0$. The following summability results we will get for this case.
\begin{theorem}{\label{mainth5.7}}
Assume $u_{0}(x) \equiv 0$ and $f \in L^r\left(0, T ; L^q(\Omega)\right)$, with $r>1, q>1$ satisfy
\begin{equation*}
    1<\frac{1}{r}+\frac{n}{2 q }.
\end{equation*}
Further, assume that for $\gamma<1$,\\
$i)$ if $\frac{1}{r}<\frac{n}{n-2} \frac{1}{q}-\frac{2}{n-2}$,  then $q> \left(\frac{2^*}{1-\gamma}\right)^\prime$, and\\
$ii)$ if $\frac{1}{r}\geq\frac{n}{n-2} \frac{1}{q}-\frac{2}{n-2}$, then $r>\frac{2}{1+\gamma}$.\\
Then there exists a positive constant $c$ such that the sequence of finite energy solutions of \cref{approx4} satisfies
\begin{equation*}
\left\|u_k\right\|_{L^{\infty}\left(0, T ; L^{2 \sigma}(\Omega)\right)}+\left\|u_k\right\|_{L^{2 \sigma}\left(0, T ; L^{2^* \sigma}\right)} \leq c,
\end{equation*}
where
\begin{equation*}
\sigma=\left\{\begin{array}{lll}
\frac{q(n-2 )(\gamma+1)}{2(n-2 q )} & \text { if } \frac{1}{r}<\frac{n}{n-2} \frac{1}{q}-\frac{2}{n-2}, \smallskip\\
\frac{q r n(\gamma+1)}{2(n r+2q-2 q r)} & \text { if } \frac{1}{r} \geq \frac{n}{n-2} \frac{1}{q}-\frac{2}{n-2}.
\end{array}\right.    
\end{equation*}
Further, the solution $u$ obtained as the limit of $u_k$ satisfies, 
$u\in L^\infty(0,T;L^{2\sigma}(\Omega))\cap L^{2\sigma}(0,T;L^{2^*\sigma}(\Omega))$.
\end{theorem}
\begin{remark}
    In \cref{mainth5.7}, we observe that the conditions needed for $\gamma<1$ are obvious, as \cref{mainth5} implies that if $r=\frac{2}{1+\gamma}$ and $q= \left(\frac{2^*}{1-\gamma}\right)^\prime$, then by Sobolev embedding, $u\in L^2(0,T;L^{2^*}(\Omega))$ which matches with \cref{mainth5.7} as we get $\sigma=1$ in this case.
\end{remark}
\begin{remark}
    We note that, for $\gamma\geq 1$, the singularity allows us to get better summability results as we can go up to $r=1$ and $q=1$, which is not allowed for the nonsingular case (see \cite{Summability,Peral}), whereas for $\gamma<1$, the singularity gives us better summability for the spatial exponent $q$ only.
\end{remark}
\subsubsection{\textbf{Existence results for $\gamma$ being a function}} We now consider $\gamma$ to be a positive continuous function on $\overline{\Omega}_T$. We will mainly see the behaviour of $\gamma$ near the parabolic boundary, and accordingly, we will state two existence results. We recall the strip around the parabolic boundary given by $(\Omega_T)_\delta=\{(x,t)\in\Omega_T:\operatorname{dist}((x,t),\Gamma_T)<\delta\}$ for $\delta>0$, where $\Gamma_T=(\Omega\times\{t=0\})\cup(\partial\Omega\times(0,T))$.
\begin{theorem}{\label{mainth7}}
Let $\exists\delta>0$ 
such that $\gamma(x,t)\leq1$ in $(\Omega_T)_\delta$. Also let $u_0\in L^2(\Omega)$ with $u_0\geq0$ and $f\geq0$ satisfies
   \smallskip \\
$i)$ $f \in L^{2}\left(0, T ; L^{\left(\frac{2n}{n+2}\right)}(\Omega)\right)$,
or\smallskip\\
$ii)$ $f \in L^{\bar{r}}\left(\Omega_T\right)$ with $\bar{r}:=\frac{2(n+2 )}{n+4}$.\\
Then \cref{problemvari} has a very weak solution $u \in 
L^2(0, T ; H_0^1(\Omega))\cap L^{\infty}(0, T ; L^2(\Omega))$ in the sense of \cref{very weak variable}.
\end{theorem}
\begin{remark}
    Since $\frac{2n}{n+2}<\frac{2(n+2 )}{n+4}=\bar{r}<2$, so the two spaces $ L^{2}\left(0, T ; L^{\left(\frac{2n}{n+2}\right)}(\Omega)\right)$ and $L^{\bar{r}}\left(\Omega_T\right)$ are not comparable. Also for the constant case (\cref{mainth5}) we got larger possible spaces $L^{\frac{2}{\gamma+1}}\left(0, T ; L^{\left(\frac{2 ^*}{1-\gamma}\right)^{\prime}}(\Omega)\right)$ and $L^{\bar{m}}\left(\Omega_T\right)$ for the belonging of initial data $f$, which gives us broader results, this is the cause of considering the constant and nonconstant cases separately.
\end{remark}
\begin{theorem}{\label{mainth8}}
    Assume that for some $\gamma^*>1$ and some $\delta>0$, we have $\|\gamma\|_{L^\infty((\Omega_T)_\delta)}<\gamma^*$. Assume that $u_0\in L^{\gamma^*+1}(\Omega)$ with $u_0\geq0$ and $f\geq0$ is such that 
   \smallskip \\
$i)$ $f \in L^{\gamma^*+1}\left(0, T ; L^{\left(\frac{n(\gamma^*+1)}{n+2\gamma^*}\right)}(\Omega)\right)$,
or\smallskip\\
$ii)$ $f \in L^{\tilde{r}}\left(\Omega_T\right)$ with $\tilde{r}:=\frac{(n+2 )(\gamma^*+1)}{n+2(\gamma^*+1)}$.\\
Then \cref{problemvari} admits a very weak nonnegative solution $u$ in the sense of \cref{very weak variable} such that $u^\frac{\gamma^*+1}{2}\in 
L^2(0, T ; H_0^1(\Omega))$ and $u\in 
L^2(t_0, T ; H_{\operatorname{loc}}^1(\Omega))\cap L^\infty(0,T;L^{1+\gamma^*}(\Omega))$ for each $t_0>0$.
\end{theorem}
\begin{remark}
    Since $
        \frac{n(\gamma^*+1)}{n+2\gamma^*}<\frac{(n+2 )(\gamma^*+1)}{n+2(\gamma^*+1)}=\tilde{r}<\gamma^*+1,
    $ 
    so the two spaces $L^{\gamma^*+1}\left(0, T ; L^{\left(\frac{n(\gamma^*+1)}{n+2\gamma^*}\right)}(\Omega)\right)$ and $L^{\tilde{r}}\left(\Omega_T\right)$ are not comparable. Also, for the constant case (\cref{mainth4}), we got the largest possible space $L^1(\Omega_T)$ for the belonging of initial data $f$.
\end{remark}
\section{Proof of main results}
\subsection*{\textbf{Proof of \cref{mainth1}}}
We consider the approximated problems \cref{approx} in this case and show that $\left\{u_k\right\}_k$ is bounded in $L^{\infty}\left(\Omega_T\right)$. Note that the existence and other properties of $\{u_k\}_k$ follow by \cref{approxexxist}. We define $w_k=$ $H\left(u_k\right)=u_k^{\gamma+1}$, and note that as $u_k\in L^2(0,T;H^1_0(\Omega))$ is bounded, so $u_k^{\gamma+1}\in L^2(0,T;H^1_0(\Omega))$. Then, using Kato inequality as in \cref{Kato}, we get
\begin{equation*}
    \left(w_k\right)_t-\Delta w_k+(-\Delta)^s w_k \leq H^{\prime}(u_k)\left(\left(u_k\right)_t-\Delta u_k+(-\Delta)^s u_k\right) \leq(\gamma+1) f  \quad \text{ in weak sense}.
\end{equation*}
Now let $\vartheta$ be the unique solution to the problem
\begin{eqnarray*}
\begin{cases}
\vartheta_t-\Delta \vartheta+(-\Delta)^s \vartheta=(\gamma+1)f & \text { in } \Omega_T, \\ \vartheta=0 & \text { in }(\mathbb{R}^n \backslash \Omega) \times(0, T), \\ \vartheta(x, 0)=H(u_0(x)) & \text { in } \Omega .
\end{cases}
\end{eqnarray*}
As $f$ and $u_0$ are bounded, so by \cref{bounded}, $\vartheta\in L^\infty(\Omega_T)$ and by comparison principle $w_k \leq \vartheta$ for all $k$. Hence, $u_k^{\gamma+1} \leq \vartheta$ and the claim follows.
\smallskip\\Now since $u_k\in L^\infty(\Omega_T)\cap L^2(0,T; H^1_0(\Omega))$ and nonnegative, then for any $\varepsilon>0$ and $\theta>0$, $(\left(u_k(x, t)+\varepsilon\right)^\theta-\varepsilon^\theta)\in L^2(0,T; H^1_0(\Omega))$. So choosing $0<\varepsilon<1/k
$, for $t\in(0,T]$, we take $(\left(u_k(x, \theta)+\varepsilon\right)^\gamma-\varepsilon^\gamma) \chi_{(0, t)}$ as a test function in \cref{approx}, and it holds that
\begin{equation*}
    \begin{array}{l}
        \int_0^t \int_{\Omega}\left(u_k\right)_t\left((u_k+\varepsilon))^\gamma-\varepsilon^\gamma\right) d x d \theta +\int_0^t\int_\Omega \nabla u_k\cdot\nabla \left((u_k+\varepsilon)^\gamma-\varepsilon^\gamma\right)d x d \theta\smallskip\\
 +\frac{1}{2} \int_0^t \int_Q \frac{(u_k(x, \theta)-u_k(y, \theta))((u_k(x, \theta)+\varepsilon)^\gamma-\left(u_k(y, \theta)+\varepsilon\right)^\gamma)}{|x-y|^{n+2 s}} d x d y d \theta\smallskip\\
\leq \iint_{\Omega_t} f(x, \theta)
d x d \theta .
    \end{array}
\end{equation*}
Now letting $\varepsilon \to 0$, by Fatou's lemma, we get for all $t \leq T$
\begin{equation}{\label{takesup}}
    \begin{array}{l}
          \frac{1}{\gamma+1} \int_{\Omega} u_k^{\gamma+1}(x, t) d x+\int_0^t\int_\Omega \nabla u_k\cdot\nabla u_k^{\gamma}d x d\theta\smallskip\\+\frac{1}{2} \int_0^t \int_Q \frac{\left(u_k^\gamma(x, \theta)-u_k^\gamma(y, \theta)\right)\left(u_k(x, \theta)-u_k(y, \theta)\right)}{|x-y|^{n+2 s}} d x d y d \theta\smallskip \\
  \leq \iint_{\Omega_t} f d x d \theta+\frac{1}{\gamma+1} \int_{\Omega} u_0^{\gamma+1}(x) d x
  \leq \|f\|_{L^\infty(\Omega_T)}|\Omega_T|+\frac{1}{\gamma+1}\|u_0\|_{L^\infty(\Omega)}^{\gamma+1}|\Omega|.
    \end{array}
\end{equation}
Since we know $
    \nabla u_k\cdot\nabla u_k^{\gamma}=\gamma u_k^{\gamma-1}|\nabla u_k|^2,$
and by \cref{algebraic}
\begin{equation*}
    \left(u_k^\gamma(x, \theta)-u_k^\gamma(y, \theta)\right)\left(u_k(x, \theta)-u_k(y, \theta)\right) \geq C(\gamma)\left(u_k^{\frac{\gamma+1}{2}}(x, \theta)-u_k^{\frac{\gamma+1}{2}}(y, \theta)\right)^2,
\end{equation*}
we get taking supremum over $t\in (0,T]$ in \cref{takesup}, that the sequence 
$\left\{u_k^{\frac{\gamma+1}{2}}\right\}_k$ is bounded in $L^{\infty}(0, T, L^2(\Omega)) \cap L^2(0, T ; X_0^s(\Omega))
\cap L^2(0, T ; H_0^1(\Omega))\equiv L^{\infty}(0, T, L^2(\Omega)) 
\cap L^2(0, T ; H_0^1(\Omega))$.\smallskip
\\Now by \cref{approxexxist} and \cref{ubase}, as the sequence $\left\{u_k\right\}_k$ is increasing in $k$, so the pointwise limit $u$ of $\{u_k\}_k$ exists; and satisfies $u(\cdot,0)=u_0(\cdot)$ and $u=0$ in $(\mathbb{R}^n\backslash\Omega)\times(0,T)$. Also since $\left\{u_k^{\frac{\gamma+1}{2}}\right\}_k$ is bounded in $
L^2(0, T ; H_0^1(\Omega))
$, therefore $u_k^{\frac{\gamma+1}{2}}\rightharpoonup u^{\frac{\gamma+1}{2}}$ in $L^2(0,T;H^1_0(\Omega))$. Again by Beppo Levi theorem $u_k\rightarrow u$ in $L^{1}(\Omega_T)$. Using Fatou's Lemma, we have $u^{\frac{\gamma+1}{2}} \in L^{\infty}(0, T ; L^2(\Omega)) 
\cap L^{\infty}(\Omega_T)$. Also, by \cref{ubase} we get $u(x, t) \geq C\left(\omega, t_0,n,s\right) $ in $\omega \times\left[t_0, T\right)$ for any $\omega\subset\subset\Omega$ and $t_0>0$. Using this positivity, it is then easy to show that $u$ is a very weak solution in the sense of \cref{very weak} and is shown in detail in the proof of \cref{mainth2}.\smallskip\\
In order to show that $u\in \mathcal{C}([0,T],L^2(\Omega))$, we fix $k\geq l$ and hence $u_k\geq u_l$ and note that the function $\tilde{u}=(u_k+\frac{1}{k})-(u_l+\frac{1}{l})\in L^2(0,T;H^1(\Omega))$ satisfies the equation 
  \begin{equation*}
    \begin{array}{lcr}
\begin{cases}
\tilde{u}_t-\Delta \tilde{u}+(-\Delta)^s\tilde{u}=\frac{f(x,t)}{(u_k+\frac{1}{k})^\gamma}-\frac{f(x,t)}{(u_l+\frac{1}{l})^\gamma}, 
&  \mbox{ in } \Omega_T
,\\
\tilde{u}=\frac{1}{k}-\frac{1}{l}, &  \mbox{ in } (\mathbb{R}^n\backslash \Omega) \times (0,T),\smallskip\\
\tilde{u}(x,0)=\frac{1}{k}-\frac{1}{l},  & \mbox{ in } \Omega;
\end{cases}
\end{array}
\end{equation*}
Now as $\left(\frac{f}{(u_k+\frac{1}{k})^\gamma}-\frac{f}{(u_l+\frac{1}{l})^\gamma}\right) \leq 0$ in the set $\{\tilde{u} \geq 0\}$, using \cref{Kato,katore}, it holds that
\begin{equation*}
(\tilde{u}_+)_t-\Delta (\tilde{u}_+)+(-\Delta)^s (\tilde{u}_+) \leq 0.    
\end{equation*}
Again we notice that $\tilde{u}_+\in L^2(0,T;H^1_0(\Omega))$, has finite energy and 
$\tilde{u}_+(x,0)= 0$, therefore by comparison principle, it holds that $\tilde{u}_+\equiv 0$, and then we have $(u_k+\frac{1}{k})\leq(u_l+\frac{1}{l})$ for $k\geq l$. Therefore we get $0\leq u_k-u_l\leq \frac{1}{l}-\frac{1}{k}$, for $k\geq l$. This implies that the sequence $\{u_k\}_k$ is Cauchy in $\mathcal{C}([0,T],L^2(\Omega))$ and hence $u\in \mathcal{C}([0,T],L^2(\Omega)) $.
\smallskip\\We now consider  that $\gamma \leq 1$, then using $u_k \chi_{(0, t)}$ as a test function in \cref{approx}, we get 
that
\begin{equation*}
    \begin{array}{l}
\frac{1}{2} \int_{\Omega} u_k^2(x, t) d x+\int_0^t\int_\Omega |\nabla u_k|^2d x d\theta+\frac{1}{2} \int_0^t \int_Q \frac{\left(u_k(x, \theta)-u_k(y, \theta)\right)^2}{|x-y|^{n+2 s}} d x d y d \theta
\smallskip\\\leq \iint_{\Omega_{\mathrm{t}}} f u_k^{1-\gamma} d x d \theta+\frac{1}{2} \int_{\Omega} u_0^2(x) d x.
    \end{array}
\end{equation*}
Since $\left\{u_k\right\}_k$ is uniformly bounded in $L^{\infty}\left(\Omega_T\right)$, and $\gamma \leq 1$, so $
     \iint_{\Omega_t} f u_k^{1-\gamma} d x d \theta \leq C\|f\|_{L^\infty\left(\Omega_T\right)},$ 
 and then we conclude that $\left\{u_k\right\}_k$ is bounded in the space $
L^2(0, T ; H_0^1(\Omega)) \cap L^{\infty}(\Omega_T)$. We note that the same conclusion holds if $\operatorname{Supp}(f) \subset \subset \Omega_T$ taking into consideration that $u_k \geq u_1 \geq C$ in $\operatorname{Supp}(f)$ for each $k$.
\subsection*{\textbf{Proof of \cref{mainth2}}}
We consider here the approximating problem \cref{approx4} and refer \cref{approxexxist} for properties of $u_k$. As $f_k$ and $u_{0k}$ are bounded and nonnegative, therefore 
$u_k\in L^2(0,T; H^1_0(\Omega))$ is bounded and nonnegative, so we can choose the same test function $(\left(u_k(x, \theta)+\varepsilon\right)^\gamma-\varepsilon^\gamma) \chi_{(0, t)}, 0<\varepsilon<1/k,$ in \cref{approx4} also, and let $\varepsilon\to 0$ to get
\begin{equation*}
    \begin{array}{l}
\quad        \frac{1}{\gamma+1} \int_{\Omega} u_k^{\gamma+1}(x, t) d x+\int_0^t\int_\Omega \nabla u_k\cdot\nabla u_k^{\gamma}d x d\theta\smallskip\smallskip\\+\frac{1}{2} \int_0^t \int_Q \frac{\left(u_k^\gamma(x, \theta)-u_k^\gamma(y, \theta)\right)\left(u_k(x, \theta)-u_k(y, \theta)\right)}{|x-y|^{n+2 s}} d x d y d \theta \smallskip\smallskip\\\leq \iint_{\Omega_t} f_k d x d \theta+\frac{1}{\gamma+1} \int_{\Omega} u_{0k}^{\gamma+1}(x) d x\leq\|f\|_{L^1\left(\Omega_T\right)}+\frac{1}{\gamma+1}\left\|u_0\right\|_{L^{\gamma+1}\left(\Omega\right)}.
    \end{array}
\end{equation*}
Similarly like the proof of \cref{mainth1}, it holds that $\left\{u_k^{\frac{\gamma+1}{2}}\right\}_k$ is bounded in $
L^{\infty}(0, T ; L^2(\Omega)) 
\cap  L^2(0, T ; H_0^1(\Omega))$.\\ Now as $\left\{u_k\right\}_k$ is increasing in $k$, we get by Fatou's Lemma and Beppo Levi's Lemma, that the pointwise limit (as mentioned in \cref{ubase}) $u$ satisfies $u^{\frac{\gamma+1}{2}} \in L^{\infty}(0, T; L^2(\Omega)) 
\cap  L^2(0, T ; H_0^1(\Omega))$ and $u_k \uparrow u$ strongly in 
$L^1(\Omega_T)$. Since $u^{\frac{\gamma+1}{2}} \in   L^2(0, T ; H_0^1(\Omega))$ and $\frac{n(1+\gamma)}{n-2}=\frac{2^*(1+\gamma)}{2}$, embedding result on $u^{\frac{\gamma+1}{2}}$ implies $u\in L^2\left(0,T;L^{\left(\frac{n(1+\gamma)}{n-2}\right)}(\Omega)\right)$. Also, we have $u(x, t) \geq$ $C\left(\omega,t_0,n,s\right) $ in $\omega \times\left[t_0, T\right)$ for any $\omega\subset\subset\Omega$ and $t_0>0$. We show that $u$ is a very weak solution to \cref{problem} in the sense of \cref{very weak}. Let us take arbitrary $\phi \in \mathcal{C}_0^{\infty}(\Omega_T)$, we then have
\begin{equation*}
\iint_{\Omega_T}(\left(u_k\right)_t-\Delta u_k+(-\Delta)^s u_k) \phi\,
d x d t=\iint_{\Omega_T} \frac{f_k}{(u_k+\frac{1}{k})^\gamma} \phi d x d t .
\end{equation*}
Now $\phi \in \mathcal{C}_0^{\infty}(\Omega_T)$ implies $ \phi_t, \Delta \phi$ and $(-\Delta)^s\phi$ all are bounded. Then as $u_k \rightarrow u$ strongly in $L^1\left(\Omega_T\right)$, we have 
\begin{equation*}
    \begin{array}{l}
\quad\iint_{\Omega_T}(\left(u_k\right)_t-\Delta u_k+(-\Delta)^s u_k) \phi\, d x d t \smallskip\smallskip\\=\iint_{\Omega_T} u_k(-(\phi)_t-\Delta \phi+(-\Delta)^s \phi)d x d t \rightarrow \iint_{\Omega_T} u(-(\phi)_t-\Delta \phi+(-\Delta)^s \phi) d x d t .
    \end{array}
\end{equation*}
Now since for all $\omega \subset \subset \Omega $ and for all $t_0>0,$ $ u_k(x, t) \geq C\left(\omega, t_0,n,s\right) $ in $\omega \times\left[t_0, T\right)$, we get the positivity of $\{u_k\}_k$ as $u_k(x, t) \geq C$ in $\operatorname{Supp} \phi$ for all $k \geq 1$. Hence, we can use the dominated convergence theorem to obtain 
\begin{equation*}
    \iint_{\Omega_T} \frac{f_k}{(u_k+\frac{1}{k})^\gamma} \phi\, d x d t \rightarrow \iint_{\Omega_T} \frac{f}{u^\gamma} \phi \,d x d t \text { as } k \rightarrow \infty .
\end{equation*}
Thus
\begin{equation*}
    \iint_{\Omega_T}(u_t-\Delta u+(-\Delta)^s u) \phi\, d x d t=\iint_{\Omega_T} \frac{f}{u^\gamma} \phi\, d x d t.
\end{equation*}
\subsection*{\textbf{Proof of \cref{mainth3}}}
Here we consider \cref{approx4} but with suffix $m\in \mathbb{N}$ as
\begin{equation}{\label{approx2}}
\begin{array}{lcr}
\begin{cases}
(u_m)_t-\Delta u_m+(-\Delta)^su_m=\frac{f_m(x,t)}{(u_m+\frac{1}{m})^\gamma} 
 & \mbox{ in } \Omega_T:=\Omega \times (0,T),\\
u_m=0 &  \mbox{ in } (\mathbb{R}^n\backslash \Omega) \times (0,T),\\
u_m(x,0)=u_{0m}(x)  & \mbox{ in } \Omega.
\end{cases}
\end{array}
\end{equation}
The properties of $u_m$'s follow from \cref{approxexxist}. As $u_m\in 
L^2(0,T; H^1_0(\Omega))$ is nonnegative and $T_k(u_m)$ is bounded, so using $\left(\left(T_k\left(u_m\right)+\varepsilon\right)^\gamma-\varepsilon^\gamma\right) \chi_{(0, t)}$ 
as a test function in \cref{approx2}, we get 
\begin{equation*}
    \begin{array}{l}
\quad\int_0^t\int_{\Omega}\left(u_m\right)_t\left((T_k(u_m)+\varepsilon)^\gamma-\varepsilon^\gamma\right) d x d \theta +\int_0^t\int_\Omega \nabla u_k\cdot\nabla \left((T_k(u_m)+\varepsilon)^\gamma-\varepsilon^\gamma\right)d x d \theta\smallskip\\
 +\frac{1}{2} \int_0^t \int_Q \frac{(u_k(x, \theta)-u_k(y, \theta))((T_k(u_m(x, \theta))+\varepsilon)^\gamma-\left(T_k(u_m(y, \theta))+\varepsilon\right)^\gamma)}{|x-y|^{n+2 s}} d x d y d \theta
\leq \iint_{\Omega_T} f(x, \theta) d x d \theta .
    \end{array}
\end{equation*}
We have chosen $0<\varepsilon<1/k$ arbitrarily in above. Now letting $\varepsilon\to 0$, by Fatou's lemma, we get for all $t \leq T$
\begin{equation*}
    \begin{array}{l}
         \quad\int_{\Omega} L_k\left(u_m(x, t)\right) d x +\int_0^t\int_\Omega \nabla u_m\cdot\nabla T_k^{\gamma}(u_m)d x d\theta\\
+\frac{1}{2} \int_0^t \int_Q \frac{\left(T_k^\gamma\left(u_m(x, \theta)\right)-T_k^\gamma\left(u_m(y, \theta)\right)\right)\left(u_m(x, \theta)-u_m(y, \theta)\right)}{|x-y|^{n+2 s}} d x d y d \theta \\
\leq\|f\|_{L^1\left(\Omega_T\right)}+C_3(k)\left\|u_0\right\|_{L^1(\Omega)}+C_4(k)|\Omega|,
    \end{array}
\end{equation*}
where $L_k(\rho)=\int_0^\rho\left(T_k(\xi)\right)^\gamma d \xi$. Notice that
\begin{equation*}
    \begin{array}{l}
\left(T_k^\gamma\left(u_m(x, \theta)\right)-T_k^\gamma\left(u_m(y, \theta)\right)\right)\left(u_m(x, \theta)-u_m(y, \theta)\right)\smallskip\\\geq \left(T_k^\gamma\left(u_m(x, \theta)\right)-T_k^\gamma\left(u_m(y, \theta)\right)\right)\left(T_k\left(u_m(x, \theta)\right)-T_k\left(u_m(y, \theta)\right)\right)   
    \end{array}
\end{equation*}
and so\begin{equation*}
\nabla u_m\cdot\nabla T_k^\gamma(u_m)\geq\gamma T_k^{\gamma-1}(u_m)|\nabla T_k(u_m)|^2    
\end{equation*}
and for $\rho>0$, trivial calculation yields that
\begin{equation*}
    C_1(k) \rho-C_2(k) \leq L_k(\rho) \leq C_3(k) \rho+C_4(k) 
    \text { and } L_k(\rho) \geq C\left(T_k(\rho)\right)^{\gamma+1} ,
\end{equation*}
where $C_1, C_2, C_3, C_4, C$ are positive constants depending only on $k$. Therefore $\left\{\left(T_k\left(u_m\right)\right)^{\frac{\gamma+1}{2}}\right\}_m$ is bounded in $L^{\infty}(0, T ; L^2(\Omega))
\cap L^2(0, T; H_0^1(\Omega))$ and $\left\{u_m\right\}_m$ is bounded in $L^{\infty}(0, T ; L^1(\Omega))$. \smallskip\\Now as $\left\{u_m\right\}_m$ is increasing in $m$, by Fatou's Lemma and Beppo Levi's theorem, the pointwise limit $u$ of $\{u_m\}_m$ (as mentioned in \cref{ubase}) satisfies $\left(T_k(u)\right)^{\frac{\gamma+1}{2}} \in L^{\infty}(0, T ; L^2(\Omega))
\cap L^2(0, T ; H_0^1(\Omega))$, $u_m\uparrow u$ strongly in $L^1(\Omega_T)$. Also $u(\cdot,0)=u_0(\cdot)$ in $L^1$ sense and $u$ is $0$ outside $\Omega$. Then $u$ can be shown to be a very weak solution of \cref{problem} in the sense of \cref{very weak}, and the proof is same as that of \cref{mainth2}.
\begin{remark}
    In view of \cref{mainth1}, if we consider the approximating problem 
\begin{equation*}
\begin{array}{lcr}
\begin{cases}
(u_k)_t-\Delta u_k+(-\Delta)^su_k=\frac{f_k(x,t)}{u_k^\gamma} 
 & \mbox{ in } \Omega_T
 ,\\
u_k=0 &  \mbox{ in } (\mathbb{R}^n\backslash \Omega) \times (0,T),\\
u_k(x,0)=u_{0k}(x)  & \mbox{ in } \Omega;
\end{cases}
\end{array}
\end{equation*}
for $\gamma\leq 1$, then for each $k$, $u_k$ exists and $u_k\in L^2(0,T;H^1_0(\Omega))\cap \mathcal{C}([0,T],L^2(\Omega))$. Also, each $u_k$ is unique (see \cref{uniqueness}). Further, the sequence $\{u_k\}_k$ satisfy the followings:    \\(a) each $u_k$ is bounded and nonnegative,\\
    (b) the sequence $\{u_k\}_k$ is increasing in $k$ (can be shown similarly like \cref{comparison}),\\
    (c) for each $t_0>0$ and $\omega\subset\subset\Omega_T$, $\exists C(\omega,t_0,n,s)$ such that $u_k\geq C(\omega,t_0,n,s)$.\smallskip\\
    Now for $\gamma\leq 1$, we have $(u_k+\varepsilon)^\gamma-\varepsilon^\gamma\leq u_k^\gamma$, and hence $\frac{f_k((u_k+\varepsilon)^\gamma-\varepsilon^\gamma)}{u_k^\gamma}\leq f_k$. Therefore we can 
    follow the same procedure as that of \cref{mainth2} to get that the pointwise limit $u$ of $\{u_k\}_k$ is a very weak solution of \cref{problem} and satisfies the corresponding properties of \cref{mainth2}. Now let $k \geq l$ be fixed, then $u_k \geq u_l$ and
\begin{equation*}
    \left(u_k-u_l\right)_t-\Delta (u_k-u_l)+(-\Delta)^s\left(u_k-u_l\right)=\frac{f_k(x, t)}{u_k^\gamma}-\frac{f_l(x, t)}{u_l^\gamma}\leq \frac{f_k-f_l}{u_k^\gamma}
    \text { in } \Omega_T .
\end{equation*}
Hence using $\left((u_k-u_l+\varepsilon)^\gamma-\varepsilon^\gamma\right)\chi_{(0,t)}$, $0<\varepsilon<<1$ as a test function in the above inequality and letting $\varepsilon\to 0$ it holds that
\begin{equation*}
    \begin{array}{l}
         \frac{1}{\gamma+1} \int_{\Omega}\left(u_k(x, t)-u_l(x, t)\right)^{\gamma+1} d x 
\leq \iint_{\Omega_t}\left(f_k-f_l\right) d x d \theta+\frac{1}{\gamma+1} \int_{\Omega}\left(u_{0 k}(x)-u_{0 l}(x)\right)^{\gamma+1}(x) d x .
    \end{array}
\end{equation*}
Now as by Dominated Convergence Theorem, $f_k \uparrow f$ strongly in $L^1(\Omega_T)$ and $u_{0 k} \uparrow u_0$ strongly in $L^{\gamma+1}(\Omega)$ if $u_0\in L^{\gamma+1}(\Omega)$, we get that the sequence $\left\{u_k\right\}_k$ is a Cauchy sequence in the space $\mathcal{C}([0,T];L^{\gamma+1}(\Omega))$ and hence in $\mathcal{C}([0,T];L^{1}(\Omega))$. Therefore this approximation allows us to have $u\in \mathcal{C}([0,T],L^1(\Omega))$.
\\
Note that the same approximation technique will not work for $\gamma>1$, as in that case we will have $u_k^\frac{\gamma+1}{2}\in 
L^2(0, T ; H_0^1(\Omega))$ which may not give us $u_k\in L^2(0, T ; H_0^1(\Omega))$ and the monotonicity of $\{u_k\}_k$ cannot be proven in similar way like \cref{comparison}. However, we will take approximations as \cref{approx4} for convenience even for $\gamma\leq 1$. 
\smallskip\\
Further, this approximation technique will not work for $u_0\in L^1(\Omega)$, as in that case, we need to take test functions like $\left((T_m\left(u_k-u_l)+\varepsilon\right)^\gamma-\varepsilon^\gamma\right)\chi_{(0, t)}$ and will end up with $\int_{\Omega} L_m\left(u_k(x, t)-u_l(x, t)\right) d x$ in the left which will not give us anything desired.
\end{remark}
\subsection*{\textbf{Proof of \cref{mainth4}}}
We consider the approximated problems \cref{approx4} here and refer \cref{approxexxist}. Since each $u_k\in L^2(0,T; H^1_0(\Omega))$ is bounded and $\gamma >1$ so we can use $u_k^\gamma \chi_{(0, t)}$ 
as a test function in \cref{approx4}, to get that, for all $t\leq T$,
\begin{equation}{\label{energy}}
\begin{array}{l}
 \quad\frac{1}{\gamma+1} \int_{\Omega} u_k^{\gamma+1}(x, t) d x+\int_0^t\int_\Omega \nabla u_k\cdot\nabla u_k^{\gamma}d x d\theta\smallskip\\\quad+\frac{1}{2} \int_0^t \int_Q \frac{\left(u_k^\gamma(x, \theta)-u_k^\gamma(y, \theta)\right)\left(u_k(x, \theta)-u_k(y, \theta)\right)}{|x-y|^{n+2 s}} d x d y d \theta \smallskip\\
\leq \iint_{\Omega_t} f d x d \theta+\frac{1}{\gamma+1} \int_{\Omega} u_{0k}^{\gamma+1}(x) d x \smallskip\\\leq\|f\|_{L^1\left(\Omega_T\right)}+\frac{1}{\gamma+1}\left\|u_0\right\|_{L^{\gamma+1}\left(\Omega\right)}.
\end{array}
\end{equation}
Using item $(i)$ of \cref{algebraic} and taking supremum over $0<t\leq T$, we get that $\left\{u_k^{\frac{\gamma+1}{2}}\right\}_k$ is uniformly bounded in $L^{\infty}(0, T ; L^2(\Omega))
\cap  L^2(0, T ; H_0^1(\Omega))$ and $\left\{u_k\right\}_k$ is uniformly bounded in $L^{\infty}(0, T ; L^{\gamma+1}(\Omega))$.
\smallskip\\We now show that $\left\{u_k\right\}_k$ is uniformly bounded in $L^2
(t_0, T ; H_{l o c}^s(\Omega))\cap 
L
^2(t_0, T ; H_{l o c}^1(\Omega))$ for each $t_0>0$. Since $\gamma>1$, and $\Omega_T$ is bounded, and $\left\{u_k\right\}_k$ is uniformly bounded in $L^{\infty}(0, T ; L^{\gamma+1}(\Omega))$, we deduce that $\left\{u_k\right\}_k$ is uniformly bounded in $L^2(0, T ; L^2(\Omega))=L^2(\Omega_T)$, in particular in $L^2(K \times(t_0, T))$, for every subset $K$ compactly contained in $ \Omega$ and for each $t_0>0
$. Further, as $K\times K \subset \Omega\times\Omega \subset Q$ and all the integrals in the left-hand-side of \cref{energy} are positive, hence we have,
\begin{equation*}
    \begin{array}{l}
         \int_{t_0}^{T} \int_K \int_K \frac{\left(u_k(x, t)-u_k(y, t)\right)\left(u_k^\gamma(x, t)-u_k^\gamma(y, t)\right)}{|x-y|^{n+2 s}} d x d y d t \leq 2\|f\|_{L^1\left(\Omega_T\right)}+\frac{2}{\gamma+1}\left\|u_0\right\|_{L^{\gamma+1}\left(\Omega\right)},
    \end{array}
\end{equation*}
and 
\begin{equation*}
    \int_{t_0}^{T} \int_K  u_k^{\gamma-1}|\nabla u_k|^2 d x d t\leq \frac{1}{\gamma}\left(\|f\|_{L^1\left(\Omega_T\right)}+\frac{1}{\gamma+1}\left\|u_0\right\|_{L^{\gamma+1}\left(\Omega\right)}\right),
\end{equation*}
for every $K \subset\subset \Omega$ and for each $t_0>0$. 
\smallskip\\We now apply the item $(iii)$ of \cref{algebraic}, to get
\begin{equation*}
    \int_{t_0}^{T} \int_K \int_K \frac{\left| u_k(x, t)-u_k(y, t)\right|^2\left|u_k(x, t)+u_k(y, t)\right|^{\gamma-1}}{|x-y|^{n+2 s}} d x d y d t \leq 2 C_\gamma\left(\|f\|_{L^1\left(\Omega_T\right)}+\left\|u_0\right\|_{L^{\gamma+1}(\Omega)}\right) .
\end{equation*}
Using the positivity of $u_k$ in $\omega \times [t_0,T)$ for all $k$, we get
\begin{equation}{\label{est1}}
\int_{t_0}^{T} \int_K \int_K \frac{\left| u_k(x, t)-u_k(y, t)\right|^2}{|x-y|^{n+2 s}} d x d y d t \leq \frac{2^{2-\gamma} C_\gamma}{c_{(K,t_0
)}^{\gamma-1}}\left(\|f\|_{L^1\left(\Omega_T\right)}+\left\|u_0\right\|_{L^{\gamma+1}(\Omega)}\right) ,
\end{equation}
and 
\begin{equation}{\label{est2}}
    \int_{t_0}^{T} \int_K |\nabla u_k|^2d x d y d t\leq \frac{1}{\gamma c_{(K,t_0
    )}^{\gamma-1}}\left(\|f\|_{L^1\left(\Omega_T\right)}+\frac{1}{\gamma+1}\left\|u_0\right\|_{L^{\gamma+1}(\Omega)}\right).
\end{equation}
Hence $\left\{u_k\right\}_k$ is uniformly bounded in $L
^2(t_0, T ; H_{l o c}^s(\Omega))\cap L
^2(t_0, T ; H_{l o c}^1(\Omega))\equiv L
^2(t_0, T ; H_{l o c}^1(\Omega))$ for each $t_0>0$.\smallskip\\Since we have $\left\{u_k^{\frac{\gamma+1}{2}}\right\}$ is uniformly bounded in $
L^2(0, T ; H_0^1(\Omega)) \subset$ $L^2\left(\Omega_T\right)$, this implies that the increasing sequence $\left\{u_k\right\}_k$ is uniformly bounded in $L^1\left(\Omega_T\right)$. Then, there exists a measurable function $u$ such that $u_k\rightarrow u$ a.e. in $\Omega_T$ and by Beppo Levi's theorem  $u_k\rightarrow u$ in $L^1(\Omega_T)$. Since $u_k=0$ in $\left(\mathbb{R}^n \backslash \Omega\right) \times(0, T)$, extending $u$ by zero outside of $\Omega$ we conclude that $u_k \rightarrow u$ a.e. in $\mathbb{R}^n \times(0, T)$ with $u=$ 0 in $\left(\mathbb{R}^n \backslash \Omega\right) \times(0, T)$. Now we use Fatou's lemma in \cref{energy,est1,est2}, to obtain for each $t_0>0$, $u \in 
L
^2(t_0, T ; H_{l o c}^1(\Omega)) \cap L^{\infty}(0, T ; L^{\gamma+1}(\Omega))$ and $u^{\frac{\gamma+1}{2}} \in
L^2(0, T ; H_0^1(\Omega))$. As $u_k(x,0)=u_{0k}(x)$ for each $k$, so $u(x,0)=u_0(x)$ in $L^1$ sense (see \cref{ubase}). The rest of the proof will follow similarly as that of \cref{mainth2}.
\subsection*{\textbf{Proof of \cref{mainth5}}}
Here also, we consider the approximating problems \cref{approx4}. 
Since $u_k\in L^2(0,T;H^1_0(\Omega))$, we take $u_k(x, t) \chi_{(0, \tau)}(t)$ as a test function in \cref{approx4}, to have
\begin{equation}{\label{eqn16}}
\begin{array}{l}
 \sup _{0 \leq \tau \leq T} \int_{\Omega} u_k^2(x, \tau) d x+2\int_0^T\int_\Omega |\nabla u_k|^2d x d t+\int_0^T \int_Q \frac{\left(u_k(x, t)-u_k(y, t)\right)^2}{|x-y|^{n+2 s}} d x d y d t \smallskip\\
 \leq 2 \iint_{\Omega_T} f u_k^{1-\gamma} d x d t+\left\|u_0\right\|_{L^{2}(\Omega)} .
\end{array}
\end{equation}
\textbf{Case 1: $f \in L^{\frac{2}{\gamma+1}}\left(0, T ; L^{\left(\frac{2^*}{1-\gamma}\right)^{\prime}}(\Omega)\right)$}\\For this case, since $f \in L^{\frac{2}{\gamma+1}}\left(0, T ; L^{\left(\frac{2^*}{1-\gamma}\right)^{\prime}}(\Omega)\right)$, we apply the Hölder inequality two times, first for the space integral and then for the time integral, to obtain
\begin{equation*}
    \begin{array}{rcl}
         \iint_{\Omega_T} f u_k^{1-\gamma} d x d t &\leq& \int_0^T\left(\int_{\Omega}|f(x, t)|^{\left(\frac{2^*}{1-\gamma}\right)^{\prime}} d x\right)^{\frac{1}{\left(\frac{2^*}{1-\gamma}\right)^\prime}}\left(\int_{\Omega}\left|u_k(x, t)\right|^{2^*} d x\right)^{\frac{1-\gamma}{2^*}} d t \smallskip\\&=&\int_0^T\|f\|_{L^{\left(\frac{2^*}{1-\gamma}\right)^\prime}{(\Omega)}}\left\|u_k\right\|_{L^{2 ^*}(\Omega)}^{1-\gamma} d t  \smallskip\\&\leq&\left(\int_0^T\|f\|_{L^{\left(\frac{2 ^*}{1-\gamma}\right)^{\prime}}(\Omega)}^{\frac{2}{1+\gamma}} d t\right)^{\frac{1+\gamma}{2}}\left(\int_0^T\left\|u_k\right\|_{L^{2^*}(\Omega)}^2 d t\right)^{\frac{1-\gamma}{2}} .
    \end{array}
\end{equation*}
We now apply the Sobolev embedding as of \cref{Sobolev embedding} in the last term on the right-hand-side to get
\begin{equation*}
    \iint_{\Omega_T} f u_k^{1-\gamma} d x d t  \quad \leq(C(n))^{\frac{1-\gamma}{2}}\|f\|_{L^{\frac{2}{\gamma+1}}\left(0, T ; L^{\left(\frac{2^*}{1-\gamma}\right)^{\prime}}(\Omega)\right)}\left[\int_0^T \int_{\Omega} |\nabla u_k|^2d x d t\right]^{\frac{1-\gamma}{2}}.
\end{equation*} 
Since $\gamma <1$, so we use Young's inequality to deduce from \cref{eqn16} that
\begin{equation*}
\begin{array}{l}
    \sup _{0 \leq \tau \leq T} \int_{\Omega} u_k^2(x, \tau) d x+\int_0^T\int_\Omega |\nabla u_k|^2d x d t+\int_0^T \int_Q \frac{\left(u_k(x, t)-u_k(y, t)\right)^2}{|x-y|^{n+2 s}} d x d y d t \leq C,
    \end{array}
\end{equation*}
where $C$ is a positive constant independent of $k$.\smallskip\\
\textbf{Case 2: $f \in L^{\bar{m}}(\Omega_T)$}\\For this case, we notice that as $f \in L^{\bar{m}}\left(\Omega_T\right)$, we can apply the Hölder inequality in the first term on the right-hand-side in \cref{eqn16} with exponents $\bar{m}$ and $\bar{m}^\prime$, to get
\begin{equation}{\label{eqn18}}
\begin{array}{l}
\sup _{0 \leq \tau \leq T} \int_{\Omega} u_k^2(x, \tau) d x+\int_0^T\int_\Omega |\nabla u_k|^2d x d t+\int_0^T \int_Q \frac{\left|u_k(x, t)-u_k(y, t)\right|^2}{|x-y|^{n+2 s}} d x d y d t \\\leq 2\|f\|_{L^{\bar{m}}\left(\Omega_T\right)}\left[\iint_{\Omega_T} |u_k|^{(1-\gamma) \bar{m}^{\prime}} d x d t\right]^{\frac{1}{\bar{m}^{\prime}}}+\left\|u_0\right\|_{L^{2}(\Omega)} .
\end{array}
\end{equation}
We observe that $(1-\gamma) \bar{m}^{\prime}=\frac{2(n+2 )}{n}$ and hence using the Hölder inequality with the exponents $\frac{n}{n-2 }$ and $\frac{n}{2 }$ and by Sobolev embedding (\cref{Sobolev embedding}), we reach that
\begin{equation*}
    \begin{array}{rcl}
\iint_{\Omega_T}\left|u_k\right|^{\frac{2(n+2 )}{n}} d x d t&=&\iint_{\Omega_T}|u_k|^2\left|u_k\right|^{\frac{4 }{n}} d x d t 
\\& \leq &\int_0^T\left[\int_{\Omega}\left|u_k(x, t)\right|^2 d x\right]^{\frac{2 }{n}}\left\|u_k\right\|_{L^{2 ^*}(\Omega)}^2 d t \\
& \leq& C(n)\left[\sup _{0 \leq t \leq T} \int_{\Omega}\left|u_k(x, t)\right|^2 d x\right]^{\frac{2 }{n}} \int_0^T \int_\Omega |\nabla u_k|^2 d x d t.
    \end{array}
\end{equation*}
So using \cref{eqn18} and convexity argument, we get
\begin{equation*}
\iint_{\Omega_T}\left|u_k\right|^{\frac{2(n+2 )}{n}} d x d t\leq C(n) 2^{\frac{2 }{n}}\left(\left(2\|f\|_{L^{\bar{m}}\left(\Omega_T\right)}\right)^{\frac{n+2 }{n}}\left(\iint_{\Omega_T} |u_k|^{(1-\gamma) {\bar{m}^{\prime}}}\right)^{\frac{n+2 }{n \bar{m}^{\prime}}}+\left(\left\|u_0\right\|_{L^{2}(\Omega)}\right)^{\frac{n+2 }{n}}\right) .
\end{equation*}
Since $\frac{n+2 }{n \bar{m}^{\prime}}=\frac{1-\gamma}{2}<1$, we use the Young inequality to obtain
\begin{equation*}
\iint_{\Omega_T}\left|u_k\right|^{\frac{2(n+2 )}{n}} d x d t \leq C,
\end{equation*}
where $C$ is a positive constant independent of $k$. Therefore, by \cref{eqn18} we deduce that the sequence $\left\{u_k\right\}_k$ is uniformly bounded in the space $L^2(0, T ; X_0^s(\Omega))\cap L^2(0, T ; H_0^1(\Omega)) \cap L^{\infty}(0, T ; L^2(\Omega))\equiv L^2(0, T ; H_0^1(\Omega)) \cap L^{\infty}(0, T ; L^2(\Omega))$.
\smallskip \\Now the rest of the proof follows similarly as that of \cref{mainth2}. However, for the sake of completeness, we include it here in a bit different way. \smallskip\\Since the sequence $\left\{u_k\right\}_k$ is uniformly bounded in the reflexive Banach space $L^2(0, T ; H_0^1(\Omega)) $, there exist a subsequence of $\left\{u_k\right\}_k$, still indexed by $k$, and a measurable function $u \in 
L^2(0, T ; H_0^1(\Omega)) $ such that $u_k \rightharpoonup u$ weakly in $
L^2(0, T ; H_0^1(\Omega)) $ and  $u_k\rightarrow u$ strongly in $L^2\left(\Omega_T\right)$ and a.e. in $\Omega \times(0, T)$. In addition, since $u_k=u=0$ on $\mathcal{C} \Omega \times(0, T)$, extending $u\equiv 0$ outside $\Omega$, we obtain $u_k \rightarrow u$ for a.e. $(x, t) \in \mathbb{R}^n \times(0, T)$. Again since $\left\{u_k\right\}$ is increasing in $k$, and $u_k(x,0)=u_{0k}(x)$, so we have $u(x,0)=u_0(x)$ in $L^1$ sense (\cref{ubase}). Also, by Fatou's Lemma, we get that $u\in L^\infty(0,T;L^2(\Omega))$. Hence it follows that
\begin{equation*}
    \frac{u_k(x, t)-u_k(y, t)}{|x-y|^{\frac{n+2 s}{2}}} \rightarrow \frac{u(x, t)-u(y, t)}{|x-y|^{\frac{n+2 s}{2}}} \text { a.e. in } Q \times(0, T) .
\end{equation*}
We take an arbitrary test function $\varphi \in \mathcal{C}_0^{\infty}(\Omega_T)$ in \cref{approx4} to get
\begin{equation*}
    \begin{array}{l}
\quad       -\iint_{\Omega_T} u_k \varphi_t d x d t +\int_0^T\int_{\Omega} u_k(-\Delta\phi) d x d t\\\quad+\frac{1}{2} \iint_{Q_T} \frac{\left(u_k(x, t)-u_k(y, t)\right)(\varphi(x, t)-\varphi(y, t))}{|x-y|^{n+2 s}} d x d y d t=\iint_{\Omega_T} \frac{f_k \varphi}{\left(u_k+\frac{1}{k}\right)^\gamma} d x d t .
    \end{array}
\end{equation*}
Since $\varphi \in \mathcal{C}_0^{\infty}(\Omega_T)$, therefore $\varphi_t$ and $\nabla \varphi$ both are in $L^2(\Omega_T)$ and since strong convergence implies weak convergence too, so it is clear that 
\begin{equation*}
    \lim _{k \rightarrow \infty}\int_0^T\int_{\Omega} u_k \phi_t d x d t=\int_0^T\int_{\Omega} u\phi_t d x d t 
\end{equation*}
and by weak convergence in $L^2(0,T;H^1_0(\Omega))$, we get
 \begin{equation*}
 \begin{array}{rcl}
   \int_0^T\int_{\Omega} u_k(-\Delta \phi) d x d t&=&\int_0^T\int_{\Omega}\nabla u_k\cdot\nabla \phi d x d t\\&\to&\int_0^T\int_{\Omega}\nabla u\cdot\nabla \phi d x d t =\int_0^T\int_{\Omega} u(-\Delta \phi) d x d t,\text{ as } k\to \infty .
   \end{array}
 \end{equation*}
We now define
\begin{equation*}
    F_k(x, y, t)=\frac{u_k(x, t)-u_k(y, t)}{|x-y|^{\frac{n+2 s}{2}}} \text { and } F(x, y, t)=\frac{u(x, t)-u(y, t)}{|x-y|^{\frac{n+2 s}{2}}} \text {. }
\end{equation*}
Then as $F_k \rightarrow F$ a.e. in $Q \times(0, T)$ and $\left\{u_k\right\}_k$ is uniformly bounded in $L^2\left(0, T ; X_0^s(\Omega)\right)$, by weak convergence we reach that
\begin{equation*}
\begin{array}{l}
    \lim _{k \rightarrow \infty} \iint_{Q_T} \frac{\left(u_k(x, t)-u_k(y, t)\right)(\varphi(x, t)-\varphi(y, t))}{|x-y|^{n+2 s}} d x d y d t \smallskip\\=\iint_{Q_T} \frac{(u(x,t)-u(y,t))(\varphi(x, t)-\varphi(y, t))}{|x-y|^{n+2 s}} d x d y d t, 
    \end{array}
\end{equation*}
for all $\varphi \in \mathcal{C}_0^{\infty}(\Omega_T)$. Now since for all $\omega \subset \subset \Omega $ and for all $t_0>0,$ we have $ u_k(x, t) \geq C\left(\omega, t_0,n,s\right) $ in $\omega \times\left[t_0, T\right)$, so we reach that $u_k(x, t) \geq C$ in $\operatorname{Supp} \phi$ (as support of $\phi$ is compact in $\Omega_T$) for all $k \geq 1$. Using this fact, we then have
\begin{equation*}
    0 \leq\left|\frac{f_k \varphi}{\left(u_k+\frac{1}{k}\right)^\gamma}\right| \leq \frac{\|\varphi\|_{L^{\infty}\left(\Omega_T\right)}|f|}{C^\gamma} \in L^1\left(\Omega_T\right).
\end{equation*}
So, by the dominated convergence theorem, we get
\begin{equation*}
\lim _{k \rightarrow \infty} \iint_{\Omega_T} \frac{f_k \varphi}{\left(u_k+\frac{1}{k}\right)^\gamma} d x d t=\iint_{\Omega_T} \frac{f \varphi}{u^\gamma} d x d t.
\end{equation*}
Finally, we pass to the limit as $k \rightarrow\infty$ to get
\begin{equation*}
\begin{array}{l}
    \quad-\iint_{\Omega_T} u(x,t) \varphi_t(x,t) d x d t-\int_0^T\int_{\Omega} u(x,t)\Delta \phi(x,t) d x d t\smallskip \\\quad+\frac{1}{2} \iint_{Q_T} \frac{(u(x, t)-u(y, t))(\varphi(x, t)-\varphi(y, t))}{|x-y|^{n+2 s}} d x d y d t  =\iint_{\Omega_T} \frac{f(x,t) \varphi(x,t)}{u^\gamma(x,t)} d x d t,
     \end{array}
\end{equation*}
for all $\phi \in \mathcal{C}_0^\infty(\Omega_T)$. 
Therefore, $u$ is a weak solution to \cref{problem}.
\subsection*{\textbf{Proof of \cref{mainth6}}}
As $u_k(\geq0)\in L^\infty(\Omega_T)\cap L^2(0,T; H^1_0(\Omega))$, so $(\left(u_k(x, t)+\varepsilon\right)^\theta-\varepsilon^\theta)\in L^2(0,T; H^1_0(\Omega))$, for any $\varepsilon,\theta>0$. 
So choosing $\gamma \leq \theta<1,0<\varepsilon<1/k
$, we take the test function $(\left(u_k(x, t)+\varepsilon\right)^\theta-\varepsilon^\theta) \chi_{(0, \tau)}(t)$ in \cref{approx4}, to get for each $\tau \leq T$
\begin{equation*}
    \begin{array}{l}
       \quad \int_0^\tau \int_{\Omega}\left(u_k\right)_t\left((u_k+\varepsilon))^\theta-\varepsilon^\theta\right) d x d t +\int_0^\tau\int_\Omega \nabla u_k\cdot\nabla \left((u_k+\varepsilon)^\theta-\varepsilon^\theta\right)d x d t\smallskip\\
 +\frac{1}{2} \int_0^\tau \int_Q \frac{(u_k(x, t)-u_k(y, t))((u_k(x, t)+\varepsilon)^\theta-\left(u_k(y, t)+\varepsilon\right)^\theta)}{|x-y|^{n+2 s}} d x d y d t\smallskip \\
\leq \iint_{\Omega_T} f(x, t)\left(u_k(x, t)+\varepsilon\right)^{\theta-\gamma} d x d t .
    \end{array}
\end{equation*}
Letting $\varepsilon\rightarrow 0$, integrating the first term, and taking the supremum over $\tau \in[0, T]$,  we get
\begin{equation}{\label{eqn20}}
\begin{array}{c}
\quad\frac{2}{\theta+1} \sup _{0 \leq \tau \leq T} \int_{\Omega}\left|u_k(x, \tau)\right|^{\theta+1} d x +2\int_0^T\int_\Omega \nabla u_k\cdot\nabla u_k^\theta d x d t\smallskip\\+\int_0^T \int_Q \frac{\left(u_k(x, t)-u_k(y, t)\right)\left(u_k^\theta(x, t)-u_k^\theta(y, t)\right)}{|x-y|^{n+2 s}} d x d y d t 
\leq 2 \iint_{\Omega_T} f u_k^{\theta-\gamma} d x d t+\frac{2}{\theta+1}\left\|u_0\right\|_{L^{2}(\Omega)} .
\end{array}
\end{equation}
Since all terms on the left are positive, taking supremum was allowed. Then by item $i)$ of \cref{algebraic}, we get
\begin{equation*}{\label{eqn211}}
\begin{array}{c}
 \quad\frac{\theta+1}{2 \theta} \sup _{0 \leq \tau \leq T} \int_{\Omega}\left|u_k(x, \tau)\right|^{\theta+1} d x+\frac{(\theta+1)^2}{2}\int_0^T\int_\Omega u_k^{\theta-1}|\nabla u_k|^2d x d t\smallskip\\+\int_0^T \int_Q \frac{|u_k^{\frac{\theta+1}{2}}(x, t)-u_k^{\frac{\theta+1}{2}}(y, t)|^2}{|x-y|^{n+2 s}} d x d y d t 
 \leq \frac{2}{\theta}\left(\iint_{\Omega_T} f u_k^{\theta-\gamma} d x d t+\left\|u_0\right\|_{L^{2}(\Omega)}\right).
\end{array}
\end{equation*}
In the previous inequality, we have used the fact that $\theta+1<2$, now we observe that the term $\frac{\theta+1}{2 \theta}>1$ in the left-hand-side can be dropped.
So we get 
\begin{equation}{\label{eqn21}}
\begin{array}{c}
 \sup _{0 \leq \tau \leq T} \int_{\Omega}\left|u_k(x, \tau)\right|^{\theta+1} d x+\frac{(\theta+1)^2}{4}\int_0^\tau\int_\Omega u_k^{\theta-1}|\nabla u_k|^2d x d t\smallskip\\+\int_0^T \int_Q \frac{|u_k^{\frac{\theta+1}{2}}(x, t)-u_k^{\frac{\theta+1}{2}}(y, t)|^2}{|x-y|^{n+2 s}} d x d y d t \leq \frac{2}{\theta}\left(\iint_{\Omega_T} f u_k^{\theta-\gamma} d x d t+\left\|u_0\right\|_{L^{2}(\Omega)}\right).
\end{array}
\end{equation}
Again, using the same tool like \cref{mainth5} and the Hölder inequality with exponent $\frac{n}{n-2 }$ and $\frac{n}{2 }$ we get
\begin{equation*}
    \begin{array}{rcl}
\iint_{\Omega_T}\left|u_k\right|^{\frac{(\theta+1)(n+2 )}{n}} d x d t&= &\iint_{\Omega _T}\left|u_k\right|^{2\frac{\theta+1}{2}}\left|u_k\right|^{\frac{4 }{n} \frac{\theta+1}{2}} d x d t \\
& \leq &\int_0^T\left(\int_{\Omega}\left|u_k(x, t)\right|^{2^* \frac{\theta+1}{2}} d x\right)^{\frac{n-2}{n}}\left(\int_{\Omega}\left|u_k(x, t)\right|^{\theta+1} d x\right)^{\frac{2 }{n}} d t \\
& =&\int_0^T\left\|u_k^{\frac{\theta+1}{2}}\right\|_{L^{2 ^*}(\Omega)}^2\left(\int_{\Omega}\left|u_k(x, t)\right|^{\theta+1} d x\right)^{\frac{2 }{n}} d t .
\end{array}
\end{equation*}
We now take the supremum over $t\in [0, T]$, and apply the Sobolev embedding as that of \cref{Sobolev embedding}, to get that
\begin{equation*}
\iint_{\Omega_T}\left|u_k\right|^{\frac{(\theta+1)(n+2 )}{n}} d x d t  \leq C(n)\left(\sup _{0 \leq t \leq T} \int_{\Omega}\left|u_k(x, t)\right|^{\theta+1} d x\right)^{\frac{2 }{n}} \iint_{\Omega_T} |\nabla u_k^{\frac{\theta+1}{2}}|^2 d x d t .
\end{equation*}
Using \cref{eqn21} and convexity argument, we now get
\begin{equation}{\label{eqn22}}
\begin{array}{rcl}
\iint_{\Omega_T}\left|u_k\right| ^{\frac{(\theta+1)(n+2 )}{n}} d x d t &
 \leq &C(n)\left(\frac{2}{\theta}\right)^{\frac{n+2 }{n}}\left(\iint_{\Omega_T} f u_k^{\theta-\gamma} d x d t+\left\|u_0\right\|_{L^{2}(\Omega)}\right)^{\frac{n+2 }{n}}\smallskip \\
&\leq& C(n) 2^{\frac{2 }{n}}\left(\frac{2}{\theta}\right)^{\frac{n+2 }{n}}\left(\left(\iint_{\Omega_T} f u_k^{\theta-\gamma} d x d t\right)^{\frac{n+2 }{n}}+\left\|u_0\right\|_{L^{2}(\Omega)}^{\frac{ (n+2)}{n}}\right) .
\end{array}
\end{equation}
\smallskip\textbf{Case 1: $m=1$}\smallskip\\Now for the case $m=1$, we take $\theta=\gamma$ in the \cref{eqn22}. So we obtain
\begin{equation}{\label{eqn23}}
\iint_{\Omega_T}\left|u_k\right| ^{\frac{(\gamma+1)(n+2 )}{n}} d x d t \leq C(n) 2^{\frac{2 }{n}}\left(\frac{2}{\gamma}\right)^{\frac{n+2 }{n}}\left(\|f\|_{L^1(\Omega_T)}^{\frac{n+2 }{n}}+\left\|u_0\right\|_{L^{2}(\Omega)}^{\frac{(n+2 )}{n}}\right) .
\end{equation}
Also, by \cref{eqn21} we easily have
\begin{equation}{\label{eqn24}}
\left\|u_k\right\|_{L^{\infty}(0, T ; L^{\gamma+1}(\Omega))} \leq \frac{ 2}{\gamma}\left[\|f\|_{L^1\left(\Omega_T\right)}+\left\|u_0\right\|_{L^{2}(\Omega)}\right].
\end{equation}
\textbf{Case 2: $m>1$}\smallskip\\For $m>1$, we take $\gamma<\theta<1$. Using \cref{eqn22} and applying Hölder inequality with exponent $m$ and $m^\prime$, we have
\begin{equation*}
    \iint_{\Omega_T} |u_k|^{\frac{(\theta+1)(n+2 )}{n}} d x d t \leq C_1\|f\|_{L^m\left(\Omega_T\right)}^{\frac{n+2 }{n}}\left[\iint_{\Omega_T} |u_k|^{m^{\prime}(\theta-\gamma)} d x d t\right]^{\frac{n+2 }{n m^{\prime}}}+C_1\left\|u_0\right\|_{L^{2}(\Omega)}^{\frac{(n+2 )}{n}},
\end{equation*}
where $C_1=C(n) 2^{\frac{2 }{n}}\left(\frac{2}{\theta}\right)^{\frac{n+2 }{n}}$. We now choose $\gamma<\theta<1$ to be such that
\begin{equation*}
\frac{(\theta+1)(n+2 )}{n}=m^{\prime}(\theta-\gamma),
\text{ i.e. }
\theta=\frac{(n+2 )(m-1)+n m \gamma}{n-2 (m-1)} .    
\end{equation*}
We note that the condition $\theta<1$ is equivalent to $m<\bar{m}$; while $\gamma<\theta$ is always fulfilled. Since $\frac{n+2 }{n m^{\prime}}<1$, applying Young's inequality with $\varepsilon>0$, we get
\begin{equation*}
\smallskip\iint_{\Omega_T} |u_k|^{ \frac{(\theta+1)(n+2 )}{n}} d x d t \leq C_1\|f\|_{L^m\left(\Omega_T\right)}^{\frac{n+2 }{n}}\left(\varepsilon \iint_{\Omega_T} |u_k|^{\frac{(\theta+1)(n+2 )}{n}} d x d t+C(\varepsilon)\right) +C_1\left\|u_0\right\|_{L^{2}(\Omega)}^{\frac{(n+2 )}{n}}.
\end{equation*}
We choose $\varepsilon$ small enough such that $\varepsilon C_1\|f\|_{L^m\left(\Omega_T\right)}^{\frac{n+2 }{n}}=\frac{1}{2}
$ and using the fact that
\begin{equation*}
    \sigma:=\frac{m(\gamma+1)(n+2 )}{n-2 (m-1)}=\frac{(\theta+1)(n+2 )}{n}=m^\prime(\theta-\gamma),
\end{equation*}
we get
\begin{equation}{\label{equa24}}
\iint_{\Omega_T}\left|u_k\right|^\sigma d x d t \leq C,
\end{equation}
where $C$ is a positive constant independent of $k$. 
Now in \cref{eqn21} we use H\"older inequality, to get
\begin{equation*}
    \sup _{0 \leq \tau \leq T} \int_{\Omega} u_k^{\theta+1}(x, \tau) d x \leq \frac{2}{\theta}\left(\|f\|_{L^m\left(\Omega_T\right)}\left\|u_k\right\|_{L^\sigma\left(\Omega_T\right)}^{\frac{\sigma}{m^{\prime}}}+\left\|u_0\right\|_{L^{2}(\Omega)}\right) .
\end{equation*}
Since $\gamma<\theta$ and by \cref{equa24} we conclude that the sequence $\left\{u_k\right\}$ is uniformly bounded in $L^{\infty}(0, T ; L^{\gamma+1}(\Omega))$, the same thing holds for the case $m=1$ by \cref{eqn24}. Finally, by \cref{eqn23} and \cref{equa24} we conclude that in both cases, that is $1 \leq m<\bar{m}$, the sequence $\left\{u_k\right\}$ is uniformly bounded in $L^\sigma(\Omega_T)$, $\sigma:=\frac{m(\gamma+1)(n+2 )}{n-2 (m-1)}$.
\smallskip\\
Now from \cref{eqn21}, again using H\"older inequality, we estimate as
\begin{equation}{\label{eqn25}}
     \iint_{\Omega_T} |\nabla u_k^{\frac{\theta+1}{2}}|^2 d x d t \leq \frac{2}{\theta}\left(\|f\|_{L^m\left(\Omega_T\right)}\left\|u_k\right\|_{L^\sigma\left(\Omega_T\right)}^{\frac{\sigma}{m^{\prime}}}+\left\|u_0\right\|_{L^{2}(\Omega)}\right) \leq C,
\end{equation}
where $C>0$ is a constant independent of $k$. Let $1<\bar{q}<2$ will be specified later. By H\"older inequality, we have
\begin{equation}{\label{3.15}}
    \begin{array}{rcl}
         \iint_{\Omega_T} |\nabla u_k|^{\bar{q}} d x d t= \iint_{\Omega_T} \frac{|\nabla u_k|^{\bar{q}}u_k^{\theta-1}}{u_k^{\theta-1}} d x d t&\leq& \left(\iint_{\Omega_T} |\nabla u_k|^{2}u_k^{\theta-1} d x d t\right)^{\frac{\bar{q}}{2}}\times\left(\iint_{\Omega_T} \frac{u_k^{\theta-1}}{u_k^{\frac{2(\theta-1)}{2-\bar{q}}}} d x d t\right)^{\frac{2-\bar{q}}{2}}\\& \stackrel{\cref{eqn25}}{\leq}& C\left(\iint_{\Omega_T} u_k^{\frac{\bar{q}(1-\theta)}{2-\bar{q}}} d x d t\right)^{\frac{2-\bar{q}}{2}}.
    \end{array}
\end{equation}
We now choose $\bar{q}$ to be such that
\begin{equation*}
\frac{\bar{q}(1-\theta)}{2-\bar{q}}=\sigma=\frac{m(\gamma+1)(n+2)}{n-2 (m-1)}
\text{ i.e. } \bar{q}=\frac{m(\gamma+1)(n+2 )}{n+2 -m(1-\gamma)}.
\end{equation*}
\smallskip We note that $\bar{q}<2$ is equivalent to $m<\bar{m}$; while $\bar{q}>1$ is always fulfilled. Then we get by embedding results and using \cref{equa24,3.15}
\begin{equation*}
\begin{array}{rcl}
    \int_0^T \int_{\Omega} \int_{\Omega} \frac{\left|u_k(x, t)-u_k(y, t)\right|^{\bar{q}}}{|x-y|^{n+\bar{q} s}} d y d x d t \leq C(n,s)\iint_{\Omega_T} |\nabla u_k|^{\bar{q}} d x d t\leq C_2\left(\iint_{\Omega_T} u_k^\sigma(x, t) d x d t\right)^{\frac{2-\bar{q}}{2}} \leq C_3,
    \end{array}
\end{equation*}
where $C_3$ is a positive constant independent of $k$. Thus, $\left\{u_k\right\}_k$ is uniformly bounded in $L^{\bar{q}}(0, T ; W_0^{s, \bar{q}}(\Omega))\cap L^{\bar{q}}(0, T ; W_0^{1, \bar{q}}(\Omega))\equiv  L^{\bar{q}}(0, T ; W_0^{1, \bar{q}}(\Omega))$.\smallskip\\
Since the sequence $\left\{u_k\right\}_k$ is uniformly bounded in the reflexive Banach space $
L^{\bar{q}}(0, T ; W_0^{1, \bar{q}}(\Omega))$, there exist a subsequence of $\left\{u_k\right\}_k$ still indexed by $k$ and a measurable function $u \in 
L^{\bar{q}}(0, T ; W_0^{1, \bar{q}}(\Omega))$ such that $u_k \rightharpoonup u$ weakly in $
L^{\bar{q}}(0, T ; W_0^{1, \bar{q}}(\Omega))$. Also by Fatou's lemma, we will get $u\in 
L^\infty(0,T;L^{\gamma+1}(\Omega))$ and $u\in L^\sigma (\Omega_T)$, with $\sigma:=\frac{m(\gamma+1)(n+2)}{n-2 (m-1)}$. As before, using the monotonicity of the sequence $\left\{u_k\right\}$, we get using Beppo Levi's theorem that $u_k\rightarrow u$ strongly in $L^1(\Omega_T)$ and $u_k\rightarrow u$ a.e in $\mathbb{R}^n\times(0,T)$. By \cref{ubase}, this pointwise limit will satisfy $u(\cdot,0)=u_0(\cdot)$ in $L^1$ sense. Then, the rest of the proof will follow from the proof of \cref{mainth2}.
\begin{remark}
    We observe that the sequence $\left\{u_k\right\}_k$ is uniformly bounded in $L^r(\Omega)$ for every $1 \leq r \leq \sigma$, then following the same lines of the previous proof, we can show that the sequence $\left\{u_k\right\}_k$ is uniformly bounded in $L^q(0, T ; W_0^{s, q}(\Omega))\cap L^q(0, T ; W_0^{1, q}(\Omega))\equiv L^q(0, T ; W_0^{1, q}(\Omega))$ for all $1<q \leq \bar{q}$ where $1 \leq m<\bar{m}$.
\end{remark}
\subsection*{\textbf{Proof of \cref{mainth5.5}}}
Let us introduce the following notations: for any measurable function $v$ we define
\begin{equation*}
\begin{array}{c}
A_m(v)=\left\{(x, t) \in \Omega_T: v(x, t)>m\right\},
\text{ and for a.e. }t \in(0, T), \quad A_m^t(v)=\{x \in \Omega: v(x, t)>m\}.
\end{array}
\end{equation*}We use the idea of the classical proof by D.G. Aronson and J. Serrin, which is to prove a uniform $L^{\infty}$ bound for $u_k$ in $\Omega \times(0, \tau)$, for a positive (small) $\tau$ (to be specified later), and then to iterate such an estimate. We consider the approximated problem \cref{approx4} and take $(u_k-m)_+\chi_{(0,t)}\in L^2(0,T;H^1_0(\Omega))$, for $m>0,t\in(0,\tau),\tau<T$,
as a test function to obtain
\begin{equation}{\label{testfunc}}
\begin{array}{c}
\int_{\Omega} \varphi_m\left(u_k(x, t)\right) d x+\int_0^t\int_\Omega \nabla u_k\cdot\nabla (u_k-m)_+ d x d \theta\smallskip\\+\frac{1}{2} \int_0^t \int_Q \frac{\left(u_k(x, \theta)-u_k(y, \theta)\right)\left((u_k-m)_+(x, \theta)-(u_k-m)_+(y, \theta)\right)}{|x-y|^{n+2 s}} d x d y d \theta \\
\leq\int_0^\tau \int_{\Omega} \frac{f_k (u_k-m)_+}{(u_k+\frac{1}{k})^\gamma} d x d \theta+\int_{\Omega} \varphi_m\left(u_k(x, 0)\right) d x,
\end{array}
\end{equation}
where $\varphi_m(\rho)=\int_0^\rho (\sigma-m)_+
d \sigma= \frac{(\rho-m)_+^2}{2}$. We choose a $m>0$ large enough such that $m >\operatorname{max}\{\left\|u_ 0\right\|_{L^{\infty}(\Omega)},1\}$, in order to neglect the last term above. Noting that $ \nabla u_k\cdot\nabla (u_k-m)_+=|\nabla (u_k-m)_+|^2$, and \begin{equation*}\left(u_k(x, \theta)-u_k(y, \theta)\right)\left((u_k-m)_+(x, \theta)-(u_k-m)_+(y, \theta)\right)\geq \left((u_k-m)_+(x, \theta)-(u_k-m)_+(y, \theta)\right)^2,\end{equation*} we take supremum over $t\in(0,\tau]$ in \cref{testfunc}, to get
\begin{equation}{\label{iter}}
\begin{array}{c}
\|(u_k-m)_+\|^2_{L^\infty(0,\tau;L^2(\Omega))}+\|(u_k-m)_+\|^2_{L^2(0,\tau;X^s_0(\Omega))}+\|(u_k-m)_+\|^2_{L^2(0,\tau;H^1_0(\Omega))} 
\smallskip\\
\leq \int_0^\tau \int_{A^t_{m,k}} f (u_k-m)_+ d x d t.
\end{array}
\end{equation}
Note that, in order to deal with the singularity, we have used the fact that $m\geq 1$ and hence $u_k> 1$ in $A^t_{m,k}$, here the subscript $k$ in $A^t_{m,k}$ denotes that we are considering the function $u_k$.
Now the term of the right-hand side above can be estimated as follows,
\begin{equation}{\label{rhs}}
    \int_0^\tau \int_{A_{m,k}^t} f(u_k-m)_+ d x d t \leq \int_0^\tau \int_{A_{m,k}^t} f (u_k-m)_+ ^2 d x d t+\int_0^\tau \int_{A_{m,k}^t} f d x d t .
\end{equation}
We now study each member present in the right-hand side of \cref{rhs}. We first define the followings,
\begin{equation*}
\bar{r}=2 r^{\prime}, \bar{q}=2 q^{\prime}, \eta=\frac{2 \eta_1}{n},\hat{r}=\bar{r}(1+\eta), \hat{q}=\bar{q}(1+\eta),
\end{equation*}
where $\eta_1=1-\frac{1}{r}-\frac{n}{2 q}$. Note that $\eta_1\in(0,1)$ as $0<\frac{1}{r}+\frac{n}{2q}<1$. Further, simple calculation yields that $\frac{1}{\hat{r}}+\frac{n}{2 \hat{q} }=\frac{n}{4 }$. Now, applying H\"older inequality repeatedly, we estimate the first term as
\begin{equation*}
\begin{array}{rcl}
\int_0^\tau \int_{A_{m,k}^t} f (u_k-m)_+^2 d x d t &\leq& \int_0^\tau \left(\int_{A_{m,k}^t}| f (x,t)|^q d x\right )^{\frac{1}{q}}\left(\int_{A_{m,k}^t}(u_k-m)_+^{2q^\prime} d x\right)^{\frac{1}{q^\prime}} d t  \smallskip\\
&\leq& \left(\int_0^\tau \left(\int_{A_{m,k}^t}| f (x,t)|^q d x\right )^{\frac{r}{q}}d t\right)^{\frac{1}{r}}\left(\int_0^\tau\left(\int_{A_{m,k}^t}(u_k-m)_+^{2q^\prime} d x\right)^{\frac{r^\prime}{q^\prime}} d t \right)^{\frac{1}{r^\prime}}.
\end{array}
\end{equation*}
As again by H\"older inequality with exponent $(1+\eta)$, we have
\begin{equation*}
    \begin{array}{rcl}
        \left(\int_0^\tau\left(\int_{A_{m,k}^t}(u_k-m)_+^{2q^\prime} d x\right)^{\frac{r^\prime}{q^\prime}} d t \right)^{\frac{1}{r^\prime}}&\leq &\left(\int_0^\tau\left(\int_{A_{m,k}^t}(u_k-m)_+^{2q^\prime(1+\eta)}d x\right)^{\frac{r^\prime}{q^\prime(1+\eta)}}|A^t_{m,k}|^{\frac{r^\prime\eta}{q^\prime(1+\eta)}} d t \right)^{\frac{1}{r^\prime}}\\&=&\left(\int_0^\tau\left(\int_{A_{m,k}^t}(u_k-m)_+^{\hat{q}}d x\right)^{\frac{2r^\prime}{\hat{q}}}|A^t_{m,k}|^{\frac{2r^\prime\eta}{\hat{q}}} d t \right)^{\frac{1}{r^\prime}}\\&\leq & \left(\int_0^\tau\left(\int_{A_{m,k}^t}(u_k-m)_+^{\hat{q}}d x\right)^{\frac{\hat{r}}{\hat{q}}}d t \right)^{\frac{1}{r^\prime(1+\eta)}}\left(\int_0^\tau |A^t_{m,k}|^{\frac{\hat{r}}{\hat{q}}}\right)^{\frac{\eta}{r^\prime(1+\eta)}} ,
    \end{array}
\end{equation*}
therefore
\begin{equation*}
    \int_0^\tau \int_{A_{m,k}^t} f (u_k-m)_+^2 d x d t \leq \|f\|_{L^r(0,T;L^q(\Omega))}\|(u_k-m)_+\|^2_{L^{\hat{r}}(0,\tau;L^{\hat{q}}(\Omega))}\mu_k(m)^{\frac{2 \eta}{\hat{r}}},
\end{equation*}
where $\mu_k(m)=\int_0^\tau\left|A_{m,k}^t\right|^{\frac{\hat{r}}{\hat{q}}} d t$. We now use Gagliardo-Nirenberg inequality (\cref{gagliardo}) for $(u_k-m)_+$ to get
\begin{equation}{\label{first}}
    \begin{array}{rcl}
\int_0^\tau \int_{A_{m,k}^t} f (u_k-m)_+^2 d x d t &\leq&  c\|f\|_{L^r(0,T;L^q(\Omega))} \mu_k(m)^{\frac{2 \eta}{\hat{r}}}\left(\int_0^\tau\left\|(u_k-m)_+\right\|_{L^2(\Omega)}^{(1-\theta) \hat{r}}\left\|\nabla (u_k-m)_+\right\|_{L^2\left(\Omega\right)}^{\hat{r} \theta} d t\right)^{\frac{2}{\hat{r}}} \\&
\leq& c\|f\|
\mu_k(m)^{\frac{2 \eta}{\hat{r}}}\bigg[\left\|(u_k-m)_+\right\|_{L^{\infty}\left(0, \tau ; L^2(\Omega)\right)}^2
+\left\| (u_k-m)_+\right\|_{L^2(0, \tau ; H^1_0(\Omega))}^2\bigg].
\end{array}
\end{equation}
where $\hat{r}\theta=2$, and $c$ is a constant independent of the choice of $k$ and $m$. Note that we have used the fact that any function in $H^1_0(\Omega)$ can be considered as a function in $H^1(\mathbb{R}^n)$. Also, we applied Young's inequality with exponents $\frac{\hat{r}}{2}$ and its conjugate.\\
On the other hand, the second term on the right-hand side in \cref{rhs} can be estimated by Hölder inequality as
\begin{equation}{\label{second}}
\begin{array}{rcl}
     \int_0^\tau \int_{A_{m,k}^t} f d x d t &\leq&\int_0^\tau \left(\int_{A_{m,k}^t} |f (x,t)|^qd x\right)^{\frac{1}{q}} |A^t_{m,k}|^{\frac{1}{q^\prime}}d t\smallskip \\&\leq&\left(\int_0^\tau \left(\int_{A_{m,k}^t} |f (x,t)|^qd x\right)^{\frac{r}{q}}dt\right)^{\frac{1}{r}} \left(\int_0^\tau|A^t_{m,k}|^{\frac{r^\prime}{q^\prime}}d t\right)^{\frac{1}{r^\prime}}\smallskip\\ &\leq&\|f\|_{L^r(0,T;L^q(\Omega))}\mu_k(m)^{\frac{2(1+\eta)}{\hat{r}}} .
\end{array}  
\end{equation}
Denoting
\begin{equation*}
||| (u_k-m)_+|||^2=\| (u_k-m)_+\|_{L^{\infty}(0, \tau ; L^2(\Omega))}^2+\|(u_k-m)_+ \|_{L^2(0, \tau ; H^1_0(\Omega))}^2,    
\end{equation*}
and using \cref{first}, \cref{second}, we get from \cref{iter},
\begin{equation*}
||| (u_k-m)_+|||^2\leq c\|f\|_{L^r(0,T;L^q(\Omega))}\left[\mu_k(m)^{\frac{2 \eta}{\hat{r}}}||| (u_k-m)_+|||^2+\mu_k(m)^{\frac{2(1+\eta)}{\hat{r}}}\right],
\end{equation*}
where $c$ is a constant which does not depend on $k$ and $m$. Note that $\mu_k(m) \leq \tau|\Omega|^{\frac{\hat{r}}{\hat{q}}}$, for all $k\in\mathbb{N}$ and $m$, so that we can fix $\tau$, independent of $u_k$ and $m$, suitable small in such a way that
$c\mu_k(m)^{\frac{2 \eta}{\hat{r}}}\|f\|_{L^r(0,T;L^q(\Omega))}=\frac{1}{2}
$ and use again the Gagliardo-Nirenberg inequality (see \cref{first}), to deduce that
\begin{equation}{\label{finalboun}}
\left\|(u_k-m)_+\right\|_{L^{\hat{r}}\left(0, \tau ; L^{\hat{q}}(\Omega)\right)}^2 \leq c||| (u_k-m)_+|||^2\leq c\|f\|_{L^r(0,T;L^q(\Omega))} \mu_k(m)^{\frac{2(1+\eta)}{\hat{r}}} .  
\end{equation}
Consider $l>m>\operatorname{max}\{\left\|u_ 0\right\|_{L^{\infty}(\Omega)},1\}:=m_0$. Then $A^t_{l,k}\subset A^t_{m,k}$, and using \cref{finalboun}, we get
\begin{equation*}
\begin{array}{rcl}
(l-m)\mu_k(l)^{\frac{1}{\hat{r}}}
=\left(\int_0^\tau\left((l-m)^{\hat{q}}\left|A_{l,k}^t\right|\right)^{\frac{\hat{r}}{\hat{q}}} d t\right)^{\frac{1}{\hat{r}}}&\leq& \left(\int_0^\tau\left(\int_{A_{l,k}^t}(u_k-m)_+^{\hat{q}}d x\right)^{\frac{\hat{r}}{\hat{q}}}d t \right)^{\frac{1}{\hat{r}}}\\&\leq &\left(\int_0^\tau\left(\int_{A_{m,k}^t}(u_k-m)_+^{\hat{q}}d x\right)^{\frac{\hat{r}}{\hat{q}}}d t \right)^{\frac{1}{\hat{r}}}
\leq c\mu_k(m)^{\frac{1+\eta}{\hat{r}}} .    
\end{array}
\end{equation*}
Therefore for all $l>m>m_0$, we have
\begin{equation*}
   \mu_k(l) \leq\frac{c}{(l-m)^{\hat{r}}} \mu_k(m)^{1+\eta} ,
\end{equation*}
where c is a constant independent of $k$.
Now applying [\citealp{Iterative}, Lemma B.$1$], we conclude that
$\mu_k(m_0+d_k)=0,$ where $d_k=c[\mu_k(m_0)]^\eta2^{\hat{r}(1+\eta)/\eta}$. As for each $k$, $d_k\leq c \left(T|\Omega|^{\frac{\hat{r}}{\hat{q}}}\right)^\eta$, we get
\begin{equation*}
\left\|u_k\right\|_{L^{\infty}(\Omega \times[0, \tau])} \leq d .
\end{equation*}
We can iterate this procedure in the sets $\Omega \times[\tau, 2 \tau], \cdots, \Omega \times[j \tau, T]$, where $T-j \tau \leqslant \tau$ to conclude that
\begin{equation*}
    \left\|u_k\right\|_{L^{\infty}\left(\Omega_T\right)} \leq C \quad \text { uniformly in } k \in \mathbb{N} .
\end{equation*}
Now, since the sequence $\{u_k\}_k$ is increasing in $k$, we have $\|u\|_{L^\infty(\Omega)}\leq C$. Further, if $\gamma\leq 1$, testing \cref{approx4} with the function $u_k$ we can deduce $u\in L^2(0,T;H^1_0(\Omega))$ (see \cref{mainth1}).
\subsection*{\textbf{Proof of \cref{mainth5.7}}}
For $\gamma\geq 1$, we choose $\delta>\frac{1+\gamma}{2}$ and for $\gamma<1$, we choose $\delta>1$, and take $u_k^{2\delta-1}\chi_{(0,t)},0<t\leq T$ as test function in \cref{approx4}, with $u_0\equiv 0$. We note that, such a test function is admissible as $u_k\in L^\infty(\Omega_T)$.
We get $\forall t\leq T,$
\begin{equation*}
 \begin{array}{c}
 \quad\frac{1}{2\delta} \int_{\Omega} u_k^{2\delta}(x, t) d x+\int_0^t\int_\Omega \nabla u_k\cdot\nabla u_k^{2\delta-1}d x d\theta\smallskip\\\quad+\frac{1}{2} \int_0^t \int_Q \frac{\left(u_k^{2\delta-1}(x, \theta)-u_k^{2\delta-1}(y, \theta)\right)\left(u_k(x, \theta)-u_k(y, \theta)\right)}{|x-y|^{n+2 s}} d x d y d \theta 
\leq \iint_{\Omega_t} f u_k^{2\delta-1-\gamma}d x d \theta.
\end{array}
\end{equation*}
Taking supremum over $t\in(0,T]$ and using item $(i)$ of \cref{algebraic}, by Hölder and Sobolev inequalities, we have
\begin{equation}{\label{test1}}
\begin{array}{c}
\sup _{t \in[0, T]} \int_{\Omega} u_k^{2 \delta} (x,t)d x+\frac{2(2\delta-1)}{ \delta} \lambda \int_0^T\left\|u_k^\delta\right\|_{L^{2^*}(\Omega)}^2 d t +\frac{(2\delta-1)}{ \delta} \lambda_s\int_0^T\left\|u_k^\delta\right\|_{L^{2_s ^*}(\Omega)}^2 d t \smallskip\\
\leq 2\delta\left\|f\right\|_{L^r\left(0, T ; L^q(\Omega)\right)}\left(\int_0^T\left[\int_{\Omega} u_k^{(2 \delta-1-\gamma) q^{\prime}} d x\right]^{\frac{r^{\prime}}{q^{\prime}}} d t\right)^{\frac{1}{r^{\prime}}} .
\end{array}
\end{equation}
Note that $\frac{2\delta-1}{\delta}>1$, so we can ignore the constants in left. Denoting by
$
A=\left(\int_0^T\left\|u_k\right\|_{L^{(2 \delta-1-\gamma) q^{\prime}}(\Omega)}^{(2 \delta-1-\gamma) r^{\prime}} d t\right)^{\frac{1}{r^{\prime}}},    
$
we get from \cref{test1},
\begin{equation}{\label{res1}}
\sup _{t\in[0, T]} \int_{\Omega} u_k^{2 \delta} d x \leq 2 \delta\left\|f\right\|_{L^r\left(0, T ; L^q(\Omega)\right)} A,    
\end{equation}
and
\begin{equation}{\label{res2}}
    \int_0^T\left[\int_{\Omega} u_k^{2_s^*\delta} d x\right]^{\frac{2}{2 _s^*}} d t \leq  c\int_0^T\left[\int_{\Omega} u_k^{2^*\delta} d x\right]^{\frac{2}{2^*}} d t \leq\frac{c}{\lambda} 2\delta\left\|f\right\|_{L^r\left(0, T ; L^q(\Omega)\right)} A .
\end{equation}
\smallskip\\
\textbf{Case 1: $1<q<\frac{n r}{n+2 (r-1)}$ \text{ i.e. } $\frac{1}{r}<\frac{n}{n-2} \frac{1}{q}-\frac{2}{n-2}$} \smallskip\\Since, $1<q<\frac{n r}{n+2 (r-1)}$ implies $q<r$ and $\frac{r^{\prime}}{q^{\prime}}<\frac{2}{2^*}$, we apply Hölder inequality with exponent $\frac{2q^{\prime}}{2^*r^{\prime}}$ to get
\begin{equation}{\label{res2.5}}
A \leq T^{\frac{1}{r^{\prime}}-\frac{2^*}{2 q^{\prime}}}\left[\int_0^T\left(\int_{\Omega} u_k^{(2 \delta-1-\gamma) q^{\prime}} d x\right)^{\frac{2}{2 ^*}} d t\right]^{\frac{2^*}{2 q^{\prime}}} .    
\end{equation}
Thus from \cref{res2}, it follows
\begin{equation*}{\label{res3}}
\int_0^T\left[\int_{\Omega} u_k^{2^* \delta} d x\right]^{\frac{2}{2^*}} d t \leq 2\delta cA \leq 2\delta c\left[\int_0^T\left(\int_{\Omega} u_k^{(2 \delta-1-\gamma) q^{\prime}} d x\right)^{\frac{2}{2 ^*}} d t\right]^{\frac{2^*}{2 q^{\prime}}} . \smallskip   
\end{equation*}
We now choose $2^* \delta=(2 \delta-1-\gamma) q^{\prime}$, that is, $\delta=\frac{1}{2} \cdot \frac{q(n-2 )(\gamma+1)}{(n-2 q )}$. Note that if $\gamma\geq 1$, then $\delta >\frac{\gamma+1}{2}$ if and only if $q>1$, and for $\gamma<1$, $\delta>1$ if and only if $q>\left(\frac{2^*}{1-\gamma}\right)^{\prime}$.  Moreover, since $q<\frac{n}{2 }$ it follows that $\frac{2^*}{2 q^{\prime}}<1$. Thus we get
\begin{equation}{\label{res4}}
   \int_0^T\left(\int_{\Omega} u_k^{2^*\delta} d x\right)^{\frac{2}{2^*}} d t=\int_0^T\left(\int_{\Omega} u_k^{\left(2 \delta-1-\gamma\right) q^{\prime}} d x\right)^{\frac{2}{2^*}} d t \leq  c,
\end{equation}
and therefore from \cref{res1}, \cref{res2.5} and \cref{res4}, it follows
\begin{equation*}
    \left\|u_k\right\|_{L^\infty\left(0, T ; L^{2\sigma}(\Omega)\right)} +\left\|u_k\right\|_{L^{2\sigma}\left(0, T ; L^{2^*\sigma}(\Omega)\right)} \leq  c, \text{ with } \sigma=\frac{q(n-2 )(\gamma+1)}{2(n-2 q )}>\frac{q}{2}.
\end{equation*}
\smallskip\\\textbf{Case 2: $\frac{nr}{n+2(r-1)} \leq q<\frac{n}{2 } r^{\prime}$ i.e. $\frac{1}{r}\geq\frac{n}{n-2} \frac{1}{q}-\frac{2}{n-2}$ and $\frac{1}{r}+\frac{n}{2 q }>1$}\smallskip
\\In this case we choose $\delta=\frac{1}{2} \frac{q r n(\gamma+1)}{nr-2 q (r-1)}$. Note that for $\gamma\geq 1$, $\delta >\frac{1+\gamma}{2},$ if and only if $\frac{1}{r}+\frac{n}{2 q }<1+\frac{n}{2}$ which is always satisfied and for $\gamma <1$, $\delta>1$ if and only if $q>\frac{2nr}{nr(\gamma+1)+4(r-1)}$ which is satisfied if $r>\frac{2}{\gamma+1}$. Now by interpolation, we get that $\frac{1}{(2 \delta-1-\gamma) q^{\prime}}=\frac{1-\theta}{2 \delta}+\frac{\theta}{2^* \delta}$, where $\theta=\frac{nq(r-1)}{nr(q-1)+2q(r-1)}\in(0,1]
$. Therefore
\begin{equation*}
    \int_0^T\left\|u_k\right\|_{L^{(2 \delta-1-\gamma) q^{\prime}}(\Omega)}^{r^{\prime}(2 \delta-1-\gamma)} d t \leq c\left\|u_k\right\|_{L^{\infty}\left(0, T ; L^{2 \delta}(\Omega)\right)}^{2 \delta \mu_1} \int_0^T\left(\int_{\Omega} u_k^{2^ * \delta} d x\right)^{\frac{2}{2^*} \mu_2} d t,
\end{equation*}
where $\mu_1=\frac{(1-\theta) r^{\prime}(2 \delta-1-\gamma)}{2 \delta}$ and $\mu_2=\frac{\theta r^{\prime}(2 \delta-1-\gamma)}{2\delta}$. Since $\mu_2\leq 1$, we have that
\begin{equation*}
       \int_0^T\left\|u_k\right\|_{L^{(2 \delta-1-\gamma) q^{\prime}}(\Omega)}^{r^{\prime}(2 \delta-1-\gamma)} d t \leq c\left\|u_k\right\|_{L^{\infty}\left(0, T ; L^{2 \delta}(\Omega)\right)}^{2 \delta \mu_1} \left(\int_0^T\left[\int_{\Omega} u_k^{2^* \delta} d x\right]^{\frac{2}{2^*}} d t\right)^{\mu_2},
\end{equation*}
and since $\mu_1+\mu_2 < r^{\prime}$, using \cref{res1,res2}, we can conclude the following which will give our final result
\begin{equation*}
    \int_0^T\left(\int_{\Omega} u_k^{(2 \delta-1-\delta) q^{\prime}} d x\right)^{\frac{r^{\prime}}{q^{\prime}}} d t \leq c.
\end{equation*}
\subsection*{\textbf{Proof of \cref{mainth7}}}
Consider the approximated problem \cref{approxvari} and corresponding properties of $\{u_k\}_k$, \cref{approxexxist}. We denote $(\omega_{\tilde{T}})_\delta=\Omega_T \backslash (\Omega_T)_\delta$, then $u_k(x) \geq C((\omega_{\tilde{T}})_\delta,n,s)>0$ in $(\omega_{\tilde{T}})_\delta$ for all $k$. Choosing $u_k\chi_{(0,\tau)}(t)$ as test function in \cref{approxvari}, we obtain
\begin{equation*}
\begin{array}{l}
\quad\frac{1}{2}\int_{\Omega} u_k^2(x, \tau) d x+\int_0^\tau\int_\Omega |\nabla u_k|^2d x d t+\frac{1}{2}\int_0^\tau \int_Q \frac{\left(u_k(x, t)-u_k(y, t)\right)^2}{|x-y|^{n+2 s}} d x d y d t \smallskip\\\leq\iint_{\Omega_T} \frac{f_ku_k}{\left(u_k+\frac{1}{k}\right)^{\gamma(x,t)}} d x d t+\frac{1}{2} \int_{\Omega} u_0^2(x) d x.
\end{array}
\end{equation*}
Now taking supremum over $\tau\in(0,T]$, we get
\begin{equation}{\label{variabletest}}
    \begin{array}{l}
    \quad \sup _{0 \leq \tau \leq T} 
\int_{\Omega} u_k^2(x, \tau) d x+2\int_0^T\int_\Omega |\nabla u_k|^2d x d t+\int_0^T \int_Q \frac{\left(u_k(x, t)-u_k(y, t)\right)^2}{|x-y|^{n+2 s}} d x d y d t \smallskip\\\leq2\iint_{\Omega_T} \frac{f_ku_k}{\left(u_k+\frac{1}{k}\right)^{\gamma(x,t)}} d x d t+\left\|u_0\right\|_{L^{2}(\Omega)} \smallskip\\
=2\iint_{(\Omega_T)_\delta} \frac{f_ku_k}{\left(u_k+\frac{1}{k}\right)^{\gamma(x,t)}} d x d t+2\iint_{(\omega_{\tilde{T}})_\delta} \frac{f_ku_k}{\left(u_k+\frac{1}{k}\right)^{\gamma(x,t)}} d x d t +\left\|u_0\right\|_{L^{2}(\Omega)}\smallskip \\
 \leq 2\iint_{(\Omega_T)_\delta}  f u_k^{1-\gamma(x,t)} d x d t+2\iint_{(\omega_{\tilde{T}})_\delta}  \frac{fu_k}{C_{(\omega_{\tilde{T}})_\delta}^{\gamma(x,t)}} d x d t+\left\|u_0\right\|_{L^{2}(\Omega)}\smallskip\\
 \leq 2\iint_{(\Omega_T)_\delta\cap\left\{u_k \leq 1\right\}}  f u_k^{1-\gamma(x,t)} d x d t +2\iint_{(\Omega_T)_\delta\cap\left\{u_k \geq 1\right\}}  f u_k^{1-\gamma(x,t)} d x d t 
\smallskip\\\quad+2\iint_{(\omega_{\tilde{T}})_\delta}  \frac{fu_k}{C_{(\omega_{\tilde{T}})_\delta}^{\gamma(x,t)}} d x d t+\left\|u_0\right\|_{L^{2}(\Omega)}  \\\leq\left\|u_0\right\|_{L^{2}(\Omega)} +2\|f\|_{L^1(\Omega_T)}+2\left(1+\left\|C_{(\omega_{\tilde{T}})_\delta}^{-\gamma(\cdot)}\right\|_{L^{\infty}(\Omega_T)}\right) \iint_{\Omega_T} f u_k d x d t. 
\end{array}
\end{equation}
\textbf{Case 1: $f \in L^{2}\left(0, T ; L^{\left(\frac{2n}{n+2}\right)}(\Omega)\right)$}\smallskip\\We note that $(2^*)^\prime=\frac{2n}{n+2}$, then by using Hölder's and Sobolev's inequalities, we have
\begin{equation*}
    \begin{array}{rcl}
         \int_0^T\int_{\Omega} f u_k d x d t&\leq &C\int_0^T\left(\int_{\Omega}|f(x, t)|^{(2^*)^\prime} d x\right)^{\frac{1}{(2^*)^\prime}}\left(\int_{\Omega}\left|u_k(x, t)\right|^{2^*} d x\right)^{\frac{1}{2 ^*}} d t
         \smallskip\\&\leq&C\int_0^T\|f\|_{L^{\left({2^*}\right)^\prime}{(\Omega)}}\left(\int_{\Omega}\left|\nabla u_k\right|^{2} d x\right)^{\frac{1}{2 }} d t\smallskip\\&\leq& C\left(\int_0^T\|f\|_{L^{\left({2^*}\right)^\prime}{(\Omega)}}^2d t\right)^{\frac{1}{2}}\left(\int_0^T\int_{\Omega}\left|\nabla u_k\right|^{2} d x d t\right)^{\frac{1}{2}}.
    \end{array}
\end{equation*}
In the above, we use Young's inequality, and then from \cref{variabletest} we get that
\begin{equation*}
\begin{array}{l}
\quad\sup _{0 \leq \tau \leq T} 
\int_{\Omega} u_k^2(x, \tau) d x+\left\|u_k\right\|_{L^2(0,T;H_0^{1}(\Omega))}^2\leq\left\|u_0\right\|_{L^{2}(\Omega)}+2\|f\|_{L^1(\Omega_T)}+C\left(1+\left\|C_{(\omega_{\tilde{T}})_\delta}^{-\gamma(\cdot)}\right\|_{L^{\infty}(\Omega_T)}\right). \end{array}
\end{equation*}
Therefore the sequence $\left\{u_k\right\}_k$ is uniformly bounded in $
L^2(0, T ; H_0^1(\Omega)) \cap L^{\infty}(0, T ; L^2(\Omega))$.
\smallskip\\\textbf{Case 2: $f \in L^{\bar{r}
}\left(\Omega_T\right)$}\smallskip\\
For this case, we apply Hölder inequality with exponents $\bar{r}$ and $\bar{r}^\prime$ to get
\begin{equation}{\label{fusobolev}}
    \begin{array}{rcl}
         \int_0^T\int_{\Omega} f u_k d x d t\leq C\|f\|_{L^{\bar{r}}\left(\Omega_T\right)}\left[\iint_{\Omega_T} |u_k|^{\bar{r}^{\prime}} d x d t\right]^{\frac{1}{\bar{r}^{\prime}}}.
         \end{array}
         \end{equation}
We observe that $ \bar{r}^{\prime}=\frac{2(n+2 )}{n}$ and hence using the Hölder inequality with exponents $\frac{n}{n-2 }$ and $\frac{n}{2 }$ and by Sobolev embedding (\cref{Sobolev embedding}), we can write
\begin{equation}{\label{fusobolev1}}
    \begin{array}{rcl}
\iint_{\Omega_T}\left|u_k\right|^{\frac{2(n+2 )}{n}} d x d t=\iint_{\Omega_T}|u_k|^2\left|u_k\right|^{\frac{4 }{n}} d x d t 
& \leq &\int_0^T\left[\int_{\Omega}\left|u_k(x, t)\right|^2 d x\right]^{\frac{2 }{n}}\left\|u_k\right\|_{L^{2 ^*}(\Omega)}^2 d t \\
& \leq& C(n)\left[\sup _{0 \leq \tau\leq T} \int_{\Omega}\left|u_k(x, \tau)\right|^2 d x\right]^{\frac{2 }{n}} \int_0^T \int_\Omega |\nabla u_k|^2 d x d t.
    \end{array}
\end{equation}
So using \cref{fusobolev} in \cref{variabletest}, we get from \cref{fusobolev1} by convexity argument
\begin{equation*}
\iint_{\Omega_T}\left|u_k\right|^{\frac{2(n+2 )}{n}} d x d t\leq C(n) 2^{\frac{2 }{n}}\left(\left(C_1C\|f\|_{L^{\bar{r}}\left(\Omega_T\right)}\right)^{\frac{n+2 }{n}}\left(\iint_{\Omega_T} |u_k|^{ {\bar{r}^{\prime}}}\right)^{\frac{n+2 }{n \bar{r}^\prime}}+C_2^{\frac{n+2 }{n}}\right) ,
\end{equation*}
where $C_1=2\left(1+\left\|C_{(\omega_{\tilde{T}})_\delta}^{-\gamma(\cdot)}\right\|_{L^{\infty}(\Omega_T)}\right) $ and $C_2=\left\|u_0\right\|_{L^{2}(\Omega)} +2\|f\|_{L^1(\Omega_T)}$.
Now since $\frac{n+2 }{n \bar{r}^{\prime}}=\frac{1}{2}$, we use Young's inequality to obtain
\begin{equation*}
\iint_{\Omega_T}\left|u_k\right|^{\bar{r}^\prime} d x d t =\iint_{\Omega_T}\left|u_k\right|^{\frac{2(n+2 )}{n}} d x d t \leq C,
\end{equation*}
where $C$ is a positive constant independent of $k$. Therefore, by \cref{variabletest,fusobolev} we deduce that $\left\{u_k\right\}_k$ is uniformly bounded in $
L^2(0, T ; H_0^1(\Omega)) \cap L^{\infty}(0, T ; L^2(\Omega))$.
\smallskip\\Since $\left\{u_k\right\}_k$ is uniformly bounded in the reflexive Banach space $L^2(0, T ; H_0^1(\Omega)) $, there exist a subsequence of $\left\{u_k\right\}_k$, still indexed by $k$, and a measurable function $u \in 
L^2(0, T ; H_0^1(\Omega)) $ such that $u_k \rightharpoonup u$ weakly in $
L^2(0, T ; H_0^1(\Omega)) $. Also, since $\left\{u_k\right\}$ is increasing in $k$, it holds by Beppo Levi's theorem that $u_k\rightarrow u$ strongly in $L^1\left(\Omega_T\right)$ and hence a.e. in $\Omega_T$. 
Applying Fatou's Lemma, we get $u\in L^\infty(0,T;L^2(\Omega))$. This pointwise limit is actually defined for each $t\in[0,T)$ and satisfies $u(\cdot,0)=u_0(\cdot)$ in $L^1 $ sense (see \cref{ubase}). We show that this $u$ is a very weak solution to \cref{problemvari} in the sense of \cref{very weak variable}. Choosing arbitrary $\phi \in \mathcal{C}_0^{\infty}(\Omega_T)$, we have
\begin{equation*}
\iint_{\Omega_T}(\left(u_k\right)_t-\Delta u_k+(-\Delta)^s u_k) \phi \,d x d t=\iint_{\Omega_T} \frac{f_k\phi}{\left(u_k+\frac{1}{k}\right)^{\gamma(x,t)}} d x d t .
\end{equation*}
Since $\phi \in \mathcal{C}_0^{\infty}(\Omega_T)$, therefore $ \phi_t, \Delta \phi$ and $(-\Delta)^s\phi$ all are bounded. As $u_k \rightarrow u$ strongly in $L^1\left(\Omega_T\right)$, we have 
\begin{equation*}
    \begin{array}{rcl}
\iint_{\Omega_T}(\left(u_k\right)_t-\Delta u_k+(-\Delta)^s u_k) \phi \,d x d t &=&\iint_{\Omega_T} u_k(-(\phi)_t-\Delta \phi+(-\Delta)^s \phi) d x d t \smallskip\\&\rightarrow &\iint_{\Omega_T} u(-(\phi)_t-\Delta \phi+(-\Delta)^s \phi) d x d t .
    \end{array}
\end{equation*}
Now since for all $\omega \subset \subset \Omega $ and for all $t_0>0,$ $ u_k(x, t) \geq C\left(\omega, t_0,n,s\right) $ in $\omega \times\left[t_0, T\right)$, we reach that $u_k(x, t) \geq C$ in $\operatorname{Supp} \phi$ (say $\omega\times[t_1,t_2]$) for all $k \geq 1$. Therefore
\begin{equation*}
    0\leq \left|\frac{f_k\phi}{\left(u_k+\frac{1}{k}\right)^{\gamma(x,t)}}\right|\leq \|\phi C_{\omega,t_1,t_2}^{-\gamma(x,t)}\|_{L^\infty(\Omega_T)}f.
\end{equation*}Hence by the dominated convergence theorem
\begin{equation*}
    \iint_{\Omega_T} \frac{f_k\phi}{\left(u_k+\frac{1}{k}\right)^{\gamma(x,t)}} d x d t \rightarrow \iint_{\Omega_T} \frac{f\phi}{u^{\gamma(x,t)}} d x d t \text { as } k\rightarrow \infty .
\end{equation*}
Thus, we get the following and conclude
\begin{equation*}
    \iint_{\Omega_T}(u_t-\Delta u+(-\Delta)^s u) \phi\, d x d t=\iint_{\Omega_T} \frac{f\phi}{u^{\gamma(x,t)}} d x d t.
\end{equation*}
\subsection*{\textbf{Proof of \cref{mainth8}}}
Consider $u_k$ to be the unique nonnegative solution to \cref{approxvari}. 
Since $\gamma^*>1$, so we can use $u_k^{\gamma^*}\chi_{(0,\tau)}(t)$ as a test function in \cref{approxvari} to get for all $\tau\leq T$,
\begin{equation*}
\begin{array}{c}
 \frac{1}{\gamma^*+1} \int_{\Omega} u_k^{\gamma^*+1}(x, \tau) d x+\int_0^\tau\int_\Omega \nabla u_k\cdot\nabla u_k^{\gamma^*}d x d t\smallskip\\+\frac{1}{2} \int_0^\tau \int_Q \frac{(u_k^{\gamma^*}(x, t)-u_k^{\gamma^*}(y, t))\left(u_k(x, t)-u_k(y, t)\right)}{|x-y|^{n+2 s}} d x d y d t 
\smallskip\\\leq \iint_{\Omega_T} \frac{f_ku_k^{\gamma^*}}{\left(u_k+\frac{1}{k}\right)^{\gamma(x,t)}} d x d t+\frac{1}{\gamma^*+1} \int_{\Omega} u_0^{\gamma^*+1}(x) d x .
\end{array}
\end{equation*}
Now taking supremum over $\tau\in(0,T]$ and using item $i)$ of \cref{algebraic}, we get
\begin{equation*}
\begin{array}{c}
\quad\frac{({\gamma^*}+1)}{2\gamma^*} \sup _{0 \leq \tau \leq T} \int_{\Omega}\left|u_k(x, \tau)\right|^{{\gamma^*}+1} d x+\frac{({\gamma^*}+1)^2}{2}\int_0^T\int_\Omega u_k^{{\gamma^*}-1}|\nabla u_k|^2d x d t\smallskip\\\quad+\int_0^T \int_Q \frac{|u_k^{\frac{{\gamma^*}+1}{2}}(x, t)-u_k^{\frac{{\gamma^*}+1}{2}}(y, t)|^2}{|x-y|^{n+2 s}} d x d y d t \smallskip\\
 \leq \frac{({\gamma^*}+1)^2}{2\gamma^*} 
 \left(\iint_{\Omega_T} \frac{f_ku_k^{\gamma^*}}{\left(u_k+\frac{1}{k}\right)^{\gamma(x,t)}} d x d t+\left\|u_0\right\|_{L^{\gamma^*+1}(\Omega)}\right) .
\end{array}
\end{equation*}
We note that $1<\frac{\gamma^*+1}{\gamma^*}<2$ and hence it follows 
\begin{equation}{\label{varitest2}}
\begin{array}{c}
\sup _{0 \leq \tau \leq T} \int_{\Omega}\left|u_k(x, \tau)\right|^{{\gamma^*}+1} d x+4\int_0^T\int_\Omega |\nabla u_k^{\frac{{\gamma^*}+1}{2}}|^2d x d t\\+2\int_0^T \int_Q \frac{|u_k^{\frac{{\gamma^*}+1}{2}}(x, t)-u_k^{\frac{{\gamma^*}+1}{2}}(y, t)|^2}{|x-y|^{n+2 s}} d x d y d t 
 \leq 4\gamma^*
 \left(\iint_{\Omega_T} \frac{f_ku_k^{\gamma^*}}{\left(u_k+\frac{1}{k}\right)^{\gamma(x,t)}} d x d t+\left\|u_0\right\|_{L^{\gamma^*+1}(\Omega)}\right) .
\end{array}
\end{equation}
Denoting $(\omega_{\tilde{T}})_\delta=\Omega_T \backslash (\Omega_T)_\delta$, we have $u_k(x,t) \geq C((\omega_{\tilde{T}})_\delta,n,s)
>0$ in $(\omega_{\tilde{T}})_\delta$ for all $k$. We estimate like \cref{mainth7} as
\begin{equation*}
    \begin{array}{rcl}
         \iint_{\Omega_T} \frac{f_ku_k^{\gamma^*}}{\left(u_k+\frac{1}{k}\right)^{\gamma(x,t)}} d x d t&=&\iint_{(\Omega_T)_\delta} \frac{f_ku_k^{\gamma^*}}{\left(u_k+\frac{1}{k}\right)^{\gamma(x,t)}} d x d t+\iint_{(\omega_{\tilde{T}})_\delta} \frac{f_ku_k^{\gamma^*}}{\left(u_k+\frac{1}{k}\right)^{\gamma(x,t)}} d x d t \smallskip\\
& \leq& \iint_{(\Omega_T)_\delta}  f u_k^{{\gamma^*}-\gamma(x,t)} d x d t+\iint_{(\omega_{\tilde{T}})_\delta}  \frac{fu_k^{\gamma^*}}{C_{(\omega_{\tilde{T}})_\delta}^{\gamma(x,t)}} d x d t\smallskip\\
& \leq& \iint_{(\Omega_T)_\delta\cap\left\{u_k \leq 1\right\}}  f u_k^{{\gamma^*}-\gamma(x,t)} d x d t +\iint_{(\Omega_T)_\delta\cap\left\{u_k \geq 1\right\}}  f u_k^{\gamma^*-\gamma(x,t)} d x d t \smallskip\\&&
+\iint_{(\omega_{\tilde{T}})_\delta}  \frac{fu_k^{\gamma^*}}{C_{(\omega_{\tilde{T}})_\delta}^{\gamma(x,t)}} d x d t\smallskip\\&\leq&\|f\|_{L^1(\Omega_T)}+\left(1+\left\|C_{(\omega_{\tilde{T}})_\delta}^{-\gamma(\cdot)}\right\|_{L^{\infty}(\Omega_T)}\right) \iint_{\Omega_T} f u_k^{\gamma^*} d x d t. 
    \end{array}
\end{equation*}
Using the above in \cref{varitest2} we get
\begin{equation}{\label{varitestmain}}
    \begin{array}{l}
 \quad  \sup _{0 \leq \tau \leq T} \int_{\Omega}\left|u_k(x, \tau)\right|^{{\gamma^*}+1} d x+2\int_0^T\int_\Omega |\nabla u_k^{\frac{{\gamma^*}+1}{2}}|^2d x d t
\smallskip\\ \leq 4\gamma^*\left(1+\left\|C_{(\omega_{\tilde{T}})_\delta}^{-\gamma(\cdot)}\right\|_{L^{\infty}(\Omega_T)}\right)\iint_{\Omega_T} fu_k^{\gamma^*}d x d t+4\gamma^*\left(\|f\|_{L^1(\Omega_T)}+\left\|u_0\right\|_{L^{\gamma^*+1}(\Omega)}\right).
 \end{array}
\end{equation}
For convenience we take $C_1=4\gamma^*\left(1+\left\|C_{(\omega_{\tilde{T}})_\delta}^{-\gamma(\cdot)}\right\|_{L^{\infty}(\Omega_T)}\right)$ and $C_2=4\gamma^*\left(\|f\|_{L^1(\Omega_T)}+\left\|u_0\right\|_{L^{\gamma^*+1}(\Omega)}\right)$.
\smallskip\\\textbf{Case 1: $f \in L^{\gamma^*+1}\left(0, T ; L^{\left(\frac{n(\gamma^*+1)}{n+2\gamma^*}\right)}(\Omega)\right)$}\\In this case using Hölder inequality first in the space variable and then by Sobolev inequality and another application of H\"older inequality (in time variable) we get
\begin{equation*}
    \begin{array}{rcl}
         \int_0^T\int_{\Omega} f u_k^{\gamma^*} d x d t&=& \int_0^T\int_{\Omega} f \left(u_k^{\frac{\gamma^*+1}{2}}\right)^{\frac{2\gamma^*}{\gamma^*+1}} d x d t\smallskip\\&\leq &C\int_0^T\left(\int_{\Omega}|f(x, t)|^{\frac{n(\gamma^*+1)}{n+2\gamma^*}} d x\right)^{\frac{n+2\gamma^*}{n(\gamma^*+1)}}\left(\int_{\Omega}|u_k^{\frac{\gamma^*+1}{2}}(x, t)|^{2^*} d x\right)^{\frac{(n-2)\gamma^*}{n(\gamma^*+1)}} d t
        \smallskip \\&\leq&C\int_0^T\|f(\cdot,t)\|_{L^{\frac{n(\gamma^*+1)}{n+2\gamma^*}}{(\Omega)}}\left(\int_{\Omega}|\nabla u_k^{\frac{\gamma^*+1}{2}}|^{2} d x\right)^{\frac{\gamma^*}{\gamma^*+1}} d t\smallskip\\&\leq& C\left(\int_0^T\|f(\cdot,t)\|_{L^{\frac{n(\gamma^*+1)}{n+2\gamma^*}}{(\Omega)}}^{\gamma^*+1} d t\right)^{\frac{1}{\gamma^*+1}}\times\left(\int_0^T\int_{\Omega}|\nabla u_k^{\frac{\gamma^*+1}{2}}|^{2} d x d t\right)^{\frac{\gamma^*}{\gamma^*+1}}.
    \end{array}
\end{equation*}
Now since $\frac{\gamma^*}{\gamma^*+1}<1$, so using Young's inequality in the above estimate, we get from \cref{varitestmain} that
\begin{equation*}
    \begin{array}{l}
  \quad \sup _{0 \leq \tau \leq T} \int_{\Omega}\left|u_k(x, \tau)\right|^{{\gamma^*}+1} d x+2\int_0^T\int_\Omega |\nabla u_k^{\frac{{\gamma^*}+1}{2}}|^2d x d t
 \smallskip\\\leq C\left\|f\right\|_{L^{\gamma^*+1}\left(0, T ; L^{\frac{n(\gamma^*+1)}{n+2\gamma^*}}(\Omega)\right)}\times\left(\varepsilon\int_0^T\int_\Omega |\nabla u_k^{\frac{{\gamma^*}+1}{2}}|^2d x d t+C(\varepsilon)\right)+C_2.
 \end{array}
\end{equation*}
Choosing $\varepsilon$ small enough such that $
    \varepsilon C\left\|f\right\|_{L^{\gamma^*+1}\left(0, T ; L^{\frac{n(\gamma^*+1)}{n+2\gamma^*}}(\Omega)\right)}=1,
$ we get
\begin{equation*}
   \sup _{0 \leq \tau \leq T} \int_{\Omega}\left|u_k(x, \tau)\right|^{{\gamma^*}+1} d x+\int_0^T\int_\Omega |\nabla u_k^{\frac{{\gamma^*}+1}{2}}|^2d x d t\leq C,
   \end{equation*}
   which implies that the sequence $\left\{u_k^{\frac{\gamma^*+1}{2}}\right\}_k$ is uniformly bounded in $
L^2(0, T ; H_0^1(\Omega)) \cap L^{\infty}(0, T ; L^{2}(\Omega))$ and $\left\{u_k\right\}_k$ is uniformly bounded in $
L^{\infty}(0, T ; L^{\gamma^*+1}(\Omega))$.
\smallskip\\\textbf{Case 2: $f \in L^{\tilde{r}}\left(\Omega_T\right)$}\\ For this case, we apply Hölder inequality with exponents $\tilde{r}$ and $\tilde{r}^\prime$ to get
\begin{equation}{\label{fusobolevsecond}}
    \begin{array}{rcl}
         \int_0^T\int_{\Omega} f u_k^{\gamma^*} d x d t\leq C\|f\|_{L^{\tilde{r}}\left(\Omega_T\right)}\left[\iint_{\Omega_T} |u_k|^{\gamma^*\tilde{r}^{\prime}} d x d t\right]^{\frac{1}{\tilde{r}^{\prime}}}.
         \end{array}
         \end{equation}
We note that $\gamma^*\tilde{r}^\prime=\frac{(n+2)(\gamma^*+1)}{n}$. Again, we use the same technique as that of \cref{mainth5}, \cref{mainth6} and \cref{mainth7}. Using the Hölder inequality with exponent $\frac{n}{n-2 }$ and $\frac{n}{2 }$ we get
\begin{equation*}
    \begin{array}{rcl}
\iint_{\Omega_T}\left|u_k\right|^{\frac{({\gamma^*}+1)(n+2 )}{n}} d x d t&= &\iint_{\Omega _T}\left|u_k\right|^{2\frac{{\gamma^*}+1}{2}}\left|u_k\right|^{\frac{4 }{n} \frac{{\gamma^*}+1}{2}} d x d t \smallskip\\
 &\leq& \int_0^T\left(\int_{\Omega}\left|u_k(x, t)\right|^{2^* \frac{{\gamma^*}+1}{2}} d x\right)^{\frac{n-2}{n}}\left(\int_{\Omega}\left|u_k(x, t)\right|^{{\gamma^*}+1} d x\right)^{\frac{2 }{n}} d t \smallskip\\
 &=&\int_0^T\|u_k^{\frac{{\gamma^*}+1}{2}}\|_{L^{2 ^*}(\Omega)}^2\left(\int_{\Omega}\left|u_k(x, t)\right|^{{\gamma^*}+1} d x\right)^{\frac{2 }{n}} d t .
\end{array}
\end{equation*}
Now taking supremum over $t\in [0, T]$ and applying the Sobolev embedding as that of \cref{Sobolev embedding}, we can write
\begin{equation*}
\iint_{\Omega_T}\left|u_k\right|^{\frac{({\gamma^*}+1)(n+2 )}{n}} d x d t  \leq C(n)\left(\sup _{0 \leq t \leq T} \int_{\Omega}\left|u_k(x, t)\right|^{{\gamma^*}+1} d x\right)^{\frac{2 }{n}} \left(\iint_{\Omega_T} |\nabla u_k^{\frac{{\gamma^*}+1}{2}}|^2 d x d t \right).
\end{equation*}
Using \cref{varitestmain} and \cref{fusobolevsecond}, from the above inequality we get by convexity argument
\begin{equation*}
\begin{array}{rcl}
\iint_{\Omega_T}\left|u_k\right| ^{\frac{({\gamma^*}+1)(n+2 )}{n}} d x d t &
 \leq &C(n)\left(C_1\iint_{\Omega_T} f u_k^{{\gamma^*}} d x d t+C_2\right)^{\frac{n+2 }{n}} \\
&\leq& C(n) 2^{\frac{2 }{n}}\left(\left(C_1\iint_{\Omega_T} f u_k^{{\gamma^*}} d x d t\right)^{\frac{n+2 }{n}}+C_2^{\frac{n+2}{n}}\right)\\&\leq&C(n) 2^{\frac{2 }{n}}\left(\left(C_1C\|f\|_{L^{\tilde{r}}(\Omega_T)}\right)^{\frac{n+2 }{n}}\left(\iint_{\Omega_T} |u_k|^{\gamma^*\tilde{r}^{\prime}} d x d t\right)^{\frac{n+2 }{n\tilde{r}^\prime}}+C_2^{\frac{n+2}{n}}\right).
\end{array}
\end{equation*}
Now as $\frac{n+2}{n\tilde{r}^\prime}=\frac{\gamma^*}{\gamma^*+1}<1$, therefore using Young's inequality we get
\begin{equation*}
\iint_{\Omega_T}\left|u_k\right|^{\gamma^*\tilde{r}^\prime}d x d t=\iint_{\Omega_T}\left|u_k\right| ^{\frac{({\gamma^*}+1)(n+2 )}{n}} d x d t \leq C,
\end{equation*}
where $C$ is a positive constant independent of $k$. The above boundedness along with \cref{fusobolevsecond,varitestmain} gives that the sequence $\left\{u_k^{\frac{\gamma^*+1}{2}}\right\}_k$ is uniformly bounded in $
L^2(0, T ; H_0^1(\Omega)) \cap L^{\infty}(0, T ; L^{2}(\Omega))$ and $\left\{u_k\right\}_k$ is uniformly bounded in $
L^{\infty}(0, T ; L^{\gamma^*+1}(\Omega))$. We now show $\left\{u_k\right\}_k$ is uniformly bounded in $
L
^2(t_0, T ; H_{l o c}^1(\Omega))$ for each $t_0>0$.
Since $\gamma^*>1$, and $\Omega_T$ is bounded and $\left\{u_k\right\}_k$ is uniformly bounded in $L^{\infty}(0, T ; L^{\gamma^*+1}(\Omega))$, we thus have $\left\{u_k\right\}_k$ is uniformly bounded in $L^2(0, T ; L^2(\Omega))=L^2(\Omega_T)$, in particular in $L^2(K \times(t_0, T))$, for every $K\subset\subset\Omega$ and for each $t_0>0
$. Again as $\left\{u_k^{\frac{\gamma^*+1}{2}}\right\}_k$ is uniformly bounded in $
L^2(0, T ; H_0^1(\Omega))
$, so we have by nonnegativity of $u_k$
\begin{equation*}
\begin{array}{c}
    \int_{t_0}^{T} \int_K  u_k^{\gamma^*-1}|\nabla u_k|^2 d x d t\leq \int_{0}^{T} \int_\Omega  u_k^{\gamma^*-1}|\nabla u_k|^2 d x d t\le \frac{4C}{(\gamma^*+1)^2},
    \end{array}
\end{equation*}
for every $K \subset\subset \Omega$ and for each $t_0>0$. 
Using the positivity of $u_k$ in $\omega \times [t_0,T)$ for all $k$, we conclude
\begin{equation*}
    \int_{t_0}^{T} \int_K |\nabla u_k|^2d x d y d t\leq \frac{C}{\gamma^* c_{(K,t_0
    )}^{\gamma^*-1}}.
\end{equation*}
Since $\left\{u_k\right\}_k$ is increasing in $k$ and is uniformly bounded in $L^{\infty}(0, T ; L^{\gamma^*+1}(\Omega))$, we get by Beppo Levi's Lemma $u_k \uparrow u$ strongly in $L^{\gamma^*+1}\left(\Omega_T\right)$ and hence in $L^1(\Omega_T)$, where $u$ is the pointwise limit of $u_k$ (possibly infinite). Then employing Fatou's lemma we will get $u^{\frac{\gamma^*+1}{2}} \in L^{\infty}(0, T ; L^2(\Omega)) 
\cap  L^2(0, T ; H_0^1(\Omega))$ and $u\in L^2(t_0, T ; H_{l o c}^1(\Omega))$ for each $t_0>0$. \cref{ubase} implies that $u$ satisfies 
$u(\cdot,0)=u_0(\cdot)$ in $L^1$ sense. Also, we have $u(x, t) \geq$ $C\left(\omega,t_0,n,s\right) $ in $\omega \times\left[t_0, T\right)$ for all $\omega \subset\subset \Omega$ and for each $t_0>0$. This $u$ will be a very weak solution to \cref{problemvari} in the sense of \cref{very weak variable}, and the proof follows exactly similar to \cref{mainth7}.
\begin{remark}{\label{uniqueness}}
In the above existence results, the cases where $u\in L^2(0,T;H^1_0(\Omega))
$, we can show that the weak solution $u$ is unique, if it has finite energy. For this we assume that $v$ is another finite energy solution to \cref{problem} or \cref{problemvari}. 
Then, by the construction of $u$ in each existence result, we get $v$ is a supersolution to all the approximating problems. Hence by the comparison principle in \cref{comparison}, we deduce that $v \geq u_k$ for all $k$. Then $v \geq u$. To prove the inverse inequality, let $\phi=v-u$, then $\phi \geq 0$ and $\phi
\in 
L^2(0, T ; H_0^1(\Omega))$ and hence we can use both $\phi_+$ or $\phi_-$ as test functions. Also, we note that $\phi$ solves the problem
\begin{eqnarray*}
    \begin{cases}
\phi_t-\Delta \phi+(-\Delta)^s \phi =f\left(\frac{1}{v^{\gamma(x,t)}}-\frac{1}{u^{\gamma(x,t)}}\right) &\text { in } \Omega_T, \\
\phi(x, t)  =0 &\text { in }\left(\mathbb{R}^n \backslash \Omega\right) \times(0, T), \\
\phi(x, 0) =0 &\text { in } \Omega .
    \end{cases}
\end{eqnarray*}
Since $f\left(\frac{1}{v^{\gamma(x,t)}}-\frac{1}{u^{\gamma(x,t)}}\right) \leq 0$ in $\Omega_T$, by comparison principle, we get $\phi \leq 0$ in $\Omega_T$. Thus $\phi=0$ i.e. $u=v$.
\end{remark}
\begin{remark}
    As we can notice, in each of the existence results, we have $u\in L^2(t_0,T;W^{1,1}_{\operatorname{loc}}(\Omega))$ for each $t_0>0$ and this gives that our definition of a weak solution is well motivated in the sense that we have derived them integrating by parts twice upon multiplying by a test function.
\end{remark}
\section{Asymptotic behaviour}
This section is devoted to the study of the asymptotic behaviour of finite energy solution to the problem \cref{problem}, as $t \rightarrow \infty$, for the particular case where $f$ depends only on $x$ and that $f>0$ with $f \in L^q
(\Omega)
$  and $q$ will be specified later. 
We first state the following existence and uniqueness result to the corresponding mixed local-nonlocal elliptic problem, which can be done as a direct application of \cref{mainth2}. Consider the problem
\begin{equation}{\label{ellip}}
 \begin{array}{lcr}
\begin{cases}
-\Delta w+(-\Delta)^sw=\frac{f}{w^\gamma} 
&  \mbox{ in } \Omega ,\\
w=0 &  \mbox{ in } (\mathbb{R}^n\backslash \Omega) .
\end{cases}
\end{array}
\end{equation}
We define the weak solution to the above problem as:
\begin{definition}{\label{elliptic}} Suppose that $f \in L^1(\Omega)$ is a nonnegative function and $\gamma>0$. We say that $w$ is a very weak solution to problem \cref{ellip} if $w \in L^1(\Omega)$ satisfies the boundary conditions and
$
    \forall \omega \subset \subset \Omega,$ $  \exists c_\omega>0$ such that $ w(x) \geq c_\omega>0, \text { a.e. in } \omega,
$
and for all $\varphi \in \mathcal{C}_0^{\infty}(\Omega)$, we have
\begin{equation*}
    \int_{\Omega} w(-\Delta \varphi+(-\Delta)^s \varphi )d x=\int_{\Omega} \frac{f \varphi}{w^\gamma} d x .
\end{equation*}
\end{definition}
\begin{theorem} Let $f \in L^1
(\Omega)$ be a non-negative function. 
Then for all $\gamma>0$, the problem \cref{ellip} has a nonnegative very weak solution $w$ such that $w^{\frac{\gamma+1}{2}}\in H_0^1(\Omega)$ in the sense of \cref{elliptic}.
\end{theorem}
\begin{proof} 
Let $w_k$ be the unique positive bounded solution to the approximating problem
\begin{equation}{\label{ellip2}}
 \begin{array}{lcr}
\begin{cases}
-\Delta w_k+(-\Delta)^sw_k=\frac{f_k}{(w_k+\frac{1}{k})^\gamma} 
&  \mbox{ in } \Omega ,\\
w_k >0 & \mbox{ in } \Omega,\\
w_k=0 &  \mbox{ in } (\mathbb{R}^n\backslash \Omega);
\end{cases}
\end{array}
\end{equation}
with $f_k:=\min (k, f)$. We note that the existence of $w_k$ follows using a simple monotonicity argument. More precisely, one can prove a similar version of existence results and comparison principle as of \cref{exist1} for the elliptic case also, and using that, one can find the elliptic version of \cref{comparison} and \cref{corl1}. Again, by the comparison principle, we deduce that the sequence $\left\{w_k\right\}_k$ is increasing in $k$. By the Harnack inequality,  see [\citealp{garainholder}, Theorem 8.3], it holds that for any set $\omega \subset \subset \Omega$, for all $k$ we have,
$
    w_k(x)\geq w_1(x)\geq c(\omega,n,s) \text { in } \omega.
$
Also, the details of existence, uniqueness, positivity and monotonicity of $w_k$'s can be found in [\citealp{garain}, Lemma 3.2].\smallskip\\Now taking $((w_k+\varepsilon)^\gamma-\varepsilon^\gamma), 0<\varepsilon<1/k,$ as a test function in \cref{ellip2} and letting $\varepsilon\to 0$, it follows that
\begin{equation*}
    \int_\Omega \nabla w_k\cdot\nabla w_k^\gamma+\frac{1}{2} \iint_{\mathbb{R}^{2n}} \frac{\left(w_k(x)-w_k(y)\right)\left(w_k^\gamma(x)-w_k^\gamma(y)\right)}{|x-y|^{n+2 s}} d x d y \leq \int_{\Omega} f(x) d x .
\end{equation*}
By \cref{algebraic}, we have
\begin{equation*}
    \left(w_k(x)-w_k(y)\right)\left(w_k^\gamma(x)-w_k^\gamma(y)\right) \geq C\left(w_k^{\frac{\gamma+1}{2}}(x)-w_k^{\frac{\gamma+1}{2}}(y)\right)^2,
\end{equation*}
and on the other hand 
\begin{equation*}
    \nabla w_k\cdot\nabla w_k^\gamma=\gamma w_k^{\gamma -1}|\nabla w_k|^2,
\end{equation*}
we deduce that the sequence $\left\{w_k^{\frac{\gamma+1}{2}}\right\}_k$ is bounded in the reflexive Banach space $X_0^s(\Omega)\cap H_0^1(\Omega)$. Now by the monotonicity of $\left\{w_k\right\}_k$, using Beppo-Levi's theorem, we get the existence of a measurable function $w$ such that $w_k \uparrow w$ a.e. in $\mathbb{R}^n,$ $ w_k^{\frac{\gamma+1}{2}}\rightharpoonup w^{\frac{\gamma+1}{2}}$ weakly in $X_0^s(\Omega)\cap H_0^1(\Omega)$. Clearly $w \geq c(\omega)$ for any set $\omega \subset \subset \Omega$. Then, letting $k \rightarrow \infty$ in the approximating problems and using the dominated convergence theorem, we can show that $w$ is a very weak solution to problem \cref{ellip} with $w^{\frac{\gamma+1}{2}} \in 
H_0^1(\Omega)$. We refer [\citealp{garain}, Theorem $2.16$] for the boundedness of $w$. One can find conditions when $w\in H^1_0(\Omega)$ in \cite{garain}. Further, [\citealp{garain}, Corollary $5.2$] gives such solution is unique.
\end{proof}
We now write the asymptotic behaviour of solutions to problem \cref{problem} under suitable conditions on $f$ and $u_0$.
\begin{theorem}{\label{asypm}}
    Suppose that $f \in L^q
(\Omega)
$ 
is a non-negative function depending only on $x$, and $q$ is such that solution to the problem \cref{ellip} is in $H^1_0(\Omega)$ and is unique. Let $u$ be a finite energy solution to problem \cref{problem} with 
$0 \leq u_0 \leq w$, where $w$ is the solution to \cref{ellip}. Then $u(\cdot, t) \rightarrow w$ as $t \rightarrow \infty$, in $L^\sigma(\Omega)$, where $\sigma<2^*$.
\end{theorem} 
\begin{proof}We divide the proof into two cases according to the value of the initial condition $u_0$.\smallskip\\
\textbf{The first case $u_0=0$:} 
In this case using the comparison principle as in \cref{comparison}, we get that $u\leq w$ in $\Omega\times (0,T)$ for all $T>0$. Hence, $u$ is globally defined. We now show that $u(x,\cdot)$ is increasing in $t$. Since $u$ has finite energy, 
 we fix $t_1>0$ 
and define $v(x,t)=u(x,t_1+t),$ then $v$ satisfies the problem
\begin{eqnarray*}
    \begin{cases}
        v_t-\Delta v+(-\Delta)^s v  =\frac{f(x)}{v^\gamma(x, t)} & \text { in } \Omega_T
        , \\
v(x, t)  =0  & \text { in }\left(\mathbb{R}^n \backslash \Omega\right) \times(0, T), \\
v(x, 0)  =u\left(x, t_1\right) & \text { in } \Omega.
    \end{cases}
\end{eqnarray*}
Now as $u\left(x, t_1\right)\geq u_0\equiv 0$ in $\Omega$, using the comparison principle as in \cref{comparison} again, we have $u \leq v$ in $\Omega_T$. Therefore for all $t_1>0$, we have $u(x, t) \leq u\left(x, t_1+t\right)$ for all $x \in \mathbb{R}^n$. Therefore $u(\cdot,t)$ is increasing in $t$. Since $u \leq w$, so there exists a measurable function $\hat{u}(\cdot)=\underset{t \rightarrow \infty}{\operatorname{lim }\operatorname{sup}}~ u(\cdot, t)$. By dominated convergence theorem, 
$u(\cdot,t)\rightarrow \hat{u}(\cdot)$ as $t\rightarrow\infty$ in $L^\sigma(\Omega)$ for $\sigma<2^*$. 
Also, we have $\hat{u} \leq w$. It is now sufficient to show that $\hat{u}$ is a very weak solution to problem \cref{ellip}.\\
Let $\varphi \in \mathcal{C}_0^{\infty}(\Omega)$, then using $\varphi
$ as a test function in \cref{problem} and integrating in $\Omega \times
(k, k+1)$, 
we get
\begin{equation*}
    \int_{\Omega}(u(x, k+1)-u(x, k)) \varphi(x) d x+\int_k^{k+1} \int_{\Omega} u(x, t)(-\Delta \varphi+(-\Delta)^s \varphi )d x d t=\int_k^{k+1} \int_{\Omega} \frac{f(x)}{u^\gamma(x, t)} \varphi(x) d x d t .
\end{equation*}
Now clearly 
we have
\begin{equation*}
    \int_{\Omega}(u(x, k+1)-u(x, k))|\varphi(x)| d x \rightarrow 0 \text { as } k \rightarrow \infty .
\end{equation*}
On the other hand, 
in the set $\Omega \times(k, k+1)$ using the monotonicity of $u$ in the variable $t$, we have
\begin{equation*}
    \frac{f(x)}{u^\gamma(x, k+1)}|\varphi(x)| \leq \frac{f(x)}{u^\gamma(x, t)}|\varphi(x)| \leq \frac{f(x)}{u^\gamma(x, k)}|\varphi(x)|, \quad\forall
    t\in(k,k+1),
\end{equation*}
\begin{equation*}
   u(x, k)\left|-\Delta \varphi\right| \leq u(x, t)\left|-\Delta \varphi\right| \leq u(x, k+1)\left|-\Delta \varphi\right|, \quad \forall
   t\in(k,k+1),
\end{equation*}
and
\begin{equation*}
    u(x, k)\left|(-\Delta)^s \varphi\right| \leq u(x, t)\left|(-\Delta)^s \varphi\right| \leq u(x, k+1)\left|(-\Delta)^s \varphi\right| ,\quad\forall 
    t\in(k,k+1).
\end{equation*}
Then, since $\phi, \Delta \phi, (-\Delta)^s\phi$ are bounded, so by the dominated convergence theorem, we will get that, as $k\rightarrow \infty$,
\begin{equation*}
    \int_k^{k+1} \int_{\Omega} u(x, t)(-\Delta \varphi+(-\Delta)^s \varphi )d x d t \rightarrow \int_{\Omega} \hat{u}(x)(-\Delta \varphi+(-\Delta)^s \varphi) d x,
\end{equation*}
and
\begin{equation*}
    \int_k^{k+1} \int_{\Omega} \frac{f(x)}{u^\gamma(x, t)} \varphi(x) d x d t \rightarrow \int_{\Omega} \frac{f(x)}{\hat{u}^\gamma(x)} \varphi(x) d x .
\end{equation*}
Hence, $\hat{u}$ is a very weak solution to problem \cref{ellip}. By the uniqueness we then get $\hat{u}=w$.\smallskip\\
\textbf{The second case $0<u_0 \leq w$:} Let us denote by $\hat{u}$ a finite energy solution to problem \cref{problem} in this case. Then by the comparison principle in \cref{comparison}, we deduce that $u \leq \hat{u} \leq w$, where $u$ is the weak solution to problem \cref{problem} with $u_0=0$. Hence by the convergence result of the first case, we get that $\hat{u}(\cdot, t) \rightarrow w$, as $t \rightarrow \infty$, strongly in $L^\sigma(\Omega)$ for all $\sigma<2^*$.
\end{proof}

\bibliography{main.bib}

\bibliographystyle{plain}


\end{document}